%% file: theorie_ergodique.tex
\documentclass[reqno,11pt]{amsart}
\usepackage{amssymb}
\usepackage{verbatim}
\usepackage [frenchb]{babel}
\usepackage{times}
\usepackage{setspace}
\newif\ifpdf
\ifpdf
\else
  \usepackage{graphicx}
\fi
\numberwithin{equation}{section}       
\theoremstyle{plain}
\newtheorem{Thm}{Th{\'e}or{\`e}me}[section]
\newtheorem{Prop}[Thm]{Proposition}
\newtheorem{Lem}[Thm]{Lemme}

\newtheorem{Prop-def}[Thm]{Proposition-D{\'e}finition}
\newtheorem{theoalph}{Th{\'e}or{\`e}me}

\newtheorem{coralph}[theoalph]{Corollaire}

\theoremstyle{definition}

\newtheorem{Def}[Thm]{D{\'e}finition}
\newtheorem{Rem}[Thm]{Remarque}

\newtheorem{qst}{Question}

\newcommand{\C}{{\mathbb{C}}}

\newcommand{\N}{{\mathbb{N}}}
\newcommand{\Q}{{\mathbb{Q}}}
\newcommand{\R}{{\mathbb{R}}}
\newcommand{\Z}{{\mathbb{Z}}}

\newcommand{\hprat}{{\mathbb{H}^{\mathbb{Q}}_p}}

\newcommand{\OK}{{\mathcal{O}_K}}
\newcommand{\MK}{{\mathfrak{m}_K}}

\newcommand{\HKo}{{\mathbb{H}_K^\mathrm{o}}}
\newcommand{\HK}{{\mathbb{H}_K}}
\newcommand{\HKrat}{{\mathbb{H}^{|K^*|}_K}}

\newcommand{\PK}{{\mathbb{P}^{1}_K}}
\newcommand{\PtK}{{\mathbb{P}^{1}_{\tK}}}
\newcommand{\PKber}{{\mathsf{P}^{1}_K}}
\newcommand{\Bber}{{\mathsf{B}}}

\newcommand{\AK}{{\mathsf{A}^{1}_K}}

\newcommand{\dpk}{\mathrm{d}_\PKber}
\newcommand{\dhyp}{\mathrm{d}_{\HK}}
\newcommand{\dhp}{\mathrm{d}_{\mathbb{H}_p}}

\newcommand{\de}{\delta}

\newcommand{\mdeg}{\underline{\mathrm{deg}}\,}

\newcommand{\fm}{{\mathfrak{m}}}

\newcommand{\fU}{{\mathfrak{U}}}

\newcommand{\fV}{{\mathfrak{V}}}
\newcommand{\fA}{{\mathfrak{A}}}

\newcommand{\cC}{{\mathcal{C}}}

\newcommand{\cE}{{\mathcal{E}}}

\newcommand{\cF}{{\mathcal{F}}}

\newcommand{\cO}{{\mathcal{O}}}
\newcommand{\cS}{{\mathcal{S}}}
\newcommand{\cScan}{{\mathcal{S}_\mathrm{can}}}
\newcommand{\cT}{{\mathcal{T}}}
\newcommand{\cU}{{\mathcal{U}}}

\newcommand{\cM}{{\mathcal{M}}}

\newcommand{\cP}{{\mathcal{P}}}

\newcommand{\hg}{{\hat{g}}}

\newcommand{\sB}{\mathsf{B}}

\newcommand{\tK}{\widetilde{K}}
\newcommand{\tR}{\widetilde{R}}

\renewcommand{\=}{ : = }
\renewcommand{\a}{\alpha}

\newcommand{\e}{\varepsilon}

\newcommand{\la}{\lambda}

\newcommand{\supp}{\mathrm{supp}\,}

\newcommand{\htop}{h_\mathrm{top}\,}

\newcommand{\diam}{\mathrm{diam}}

\DeclareMathOperator{\Crit}{Crit}
\DeclareMathOperator{\Jac}{Jac}
\DeclareMathOperator{\degtop}{deg_{\mathrm{top}}}

\begin{document}
%
%

\setcounter{tocdepth}{1}

\title[Dynamique mesurable ultram{\'e}trique]{Th{\'e}orie ergodique des 
fractions rationnelles sur un corps ultram{\'e}trique}

\date{\today}
\author[Charles Favre \and Juan Rivera-Letelier]{Charles Favre$^\dag$ \and Juan Rivera-Letelier$^\ddag$}

\address{Charles Favre \\ CNRS et Institut de Math{\'e}matiques de Jussieu\\
        {\'E}quipe G{\'e}om{\'e}trie et Dynamique\\
         Case 7012, 2  place Jussieu\\
         F-75251 Paris Cedex 05\\
         France  \\ and
         Unidade Mista CNRS-IMPA\\
         Dona Castorina 110\\
          Rio de Janeiro\\
          22460-320\\
         Brazil.}
\email{favre@math.jussieu.fr}

\address{Juan Rivera-Letelier \\ Facultad de Matem{\'a}ticas \\ Campus San Joaqu{\'\i}n \\ P.~Universidad Cat{\'o}lica de Chile \\ Avenida Vicu{\~n}a Mackenna~4860 \\ Santiago \\ Chile} 
\email{riveraletelier@mat.puc.cl}

\subjclass[2000]{Primary: 37F10, Secondary: 11S85, 37E25}
\thanks{Les deux auteurs ont \'et\'e financ\'e partiellement par le projet ECOS-Sud No. C07E01, ainsi que par le projet ANR-Berko.}
\thanks{\dag~Remercie chaleureusement le project FONDECYT N 7050221 de la CONICYT, Chile, qui a permis son s{\'e}jour {\`a} l'Universidad Cat{\'o}lica del Norte.}
\thanks{\ddag~Partiellement soutenu par les projets FONDECYT N~1040683 et Research Network on Low Dimensional Dynamics, PBCT ACT-17, de la CONICYT, Chile, et par le project MeceSup UCN~0202.}

%
%

\begin{abstract}
  On donne les premiers {\'e}l{\'e}ments pour l'{\'e}tude des propri{\'e}t{\'e}s
  ergodi\-ques d'une fraction rationnelle {\`a} coefficients dans un corps
  alg{\'e}briquement clos et complet pour une norme non archim{\'e}\-dienne.
  En particulier, pour une telle fraction rationnelle~$R$ on montre
  l'existence d'une mesure naturelle~$\rho_R$ repr{\'e}sentant la
  distribution asymptotique des pr{\'e}images it{\'e}r{\'e}es de chaque point non
  exceptionnel de~$R$.  On montre que cette mesure est
  (exponentiellement) m{\'e}langeante, et qu'elle satisfait au th{\'e}or{\`e}me
  limite central.  De plus, on donne une estimation de l'entropie
  m{\'e}trique de cette mesure, et de l'entropie topologique de~$R$, qui
  permettent de caract{\'e}riser les fractions rationnelles d'entropie
  topologique nulle.

\medskip

\noindent {\sc Abstract.} We make the first steps towards an understanding of the ergodic properties of a rational map defined over 
a complete algebraically closed non-archimedean field.
For such a rational map $R$, we construct a natural invariant probability measure $\rho_R$
which reprensents the asymptotic distribution of preimages of non-exceptional point.
We show that this measure  is exponentially mixing, and satisfies the central limit theorem.
We prove some general bounds on the metric entropy of $\rho_R$, and on the topological entropy of $R$.
We finally prove that rational maps with vanishing topological entropy have potential good reduction.

\end{abstract}

\maketitle
\tableofcontents

\section{Introduction}
Cet article est d{\'e}di{\'e} {\`a} l'{\'e}tude des propri{\'e}t{\'e}s ergodiques d'une
fraction rationnelle {\`a} coefficients dans un corps~$K$ alg{\'e}briquement
clos et complet par rapport {\`a} une norme ultram{\'e}trique~$| \cdot |$.
Plus pr{\'e}cisement, on {\'e}tudie l'action d'une telle fraction rationnelle
sur l'espace analytique de Berkovich~$\PKber$, associ{\'e} {\`a} la droite
projective~$\PK$.  Cet article pr{\'e}cise et compl{\`e}te la note~\cite{FR}.

La dynamique des fractions rationnelles {\`a} coefficients dans un corps
ultram{\'e}tri\-que appara{\^\i}t naturellement dans plusieurs contextes.  Un
travail r{\'e}cent de Kiwi~\cite{kiwi} relie l'espace des modules des
polyn{\^o}mes cubiques complexes {\`a} la dynamique de certaines fractions
rationnelles sur le corps ultram{\'e}trique des s{\'e}ries de Puiseux.  Une
deuxi{\`e}me source d'exemples provient de questions arithm{\'e}tiques.  {\`A} une
fraction rationnelle~$R$ {\`a} coefficients alg{\'e}briques sur~$\Q$ est
associ{\'e}e une fonction r{\'e}elle sur une cl{\^o}ture
alg{\'e}brique~$\overline{\Q}$ de~$\Q$, appel{\'e}e \textit{hauteur
  dynamique}.  On montre que cette hauteur est compl{\`e}tement
d{\'e}ter\-min{\'e}e par la dynamique de~$R$ en chacune des com\-pl{\'e}\-tions
de~$\overline{\Q}$, y compris les com\-pl{\'e}tions $p$-adiques de ce
corps, pour chaque nombre premier~$p$.  On est donc naturellement
amen{\'e} {\`a} regarder l'action de~$R$ sur ces corps non archim{\'e}diens.  Nous
renvoyons {\`a}~\cite{Be3,BH,BR,ACL,ChaThu,FRL,ST} pour une analyse de ces
hauteurs s'appuyant sur la dynamique des fractions rationnelles {\`a}
coefficients dans un corps ultram{\'e}trique.  On pourra consulter le
livre r{\'e}cent de Silverman~\cite{Sil07} pour d'autres exemples et
r{\'e}f{\'e}rences.

Dans cet article, on s'int{\'e}resse principalement aux propri{\'e}t{\'e}s
ergodiques d'une fraction rationnelle agissant sur l'espace analytique
de Berkovich~$\PKber$ associ\'e {\`a}~$\PK$, voir \cite{Ber}, ou~\cite{BR2,
  Esc03, R3.5} pour une approche {\'e}l{\'e}mentaire.  Dans le cas complexe
l'entropie topologique est {\'e}gale au logarithme du degr{\'e} de la fraction
rationnelle, et il existe une unique mesure d'entropie
maximale~\cite{Gro03, Lyu, Man83}.  De plus cette mesure poss{\`e}de des
propri{\'e}t{\'e}s remarquables d'{\'e}quidistribution: elle d{\'e}crit la
distribution asymptotique des pr{\'e}images it{\'e}r{\'e}es d'un point non
exceptionnel, ainsi que la distribution asymptotique des points
p{\'e}riodiques~\cite{Brolin, Tortrat, FLM, Lyu}.  Le d{\'e}veloppement r{\'e}cent
d'une th{\'e}orie du potentiel dans un cadre
non archim{\'e}dien~\cite{BR2,valtree,thuillier}, analogue {\`a} la th{\'e}orie du
potentiel complexe, a permis la construction d'une mesure analogue
lorsque le corps de base est ultram{\'e}trique.  Notre but est ici de
d{\'e}crire les propri{\'e}t{\'e}s g{\'e}n{\'e}rales de cette mesure, et d'en tirer
quelques conclusions sur l'entropie topologique des fractions
rationnelles en question.

\subsection{Dynamique sur la droite projective de Berkovich}
\label{ss:berkovich}

Dans toute la suite, $(K,|\cdot|)$ d{\'e}signe un corps m{\'e}tris{\'e}
non archim{\'e}dien, que l'on suppose \emph{complet et alg{\'e}briquement
  clos}.  Pour la plupart des r{\'e}sultats, nous ne ferons aucune
hypoth{\`e}se sur la caract{\'e}ristique de~$K$, ni sur la caract{\'e}ristique
r{\'e}siduelle de~$K$.

L'espace projectif standard~$\PK$ muni de la m{\'e}trique sph{\'e}rique est
tout {\`a} la fois totalement discontinu et non localement compact.  Ces
deux faits cumul{\'e}s rendent d{\'e}licate toute th{\'e}orie de la mesure, et plus
g{\'e}n{\'e}ralement toute analyse sur~$\PK$.  C'est pour cette raison qu'il
est plus naturel de consid{\'e}rer l'action des fractions rationnelles sur la
\textit{droite projective de Berkovich}~$\PKber$.  Cet espace peut
{\^e}tre construit de plusieurs mani{\`e}res dif\-f{\'e}rentes.  Il s'identifie par
exemple {\`a} l'arbre r{\'e}el obtenu par compl{\'e}tion de l'espace des boules
de~$K$, par rapport {\`a} une m{\'e}trique ad{\'e}quate.
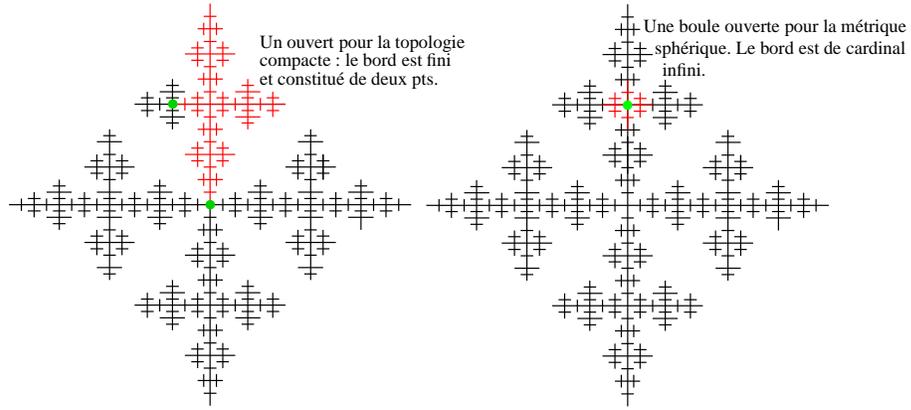
\begin{figure}[h]
\centering
\input{arbre.pstex_t}
\caption{Deux topologies sur un arbre}
\label{fig-arbre}
\end{figure}
On a  un plongement naturel de~$\PK$ dans~$\PKber$, associant
{\`a} $z\in \PK$ la boule de centre~$z$ et de rayon nul.  On notera
$\HK$ le compl{\'e}mentaire de~$\PK$ dans~$\PKber$.  L'espace~$\PKber$
poss{\`e}de deux topologies naturelles: une m{\'e}trisable provenant
de la m{\'e}trique sph{\'e}rique, qui n'est pas localement compacte, et
une autre plus grossi{\`e}re et compacte, qui n'est pas m{\'e}trisable
en g{\'e}n{\'e}ral, voir
Figure~\ref{fig-arbre}.  Cette derni{\`e}re topologie est adapt{\'e}e pour
la th{\'e}orie de la mesure.

L'action d'une fraction rationnelle~$R$ sur~$\PK$ s'{\'e}tend de
fa{\c c}on naturelle en une action continue sur~$\PKber$.  En
particulier, {\`a} chaque point~$\cS$ de~$\PKber$ on peut associer de
fa{\c c}on naturelle un \textit{degr{\'e} local}~$\deg_R(\cS)$, qui
co{\"\i}ncide avec le degr{\'e} local usuel lorsque~$\cS \in \PK$.  De
fa{\c c}on analogue au cas complexe, lorsque~$R$ est de degr{\'e} au
moins deux on d{\'e}compose~$\PKber$ en deux ensembles disjoints et
totalement invariants: \textit{l'ensemble de Fatou}~$F_R$, et
\textit{l'ensemble de Julia}~$J_R$.  Le premier est l'ouvert des
points o{\`u} la dynamique est \og r{\'e}guli{\`e}re\fg; le
second est un compact sur lequel la dynamique est
\og chaotique \fg, voir~\cite{R1, R4}.
\subsection{La mesure d'{\'e}quilibre et ses propri{\'e}t{\'e}s d'{\'e}quidistribution}
Pour chaque~$\cS \in \PKber$ on note~$[\cS]$ la masse de Dirac
en~$\cS$.  Chaque fraction rationnelle~$R$ induit une ac\-tion
continue $R^*$ sur les mesures bor{\'e}liennes et positives, telle que pour
chaque~$\cS \in \PKber$, on~ait
$$
R^{*} [\cS] = \sum_{\cS' \in R^{-1}(\cS)} \deg_{R}(\cS') [\cS']~.
$$
Voir~\S\ref{ss:fr} pour plus de pr{\'e}cisions.

Rappelons qu'un point $z \in \PK$ est \textit{exceptionnel} pour $R$
si l'ensemble de ses pr{\'e}images it{\'e}r{\'e}es est fini.  Une fraction
rationnelle admet au plus deux points exceptionnels, sauf dans le cas
o{\`u} la caract{\'e}ristique de $K$ est strictement positive et $R$ est conjugu{\'e}e
{\`a} un it{\'e}r{\'e} de l'automorphisme de Frobenius.  Dans ce dernier
cas l'ensemble exceptionnel est infini d{\'e}nombrable.
\begin{theoalph}\label{t:equi-measures}
Soit~$R$ une fraction rationnelle {\`a} coefficients dans~$K$ et de degr{\'e}
au moins~2.  Alors il existe une mesure de probabilit{\'e}~$\rho_R$
d{\'e}finie sur~$\PKber$, telle que pour toute mesure de
probabilit{\'e}~$\rho$ d{\'e}finie sur~$\PKber$, on ait convergence
\begin{equation}\label{e:convergence generale}
\lim_{n \to + \infty} \deg(R)^{-n} R^{n*}\rho =\rho_R ~,
\end{equation}
si et seulement si $\rho$ ne charge pas l'ensemble exceptionnel
de~$R$.  En particulier la mesure~$\rho_R$ est caract{\'e}ris{\'e}e comme
l'unique mesure de probabilit{\'e}~$\rho$ ne chargeant pas l'ensemble
exceptionnel et telle que ~$R^* \rho = \deg(R) \rho$.

De plus, la mesure~$\rho_R$ est m\'elangeante, ne charge aucun point de~$\PK$, et son
support topologique est {\'e}gal {\`a} l'ensemble de Julia de~$R$.
\end{theoalph}
On appellera $\rho_R$ la \emph{mesure d'{\'e}quilibre} de $R$.  Nous
donnerons une version quantitative du fait que~$\rho_R$ ne charge pas
les points de~$\PK$ en estimant la masse des boules en fonction de
leur diam{\`e}tre, voir la Proposition~\ref{p:holder}.  Par contre,
cette mesure peut charger un point de~$\HK
= \PKber \setminus \PK$, voir le Th{\'e}or{\`e}me~\ref{t:htopzero} ci-dessous.

On montrera que la mesure d'{\'e}quilibre est
exponentiellement m{\'e}\-lan\-geante et satisfait au th{\'e}or{\`e}me limite central
par rapport aux observables de classe~$\cC^1$, voir les
Propositions~\ref{p:melange} et~\ref{p:clt} du~\S\ref{ss:melange et
clt}.  Dans le cas complexe le m{\'e}lange exponentiel a {\'e}t{\'e}
montr{\'e} dans~\cite{ForSib94,ForSib95,H}, et le th{\'e}or{\`e}me limite central
dans~\cite{DPU}.
On utilise ici les m{\'e}thodes de la th{\'e}orie du potentiel d{\'e}velopp{\'e}es dans~\cite{CanLeB05, DinSib06, ForSib94, ForSib95}.

Lorsque la mesure~$\rho=[\cS]$ est la masse de Dirac situ{\'e}e en un point
non exceptionnel~$\cS$ dans~$\PKber$, pour tout entier positif~$n$
on~a,
$$
R^{n*} [\cS] = \sum_{\cS' \in R^{-n}(\cS)} \deg_{R^n}(\cS') [\cS']~.
$$ Le Th{\'e}or{\`e}me~\ref{t:equi-measures} montre alors que la mesure
$\rho_R$ d{\'e}crit la distribution asymptotique des pr{\'e}images
it{\'e}r{\'e}es de tout point non exceptionnel.  Dans le cas
arithm{\'e}tique cette propri{\'e}t{\'e} est une cons{\'e}quence de
l'{\'e}quidistribution des points de petite hauteur montr{\'e}e
dans~\cite{BR, ACL, FR3}, et on~a de plus des estimations quantitatives de la convergence~\cite{FR3}.  La propri{\'e}t{\'e} analogue dans le cas
complexe a {\'e}t{\'e} montr{\'e}e dans~\cite{Brolin,FLM,Lyu}.

Dans le cas complexe la mesure d'{\'e}quilibre d{\'e}crit la distribution
asymptotique de suites plus g{\'e}n{\'e}rales de points, voir~\cite{Lyu} et aussi~\cite{Tortrat}.
Lorsque le corps de base~$K$ est de caract{\'e}ristique nulle, nous
obtenons l'analogue non archim{\'e}dien.
\begin{theoalph}\label{t:equi-periodiques}
Supposons que la carat{\'e}ristique du corps~$K$ soit nulle, et soit~$R$
une fraction rationnelle {\`a} coefficients dans~$K$ de degr{\'e} au
moins~2.  Pour toute fraction rationnelle non constante $S\in K(z)$ et
pour tout entier $n\ge0$ tel que $R^n \neq S$, notons $[R^n=S]$ la
mesure support{\'e}e sur l'ensemble fini $\{ z \in \PK, \, R^n(z) =
S(z) \}$, et dont la masse en~$z$ est {\'e}gal {\`a} la multiplicit{\'e}
de~$z$ comme solution de $R^n = S$.  Alors, on~a
\begin{equation}\label{e:cvg}
\lim_{n \to + \infty} \deg(R)^{-n} [R^n=S] =\rho_R~.
\end{equation}
\end{theoalph}
Lorsque~$S(z) = z$ la mesure~$[R^n = S]$ est support{\'e}e sur les
points p{\'e}riodiques de p{\'e}riode~$n$ dans~$\PK$, chargeant chaque
point selon sa multiplicit{\'e} comme point p{\'e}riodique.  Le
Th{\'e}or{\`e}me~\ref{t:equi-periodiques} montre alors que la
mesure~$\rho_R$ d{\'e}crit la distribution asymptotique des points
p{\'e}riodiques dans~$\PK$.  Lorsque le corps de base est de
caract{\'e}ristique~$p > 0$, cette derni{\`e}re propri{\'e}t{\'e} n'est pas
v{\'e}rifi{\'e}e en g{\'e}n{\'e}ral.  Par exemple, si l'on pose~$P_0(z) = z +
z^p$, alors pour tout entier positif~$n$ on~a $P_0^{p^n}(z) = z +
z^{p^{p^n}}$, d'o{\`u} $[P_0^n = S] = p^{p^n} [0] + [\infty]$ et
$p^{-p^n} [P_0^{p^n} = S] \to [0]$ lorsque $n \to + \infty$.  Or  la
mesure~$\rho_{P_0}$ ne charge pas~$0$ par le Th{\'e}or{\`e}me~\ref{t:equi-measures}.
  Il serait int{\'e}ressant de
d{\'e}crire toutes les fractions rationnelles~$R$ et~$S$ pour lesquelles
$\deg(R)^{-n} [R^n = S]$ ne converge pas vers la mesure d'{\'e}quilibre
lorsque $n \to + \infty$.

Notons que pour $S(z)=z$ seuls les points p{\'e}riodiques
dans~$\PK$ interviennent dans~\eqref{e:cvg}.  En r{\'e}alit{\'e}, l'\'equation~\eqref{e:cvg} ne
peut {\^e}tre valide si l'on prend en compte tous les points
p{\'e}riodiques de~$R$ dans~$\PKber$, car en g{\'e}n{\'e}ral $R$ en
admet une infinit{\'e} non d{\'e}nombrable (c'est le cas si~$K$ est de caract\'eristique mixte et~$R$ admet
un point p{\'e}riodique indiff{\'e}rent, voir~\cite[Proposition~5.8]{R2}).  On
peut cependant formuler la question suivante.
\begin{qst}
  Les points p{\'e}riodiques r{\'e}pulsifs sont-ils \' equidistribu{\'e}s
  asymptotiquement par rapport {\`a} la mesure d'{\'e}quilibre?
\end{qst}
Enfin dans le cas o{\`u} $R$ est d{\'e}finie sur un \emph{corps de nombres},
Szpiro-Tucker~\cite{ST} ont d{\'e}montr{\'e} une version plus forte
de~\eqref{e:cvg} qui permet entre autres d'estimer l'exposant de Lyapunov de~$R$ par rapport {\`a} la mesure~$\rho_R$~:
$$
\deg(R)^{-n} \!\!\!\!\!\sum_{z \in \PK, \ R^n(z) = z}\!\!\!\!\!\! \log|R'|(z) \to \int \log|R'|\, d\rho_R~. 
$$
\subsection{Entropie}\label{ss:entropie}
Pour {\'e}tudier l'entropie topologique nous introduisons le nombre
$$
\degtop(R) \= \max \{ \# (R^{-1}(z)), \, z \in \PK \} ~,
$$ qu'on appellera \textit{le degr{\'e} topologique de~$R$}.  Lorsque
la caract{\'e}ristique de~$K$ est nulle, cet entier est {\'e}gal
{\`a}~$\deg(R)$.  Lorsque la caract{\'e}ristique de~$K$ est $p > 1$,
alors on peut {\'e}crire~$R$ de fa{\c c}on unique comme $R(z) = Q(z^{q})$,
o{\`u}~$Q$ est une fraction rationnelle s{\'e}parable et $q \ge 1$ est
une puissance de~$p$.  Dans ce cas on~a $\degtop(R) = \deg(Q)$.  Dans
tous les cas $\deg(R) / \degtop(R)$ est un entier, {\'e}gal au degr{\'e}
local de~$R$ en chaque point de~$\PK$, avec un nombre fini
d'exceptions.

Dans le cas complexe, l'entropie topologique d'une fraction rationnelle
est {\'e}gale au logarithme de son degr{\'e}, et il existe une unique
mesure d'entropie maximale~\cite{Gro03, Lyu, Man83}.  La situation
dans le cas non archim{\'e}dien est nettement plus compliqu{\'e}e.
Par exemple, pour chaque entier~$d \ge 5$ et chaque $a \in K$
satisfaisant $|a| \in (0, 1)$, la fraction rationnelle
$$
R_0(z) = \frac{z^{d - 2}}{1 + (az)^d}~,
$$
satisfait,
\begin{equation}\label{e:contreexample}
0 < h_{\rho_{R_0}}(R_0) < \htop(R_0) = \log 2 < \log d = \log \degtop(R_0) ~,
\end{equation}
o\`u $h_{\rho_{R_0}}(R_0)$ est l'entropie m\'etrique de la mesure $\rho_{R_0}$, et
$\htop(R_0)$ est l'entropie topologique de $R_0$.
Et ceci reste m{\^e}me valable dans un voisinage ouvert de~$R_0$ dans
l'espace des fractions rationnelles de degr{\'e}~$d$,
voir~\S\ref{ss:contreexample}.
\begin{theoalph}\label{t:estimation htop}
Soit $R$ une fraction rationnelle {\`a} coefficients dans~$K$ et de
degr{\'e} au moins deux.  Alors on~a
\begin{equation}\label{e:000}
\htop(R) = \htop(R|_{J_R}), \text{ et } 0 \le h_{\rho_R}(R) \le
\htop(R) \le \log \degtop(R)~.
\end{equation}
\end{theoalph}
Nous donnerons aussi une estimation de l'entropie m{\'e}trique de la
mesure d'{\'e}quilibre en termes de la fonction degr{\'e} local.
Pour cela on d{\'e}finit le \emph{degr\'e moyen} de~$R$ par
$$
\mdeg (R) \= \exp \left[ \int_{\PKber} \log \deg_R(\cS)\,
    d\rho_R(\cS)\right]~.
$$
\begin{theoalph}\label{t:estimation h metrique}
On~a
\begin{equation}\label{e:hmetr}
h_{\rho_R}(R) \ge \log \left( \frac{\deg(R)}{\mdeg(R)} \right) ~.
\end{equation}
Si de plus $\rho_R$ ne charge pas $\HK$, alors on~a
$$
h_{\rho_R}(R) = \htop(R)
=
\log \degtop(R) > 0 ~.
$$
\end{theoalph}
On obtient l'estimation~\eqref{e:hmetr} comme une cons{\'e}quence de
la formule de Rokhlin, et par cons{\'e}quent l'{\'e}galit{\'e} est
{\'e}quivalente {\`a} l'existence d'une partition g{\'e}n{\'e}ratrice
d'entropie finie associ{\'e}e {\`a}~$R$, voir~\cite{parry,PU}.
Dans la Section~\ref{ss:exemple sauvage} nous d{\'e}crirons un exemple o{\`u}
$R$ n'est pas localement injective sur un sous-ensemble non
d{\'e}nombrable de son ensemble de Julia, ce qui rend extr{\`e}mement
d{\'e}licat la construction d'une telle partition.  Malgr{\'e}
cela, nous posons la question suivante.
\begin{qst}
  Pour toute fraction rationnelle, a-t'on {\'e}galit{\'e}
  dans~\eqref{e:hmetr}?
\end{qst}
On verra sur des exemples que les difficult{\'e}s dans les calculs
d'entropie sont toutes li{\'e}es au fait que la fraction rationnelle
n'est pas s{\'e}parable ou poss{\`e}de une r{\'e}duction
non s{\'e}parable.  Lorsque ces ph{\'e}no\-m{\`e}nes sont absents (par
exemple lorsque la caract{\'e}ristique r{\'e}siduelle est nulle), il
est possible de faire une analyse plus fine, de pr{\'e}ciser le
Th{\'e}or{\`e}\-me~\ref{t:estimation htop}, et de r{\'e}pondre par
l'affirmative {\`a} la question pr{\'e}c{\'e}dente. Cette {\'e}tude
sera d{\'e}taill{\'e}e dans~\cite{FR3}.

L'estimation~\eqref{e:hmetr} permet aussi de caract{\'e}riser les
fractions rationnelles d'entropie topologique nulle.  Rappelons qu'une
fraction rationnelle $R \= P/Q$ avec $P,Q \in K[z]$ \textit{a bonne r{\'e}duction} si les r{\'e}ductions $\widetilde{P}, \widetilde{Q}$
de~$P$ et~$Q$ dans le corps r{\'e}siduel de~$K$ sont non nulles, et si la fraction $\widetilde{P}/\widetilde{Q}$ a m{\^e}me degr{\'e}
que~$R$.
\begin{theoalph}\label{t:htopzero}
  Les assertions suivantes sont {\'e}quivalentes:
\begin{enumerate}
\item
$\htop(R) =0$~;
\item
$h_{\rho_R}(R) =0$~;
\item
il existe une coordonn{\'e}e dans laquelle~$R$ a bonne r{\'e}duction~;
\item la mesure~$\rho_R$ est la masse de Dirac situ{\'e}e en un point de~$\HKrat$~;
\item la mesure~$\rho_R$ charge un point de~$\PKber$~.
\end{enumerate}
\end{theoalph}
L'ensemble $\HKrat$ est d{\'e}fini au~\S\ref{s:generalites} et est
constitu{\'e} des points de branchements de l'arbre r{\'e}el $\PKber$.

Mentionnons tout de suite le r\'esultat suivant que nous g\'en\'eraliserons dans~\cite{FR3}.
Une fraction rationnelle est dite \emph{mod\'er\'ee} si le sous-ensemble 
$\{ \deg_R \ge 2\}$ de $\PKber$ est inclus dans un arbre fini. Par exemple,
lorsque la caract\'eristique r\'esiduelle de $K$ est nulle, toute fraction
rationnelle est mod\'er\'ee.
Les Th\'eor\`emes~\ref{t:estimation h metrique} et~\ref{t:htopzero}, et le fait que pour un polyn{\^o}me mod{\'e}r{\'e} l'intersection de l'ensemble de Julia avec $\{ \deg_R \ge 2 \}$ est finie, impliquent le r\'esultat suivant.
\begin{coralph}\label{c:modere}
Soit $P$ un \emph{polyn\^ome} mod\'er\'e qui ne
soit pas conjugu\'e \`a un polyn\^ome ayant bonne r\'eduction.
Alors $h_{\rho_R}(R) = \htop(R) = \log \degtop(R)$.
\end{coralph}

\subsection{Plan}
Nous commen{\c c}ons en \S\ref{ss:droite} par rappeler la d{\'e}finition de l'espace
projectif au sens de Berkovich et nous y d{\'e}crivons rapidement sa
g{\'e}om{\'e}trie.  Au~\S\ref{ss:fr}, nous
donnons les propri{\'e}t{\'e}s principales de l'action d'une fraction rationnelle sur
cet espace, et au~\S\ref{ss:fatou-julia} nous rappelons les
d{\'e}finitions des ensembles de Fatou et Julia, ainsi que certaines
propri{\'e}t{\'e}s qui nous seront utiles par la suite.  Nous indiquons
alors comment construire une th{\'e}orie du potentiel sur $\PKber$
adapt{\'e}e {\`a} nos besoins (\S\ref{ss:potentiel}).

La Section~\ref{s:equilibre} contient la construction de la mesure
d'{\'e}quilibre (\S\ref{ss:construction}), ainsi que la preuve du
Th{\'e}or{\`e}me~\ref{t:equi-measures}~\S\ref{ss:equi-measures} et du
Th{\'e}or{\`e}me~\ref{t:equi-periodiques}~\S\ref{ss:equi-periodiques}.
Nous montrons aussi les propri{\'e}t{\'e}s de m{\'e}lange de la mesure
d'{\'e}quilibre (\S\ref{ss:melange et clt}).

Nous attaquons les probl{\`e}mes d'entropie au~\S\ref{s:entropie}.
On commence par quelques g{\'e}n{\'e}rali\-t{\'e}s sur l'entropie
topologique dans les espaces compacts non m{\'e}trisables
\S\ref{ss:generalites entropie}, puis nous donnons la preuve du
Th{\'e}or{\`e}me~\ref{t:estimation htop}~\S\ref{ss:estimation htop},
et les preuves des Th{\'e}or{\`e}mes~\ref{t:estimation h metrique}
et~\ref{t:htopzero} et du Corollaire~\ref{c:modere}~\S\ref{ss:estimation h metrique}.

Nous concluons cet article en explicitant quatre exemples qui nous
semblent caract{\'e}ristiques en \S\ref{sec:exemple}.
\subsection{Remerciements:}
nous tenons \`a remercier les deux rapporteurs pour leur lecture extr\`emement d\'etaill\'ee du papier
et leurs suggestions pour en am\'eliorer la r\'edaction.

\section{G{\'e}n{\'e}ralit{\'e}s}\label{s:generalites}

Cette partie contient un certain nombre de r{\'e}sultats et de faits
sur la g{\'e}om{\'e}trie de la droite projective sur un corps
norm{\'e} non archim{\'e}dien, ainsi que sur les propri{\'e}t{\'e}s de
base des fractions rationnelles.

Dans le reste de cet article on fixe un corps alg{\'e}briquement
clos~$K$, muni d'une norme non archi\-m{\'e}dienne~$| \cdot |$ pour
laquelle il est complet.  On note $\OK \= \{ z \in K, |z| \le 1 \}$
l'anneau des entiers de~$K$, et $\MK \= \{ z \in K, |z|<1\} $ son
unique id\'eal maximal.  Le corps r{\'e}siduel sera not{\'e} $\tK \=
\cO_K / \MK$, c'est un corps alg{\'e}briquement clos.

Rappelons qu'une semi-norme multiplicative sur un anneau commutatif $R$ muni d'une unit\'e
est une fonction
$|\cdot|: R \to \R_+$ telle que $|1| =1$, $|ab| = |a| \cdot |b|$, et $|a+b| \le\max \{ |a|, |b| \}$
pour tout $a,b\in R$. Si $\{ a, |a| =0\}$ est r\'eduit \`a $\{ 0 \}$ alors $|\cdot|$ est une norme
non-archim\'edienne sur $R$.


\subsection{La droite projective au sens de Berkovich}\label{ss:droite}
Nous renvoyons \`a~\cite{BR2,Ber} pour plus d'informations.

Soit $\AK$ l'espace de toutes les semi-normes multiplicatives
d{\'e}finies sur $K[z]$, dont la restriction {\`a}~$K$ est {\'e}gale
{\`a} $|\cdot|$.  On note de plus~$\cS_\infty$ la fonction d{\'e}finie
sur $K[z]$, qui est constante {\'e}gale {\`a}~$\infty$ sur tous les
polyn{\^o}mes non constants de~$K$, et telle que pour chaque
polyn{\^o}me constant $P \equiv a$ on ait $\cS_\infty(P) = |a|$.  On
pose $\PKber = \AK \sqcup \{ \cS_\infty \}$ et on munit~$\PKber$ de la
topologie la moins fine telle que pour chaque $P
\in K[z]$ la fonction $\cS \mapsto \cS(P)$ soit continue.  L'espace
$\PKber$ est alors compact et s{\'e}quentiellement compact.
On l'appelle \emph{espace analytique de Berkovich}
associ{\'e} {\`a}~$\PK$.

Chaque point $z \in K$ induit une semi-norme, qu'on notera aussi par~$z$, d{\'e}finie par $z(P) = |P(z)|$. On obtient ainsi un
hom{\'e}omorphisme de $\PK = K \cup \{ \infty \}$ sur son image.
Dans la suite, on identifiera $\PK$ avec son image dans $\PKber$.

{\`A} chaque boule $B = \{ |z - z_0| \le r \}$ de~$K$ correspond la
semi-norme~$\cS_B$ dans~$\AK$, d{\'e}finie par $\cS_B(P) = \sup_B
|P(z)|$.  Plus g{\'e}n{\'e}ralement, toute suite d{\'e}croissante $\{
B_i \}_{i \ge 0}$ de boules de $K$ induit une semi-norme $P \mapsto
\lim_{i \to \infty} \cS_{B_i}(P)$. R{\'e}cipro\-que\-ment, toute
semi-norme dans $\AK$ est de cette forme et les points de $\PKber$ se
rangent donc dans l'une des quatre cat{\'e}gories suivantes (voir par
exemple~\cite[p.18]{Ber}):
\begin{enumerate}
\item[{\rm i)}]
{\it les points de $\PK$};
\item[{\rm ii)}] {\it les points rationnels}, de la forme $\cS_B$,
avec $B = \{ |z - a| \le r \}$ et $r \in |K^*|$;
\item[{\rm iii)}] {\it les points irrationnels}, de la forme $\cS_B$,
avec $B = \{ |z - a| \le r \}$ et $r \not \in |K^*|$;
\item[{\rm iv)}] {\it les points singuliers}, associ{\'e}s {\`a} une suite
d{\'e}croissante de boules de $K$ dont l'intersection est vide.
\end{enumerate}
Notons que tous les points de type (ii),~(iii) et~(iv) sont des normes
qui s'{\'e}tendent {\`a}~$K(z)$, alors que la semi-norme $|\cdot|_{z_0}$ associ{\'e}e
{\`a} un point $z_0\in\PK$ v{\'e}rifie $|z-z_0|_{z_0} =0$.  On notera
par~$\HK$ l'ensemble~$\PKber \setminus \PK$, par~$\HKrat$ le
sous-ensemble de~$\HK$ des points rationnels de~$\PKber$, et par $\HKo$ le compl\'ementaire 
dans $\HK$ de l'ensemble de ses points singuliers.
 On appelle
\emph{point canonique}\footnote{aussi appel\'e point de Gauss, voir~\cite{BR2}} 
la norme associ{\'e}e {\`a} la boule unit{\'e}
$\{ z \in K, |z| \le 1 \}$ et on le note~$\cScan$.  Etant donn{\'e} un
point rationnel ou irrationnel $\cS$, on d{\'e}signe par $B_\cS$ la
boule de $K$ correspondante. Elle co\"{\i}ncide avec l'intersection avec $K$
du compl\'ementaire dans~$\AK$ de la composante connexe non born{\'e}e de $\AK \setminus \cS$.
Lorsque $z \in K$ on pose $B_z = \{ z
\}$, et  si $z = \infty$, on note $B_\infty = K$.

Chaque fraction rationnelle $R \in K(z)$ agit sur $\PKber\setminus
\PK$. Un point $\cS$ dans cet espace est en effet une norme sur
$K[z]$, et donc induit une norme sur le corps des fractions $K(z)$.
On d{\'e}finit alors $R(\cS)$ comme la norme v{\'e}rifiant $R (\cS)(P)
\= \cS ( P \circ R)$.  Cette action s'{\'e}tend contin{\^u}ment en une
action de $R$ sur $\PKber$ qui co{\"\i}ncide avec l'action naturelle
de $R$ sur $\PK$.

\subsubsection*{Structure d'arbre}
C'est un fait fondamental que $\PKber$ poss{\`e}de une structure
d'arbre r\'eel, que nous allons maintenant d{\'e}crire
bri{\`e}vement. Consid{\'e}rons l'ordre partiel $\le$ d{\'e}fini sur
l'espace $\PKber$ par $\cS \le \cS'$ si et seulement si pour tout $P
\in K[z]$ on~a $\cS(P) \le \cS'(P)$.  Lorsque $\cS$ et $\cS'$ sont non
singuliers, on~a $\cS \le \cS'$ si et seulement si $B_\cS \subset
B_{\cS'}$. On v{\'e}rifie que le point $\cS_\infty$ est l'unique
{\'e}l{\'e}ment maximal de $\PKber$ et que l'ensemble des
{\'e}l{\'e}ments minimaux co{\"\i}ncide avec l'union de $K$ et des
points singuliers.

{\'E}tant donn{\'e}s $\cS$ et $\cS'$ dans $\PKber$,  on d{\'e}finit $\cS
\vee \cS' \in \PKber$ par
$$
(\cS \vee \cS')(P) = \inf \{ \widehat{\cS}(P) , \,  \widehat{\cS}
\in \PKber, \ \cS \le \widehat{\cS}, \ \cS' \le \widehat{\cS} \}.
$$
On v{\'e}rifie qu'on a $\cS \vee \cS' = \cS$ si et seulement si $\cS'
\le \cS$ et que $\cS \vee \cS' = \cS_\infty$ si et seulement si
$\cS$ ou $\cS'$ est {\'e}gale {\`a} $\cS_\infty$.  Lorsque $\cS$ et $\cS'$
sont des points non singuliers dans $\AK$, le point $\cS \vee
\cS'$ est la semi-norme associ{\'e}e {\`a} la plus petite boule de $K$
qui contient $B_\cS$ et $B_{\cS'}$.

L'ordre partiel $\le$ d{\'e}finit alors une structure d'arbre dans
$\AK$ (resp. $\PKber$) au sens suivant. Pour chaque paire de points
distincts $\cS$ et $\cS'$, l'ensemble
$$
[\cS, \cS']
=
\{ \widetilde{\cS} , \,  \cS \le \widetilde{\cS} 
\le 
\cS \vee \cS' \ \mbox{ ou } \ \cS' \le \widetilde{\cS} 
\le \cS \vee \cS' \}.
$$ est l'unique arc topologique dans $\AK$ (resp. $\PKber$) ayant
$\cS$ et $\cS'$ comme extr{\'e}mi\-t{\'e}s. Un ensemble de la forme $ [\cS,
\cS']$ est appel{\'e} \emph{segment}.  On dira qu'un point $\cS$ est
{\it entre} les points $\cS'$ et $\cS''$ lorsque $\cS \in [\cS',
\cS'']$.  Dans ce cas on~a $[\cS', \cS''] = [\cS', \cS] \cup [\cS,
\cS'']$.  Notons que pour chaque triplet de points $\cS$, $\cS'$ et
$\cS''$ il existe un unique point qui est entre $\cS$ et $\cS'$, entre
$\cS'$ et $\cS''$ et entre $\cS''$ et $\cS$.

\subsubsection*{M{\'e}trique sph{\'e}rique}
Diverses fonctions d{\'e}finies sur $K$ s'{\'e}tendent de ma\-ni{\`e}\-re
naturelle {\`a} $\AK$. Ceci permet de d{\'e}finir une m{\'e}trique sur $\PKber$
{\'e}tendant la m{\'e}trique sph{\'e}rique de~$\PK$.

Commen{\c c}ons par d{\'e}finir les fonctions $| \cdot |$ et $\diam :
\AK \to [0, + \infty)$.
Pour $z_0 \in K$ on pose $P_{z_0} (z) = z - z_0 \in K[z]$.
Alors,
$$
|\cS| = \cS(P_0)
\ \mbox{ et } \
\diam(\cS) = \inf_{z \in K} \cS(P_z)~.
$$
Lorsque $\cS$ est un point non singulier de $\AK$, on~a
$$
| \cS | = \sup_{B_\cS}|z|
\ \mbox{ et } \
\diam(\cS) = \diam(B_\cS) ~.
$$ En particulier, la restriction de $| \cdot |$ {\`a} $K$
co{\"\i}ncide avec la norme de~$K$.  La fonction~$| \cdot |$ s'annule
uniquement au point~$0$.  Pour tout $\cS \in \AK$, on~a $|\cS| \ge
\diam(\cS)$ et $\diam (\cS) =0$ si et seulement si $\cS \in K$.
Enfin, la fonction $|\cdot|$ est continue et s'{\'e}tend
contin{\^u}ment {\`a}~$\PKber$ en posant $|\infty| = + \infty$.

A l'aide des fonctions pr{\'e}c{\'e}dentes, on d{\'e}finit maintenant:
$$ \sup \{ \cS, \cS' \} = \diam(\cS \vee \cS')~, \text{ pour }
\cS, \cS' \in\PKber ~.
$$
Lorsque $\cS$ et $\cS'$ sont des points non singuliers de $\AK$, on~a
$$
\sup \{ \cS, \cS' \} = \sup \{ |z - z'| , \,  z \in B_\cS, z' \in B_{\cS'} \}~,
$$ et en particulier pour tous $z, z' \in K$ on~a $\sup \{ z, z' \}
= |z - z'|$.  On v{\'e}rifie ais{\'e}ment que
$$
\sup \{ \cdot, 0 \} = | \cdot |
\mbox{ et }
\sup \{ \cdot, \cScan \} = \max \{ 1, |\cdot| \}~.
$$

La m{\'e}trique sph{\'e}rique\footnote{aussi appel\'ee \og small model metrics \fg{} dans~\cite{BR2}}~$\dpk$ sur~$\PK$ est d{\'e}finie, pour
$z, w \in K$, par
$$
\dpk (z, w) = \frac{2 |z-w|}{\max\{1, |z|\} \times \max\{ 1, |w|\}} ~,
$$ et $\dpk(z, \infty) = 2 \max \{ 1, |z| \}^{-1}$.  On {\'e}tend
naturellement cette m{\'e}trique {\`a}~$\PKber$ en posant, pour $\cS,
\cS' \in \AK$~:
$$ \dpk (\cS, \cS') = \frac{ 2 \sup \{ \cS, \cS'\}} {\max \{ 1, |\cS|
\} \times \max \{ 1, | \cS'| \}}
- \frac{\diam(\cS)}{\max \{ 1, |\cS| \}^2}
- \frac{\diam(\cS')}{\max \{ 1, |\cS'| \}^2} ~,
$$ et $\dpk(\cS, \infty) = 2 \max \{ 1, |\cS| \}^{-1}$.  On
v{\'e}rifie que cette m{\'e}trique est compatible avec la structure
d'arbre de $\PKber$ au sens que pour tous $\cS, \cS', \cS'' \in
\PKber$ on~a $\dpk (\cS, \cS') = \dpk(\cS,\cS'') + \dpk(\cS'', \cS')$
si et seulement si $\cS'' \in [\cS,\cS']$.  L'espace m{\'e}trique
$(\PKber,\dpk)$ est complet, mais il n'est pas localement compact.
En particulier la topologie sur~$\PKber$ induite par cette distance ne co{\"{\i}}ncide pas avec la topologie introduite pr{\'e}c{\'e}demment.
Elle ne jouera pas de r\^ole dans la suite.

Introduisons maintenant quelques notations.  Une \textit{boule
ouverte} (resp. \textit{ferm{\'e}e}) de~$\AK$ est un ensemble de la
forme $ \{ \cS \in \PKber, \sup \{ \cS, a \} < r \}$ (resp. $\{ \cS
\in \PKber, \sup \{ \cS, a \} \le r \}$), o{\`u} $a \in K$ et $r > 0$.
On notera $\sB(z,r)$ (resp. $\bar{\sB}(z,r)$) ces ensembles.
Leurs intersections avec $K$ seront not\'ees
$$ B(z,r) = \sB(z,r) \cap K,
\bar{B}(z,r) = \bar{\sB}(z,r) \cap K. $$

Une \textit{boule ouverte} (resp. \textit{ferm{\'e}e}) de~$\PKber$ est
une boule ouverte (resp. ferm{\'e}e) de~$\AK$ ou le compl{\'e}mentaire
dans~$\PKber$ d'une boule ferm{\'e}e (resp. ouverte) de~$\AK$.  Il est
facile de voir que toute boule de~$\PKber$ est connexe.

Un \emph{affino{\"\i}de} (resp. \emph{ouvert fondamental}) est une r{\'e}union finie d'intersections finies (non vides) de boules ferm\'es (resp. ouvertes).
Un ouvert fondamental~$U$ poss{\`e}de un nombre fini de points dans son bord.
Notons enfin que les boules ouvertes (resp. les ouverts fondamentaux) 
de~$\PKber$ forment une sous-base (resp. base) de la
topologie de~$\PKber$.

\subsubsection*{L'espace hyperbolique~$\HK$}
Rappelons que~$\HK=\PKber\setminus\PK$.  C'est un ensemble
conne\-xe, donc un sous-arbre de $\PKber$. La fonction\footnote{aussi appel\'ee \og big model metrics \fg dans~\cite{BR2}}~$\dhyp$
d{\'e}finie par
$$
\dhyp(\cS, \cS')
=
2 \log\sup \{ \cS, \cS' \} - \log \diam(\cS) - \log \diam(\cS') ~,
$$ est une distance sur $\HK$.  Lorsque $\cS \le \cS'$ on~a $
\dhyp(\cS, \cS') = \log \left( \diam(\cS')/ \diam(\cS)\right)$.
L'espace m{\'e}trique $(\HK, \dhyp)$ est complet et la m{\'e}trique
est {\`a} nouveau compatible avec la structure d'arbre de
$\HK$. Notons de plus que $\dhyp$ est invariante par l'action du
groupe $\mathrm{PGL}(2,K)$ des automorphismes de~$\PK$. De ce fait, on d{\'e}duit que
$(\HK,\dhyp)$ est isom{\'e}trique {\`a} l'arbre r{\'e}el de
$\mathrm{PGL}(2,K)$ d{\'e}crit dans~\cite{T},
voir~\cite[\S7.2]{R2}.
Enfin pour pour chaque $\cS \in \HK$ on~a
\begin{equation}\label{eforget}
\dhyp(\cS, \cScan) = - \log \dpk (\cS, \PK)~.
\end{equation}
Fixons un point base $\cS_0 \in \HK$.  Le \textit{produit de Gromov} est la
fonction
$$ \langle \cdot\, , \cdot \rangle_{\cS_0} : \PKber \times \PKber \to [0, + \infty] $$
d{\'e}finie comme suit.
{\'E}tant donn{\'e}s $\cS, \cS' \in
\PKber$, notons $\cS''$ l'unique point de $\PKber$ qui est entre $\cS$
et $\cS'$, entre $\cS$ et $\cS_0$ et entre $\cS'$ et $\cS_0$.
On pose alors
$$
\langle \cS , \cS' \rangle_{\cS_0}
=
\begin{cases}
\dhyp(\cS'', \cS_0) & \text{si } \cS'' \in \HK ~;\\
+ \infty & \text{si } \cS'' \in \PK~.
\end{cases}
$$
On   v{\'e}rifie    facilement   que    $\langle   \cS    ,   \cS'
\rangle_{\cS_0}= + \infty$ si et  seulement si $\cS = \cS'  \in \PK$; et
que $\langle \cS , \cS' \rangle_{\cS_0} = 0$ si et seulement si $\cS_0
\in  [\cS,\cS']$.
En  particulier, pour  tout  $\cS \in  \PKber$ on  a
$\langle \cS , \cS_0 \rangle_{\cS_0} = 0$.
En g{\'e}n{\'e}ral, on~a $\langle \cS , \cS' \rangle_{\cS_0} \le \dhyp(\cS,
\cS_0)$, avec {\'e}galit{\'e} si et seulement si $\cS \in [\cS', \cS_0]$.

\subsection{Fractions rationnelles}\label{ss:fr}
Fixons une fraction rationnelle non constante~$R$ {\`a} coefficients
dans~$K$.  On v{\'e}rifie que l'image de tout ouvert fondamental est
encore un ouvert fondamental.  Une preuve est donn{\'e}e
dans~\cite[Proposition~2.6]{R1}.

\subsubsection*{Degr{\'e} topologique}
Le \textit{degr{\'e} topologique} $\degtop(R)$ de~$R$ est par
d{\'e}finition l'entier
$$
\degtop(R)
\=
\max_{\cS \in \PK} \# (R^{-1}(\cS))~.
$$ Lorsque la caract{\'e}ristique de~$K$ est nulle, cet entier est {\'e}gal
{\`a}~$\deg(R)$.  Lorsque la caract{\'e}ristique de~$K$ est $p > 1$,
on peut alors {\'e}crire~$R$ de fa{\c c}on unique comme $R(z) = Q(z^{q})$,
o{\`u}~$Q$ est une fraction rationnelle s{\'e}parable et $q \ge 1$ est
une puissance de~$p$.  Dans ce cas on~a $\degtop(R) = \deg(Q)$.
Notons que dans tous les cas $\degtop(R)$ divise~$\deg(R)$.

On v{\'e}rifie de plus que lorsque $\deg(R) > 1$, on~a~$\degtop(R) = 1$
si et seulement si la caract{\'e}ristique de~$K$ est $p>1$, et~$R$ est
conjugu{\'e}e {\`a} un it{\'e}r{\'e} de l'automorphisme de Frobenius, 
c'est-{\`a}-dire $R(z) = z^q$ avec $q$ est une puissance de~$p$ pour un choix
convenable de coordonn{\'e}es.

\subsubsection*{Degr{\'e} local}
On d{\'e}finit le degr{\'e} local de~$R$ en un point $z_0 \in \PK$
comme suit.  Quitte {\`a} faire un changement de coordonn{\'e}es, on
peut supposer que $\infty\notin\{ z_0, R(z_0)\}$. Dans ce cas, on
{\'e}crit $R(z_0+h) = R(z_0) + a h^k + \mathcal{O}(h^{k+1})$ avec
$a\neq 0$, et on pose $\deg_R(z_0) \= k$. C'est un nombre entier
strictement positif. On v{\'e}rifie facilement que $\deg_{R\circ R'} =
\deg_R \circ R' \times \deg_{R'}$ pour tout couple $R,R'\in K(z)$, et
donc que le degr{\'e} local ne d{\'e}pend pas des choix de
coordonn{\'e}es.  On v{\'e}rifie aussi que pour tout $z \in \PK$ on~a
$\sum_{R(w)=z} \deg_R(w) = \deg(R)$.  On v{\'e}rifie sans difficult{\'e}
que le degr{\'e} local de~$R$ en chaque point de~$\PK$ est {\'e}gal {\`a}
l'entier $\deg(R) / \degtop(R)$, avec au plus un nombre fini
d'exceptions.

La proposition suivante permet d'{\'e}tendre la d{\'e}finition du
degr{\'e} local {\`a}~$\PKber$.  Une d{\'e}finition g{\'e}om{\'e}trique est
donn{\'e}e en~\cite[\S2]{R1}.  On donne ici une approche plus
alg{\'e}brique qui a l'avantage de se g{\'e}n{\'e}raliser en toute
dimension.
\begin{Prop-def}\label{pf:degre local}
La fonction degr{\'e} local $\deg_R$ s'{\'e}tend de mani\`ere unique en une fonction
d{\'e}finie sur $\PKber$ et {\`a} valeurs dans les entiers strictement
positifs, v{\'e}rifiant la propri{\'e}t{\'e} suivante~:
\begin{itemize}
 \item[(*)]
pour tout ouvert fondamental~$V$, toute composante connexe~$U$
de $R^{-1}(V)$ et tout $\cS_0 \in V$, l'entier
$$
\sum_{R(\cS) = \cS_0, \cS \in U} \deg_R(\cS)~,
$$
est ind{\'e}pendant du choix du point $\cS_0 \in V$. 
\end{itemize}
\end{Prop-def}
On en d{\'e}duit facilement le r{\'e}sultat suivant~:
\begin{Prop}
  La fonction~$\deg_R$ est semi-continue sup{\'e}rieurement et prend
ses valeurs dans $[\deg(R)/\degtop(R), \ldots, \deg(R)]$.  De plus, pour
tout~$\cS \in \PKber$ on~a,
\begin{equation}\label{e:deg-local}
\# (R^{-1}(\cS)) \le \degtop(R) ~,
\text{ et }
\sum_{R(\cS')=\cS} \deg_R(\cS') = \deg(R) ~.
\end{equation}
Enfin, pour tout couple de fractions
rationnelles $R,R'$ non nulles, on~a $\deg_{R\circ R'} = \deg_R
\circ R' \times \deg_{R'}$.
\end{Prop}
Lorsque~$\cS$ est un point rationnel de~$\PKber$, on peut calculer
$\deg_R(\cS)$ de la mani{\`e}re suivante. Quitte {\`a} faire un
changement de coordonn{\'e}es {\`a} la source et au but, on peut
supposer que $\cS = R(\cS) = \cScan$.  Ce qui signifie que l'on peut
{\'e}crire $R = P/Q$ avec deux polyn\^omes $P, Q \in \OK[z]$ dont les
r{\'e}duction $\widetilde{P}, \widetilde{Q}$ sont non nulles, et telles
que la fraction rationnelle $\widetilde{R} \= \widetilde{P}/\widetilde{Q}$ 
soit non constante. Le degr{\'e} local en $\cScan$ est alors le degr{\'e} de
$\widetilde{R}$.

\begin{proof}[D{\'e}monstration de la 
Proposition-D{\'e}finition~\ref{pf:degre local}] Notons tout d'abord que
  $R^{-1} \{ \cS \}$ est un ensemble fini de cardinal $\le d$ pour
  tout $\cS \in \PKber$. Pour $\cS \in \PK$, c'est clair. Sinon $\cS
  \in \HK$ d{\'e}finit une norme sur $K(z)$, et donc sur le sous-corps
  $R^* K(z) \= \{ \phi \circ R, \,\phi \in K(z) \}$ de $K(z)$.  Le
  corps $K(z)$ est une extension finie de degr{\'e} au plus $d$ de
  $R^* K(z)$.  Donc $\cS|_{R^*K(z)}$ admet au plus $d$ extensions
  {\`a} $K(z)$, voir~\cite{ZS}. Ceci est exactement dire que $\cS$
  poss{\`e}de au plus $d$ pr{\'e}images.

On montre maintenant facilement l'unicit{\'e} de la fonction $\deg_R$.
Soit $\cS \in \PKber$ et fixons $U$ un  ouvert fondamental contenant
$\cS$, tel que $R: U \to R(U)$ soit propre et $R^{-1}\{ R(\cS) \} \cap
U= \{ \cS\}$. Un tel ouvert existe toujours car $R(\cS)$ admet un
nombre fini de pr{\'e}images. On applique maintenant la relation (*) 
pour tout point $z_0 \in
R(U)\cap \PK$, et l'on obtient $\deg_R(\cS) = \sum_{z\in U,\, R(z) =z_0} \deg_R(z)$, ce qui
d{\'e}termine $\deg_R$ de mani{\`e}re unique.

Pour tout $z\in \PKber$, notons
$\cO_z$ l'anneau des germes de fonctions analytiques en $z$, et
$\fm_z$ l'id{\'e}al des fonctions qui s'annulent en $z$. Lorsque $z\in
\PK$, $\cO_z$ est un anneau local d'id{\'e}al maximal $\fm_z$; sinon
$\cO_z$ est un corps et $\fm_z = (0)$.  Dans tous les cas, on note
$\kappa(z)$ le corps $\cO_z /\fm_z$.

L'application analytique $R$ est finie~\cite[3.1.10]{berko2}, donc
pour tout $z\in \PKber$ l'anneau $\cO_z$ est un module de type fini
sur $\cO_{R(z)}$~\cite[3.1.6]{berko2}. On d{\'e}finit alors $$\deg_R(z)
= \dim_{\kappa(R(z))} (\cO_z/ \fm_{R(z)} \cO_z)~.$$ C'est le nombre
minimal de g{\'e}n{\'e}rateurs de $\cO_z$ vu comme
$\cO_{R(z)}$-module~\cite[Theorem 2.3]{matsu}.

Pour v{\'e}rifier tout d'abord que la fonction ainsi d\'efinie co\"{\i}ncide bien avec
le degr\'e local sur $\PK$, il suffit de traiter le cas o{\`u} $z =0 =
R(z)$. Par un changement ad{\'e}quat de coordonn{\'e}es analytiques, on
peut supposer $R(z) = az^k + O(z^{k+1})$ avec $a \neq0$ et $k\ge1$, et
on veut montrer $\deg_R(0) =k$. Lorsque la caract{\'e}ristique de $K$ ne
divise pas $k$, alors on peut se ramener {\`a} $R(z) = z^k$. Sinon la
situation est plus compliqu{\'e}e mais dans de bonnes coordonn{\'e}es, on
peut toujours {\'e}crire $R(z) = az^k + \sum_{j>k} a_j z^j$ avec
$|a_j|\to 0$, $\max\{|a|, |a_j|\} =1$. Il faut montrer que
l'anneau de s{\'e}rie convergente $K\{ z \}$ est un module libre sur
$K\{ R(z)\}$ de rang $k$. Pour cela, on montre qu'il est engendr{\'e} par la
famille libre $1, z, ... , z^{k-1}$. Le fait que cette famille soit libre
est facile. Pour conclure, il suffit de trouver des $f_i \in K\{z \}$
tels que $z^k = f_0 \circ R + ... + z^{k-1} f_{k-1} \circ R$, ce qui
se fait classiquement en r{\'e}solvant l'{\'e}quation successivement
modulo $z^n$ avec $n$ croissant.

Pour v{\'e}rifier (*), on remarque que sous les hypoth{\`e}ses de
l'{\'e}nonc{\'e}, l'application analytique entre courbes analytiques $R: U
\to V$ est ferm{\'e}e (au sens de Berkovich) et non localement constante,
et donc plate par~\cite[3.2.9]{berko2}.  On regarde maintenant le
faisceau $R_* \cO_U$, dont la fibre en $w\in V$ est donn{\'e}e par
$(R_*\cO_U)_w \simeq \oplus_{z \in R^{-1}(w)\cap U} \cO_{z}$. 
Comme $R$ est plat, $(R_* \cO_U)_w$ est libre sur $\cO_w$. Or le
faisceau $R_*\cO_U$ est
coh{\'e}rent~\cite[9.4.4/3]{remmert},~\cite[1.3.4]{berko2}, donc est
lui-m{\^e}me localement libre. Il est en particulier de rang
constant\footnote{le rang d'un faisceau $\cF$ en $z$ est {\'e}gal {\`a}
$\dim_{\kappa(z)} \cF_z/\fm_z\cF_z$}, ce qui implique (*).
\end{proof}

\subsubsection*{Action sur les mesures positives}
Rappelons que pour toute fonction $f: \PKber \to\PKber$ continue et pour toute mesure (sign\'ee)
de Radon $\rho$ sur $\PKber$, on d\'efinit l'image directe $f_*\rho$ en imposant
$\int \phi\, d(f_*\rho) = \int (\phi \circ f) \, d\rho$ pour toute fonction mesurable. On~a alors pour tout bor\'elien $E$, $(f_*\rho)(E) = \rho (f^{-1}(E))$. Si $\rho$ est une mesure de probabilit\'e, alors $f_*\rho$ l'est encore.

On d{\'e}finit maintenant l'action d'une fraction rationnelle
par image r\'eciproque sur les mesures de Radon de mani\`ere analogue.
\begin{Def}
 Pour toute fonction mesurable 
$\phi: \PKber \to \R$, on pose
\begin{equation}\label{eq-pousse}
(R_* \phi)(\cS) \= \sum_{R(\cS') = \cS} \deg_R(\cS') \phi (\cS')~,
\end{equation}
\end{Def}

\begin{Prop}\label{prop-def}
Si $\phi: \PKber \to \R$ est continue, alors $R_*\phi$ l'est aussi, et
$R_*$ d{\'e}finit un op{\'e}rateur
lin{\'e}aire sur l'espace des fonctions continues, v{\'e}rifiant $\sup |R_*
\phi| \le \deg(R) \times \sup|\phi|$.
  
Cet op{\'e}rateur induit donc par dualit{\'e} une action sur les mesures de
Radon, not{\'e}e~$R^*$, telle que pour toute fonction continue~$\phi :
\PKber \to \R$ et toute mesure~$\rho$ d{\'e}finie sur~$\PKber$, on ait
$\int \phi\, d(R^*\rho) = \int (R_*\phi) \, d\rho$.
\end{Prop}
\begin{proof}
Il suffit de montrer la premi{\`e}re assertion.  Fixons $\e>0$, et
$\cS_* \in \PKber$.  Pour chaque $\cS \in R^{-1}(S_*)$ choisissons un
voisinage~$U_{\cS}$ de~$\cS$ tel que $\sup_{U_\cS} |\phi -
\phi(\cS)| \le \e$.  Choisissons de plus un voisinage~$U_*$
de~$\cS_*$ tel que $R^{-1}(U_*) \subset \bigcup_{\cS \in R^{-1}(U_*)}
U_\cS$ et l'application restreinte $R: U_\cS \to U_*$ est de degr\'e $\deg_R(\cS)$.  On obtient alors
$$
\sup_{U_*} |R_*\phi - R_*\phi (\cS_*)|
\le
\sum_{\cS \in R_*^{-1}\{ \cS_*\}} \deg_R(\cS) \times \e \le \deg(R) \times \e ~.
$$
\end{proof}
Indiquons sans preuve quelques propri\'et\'es de ces actions:
\begin{Prop}
Pour toute fraction rationnelle et pour toute mesure de Radon, 
on~a $R_* R^* \rho = d \rho$. Pour tout couple de fractions rationnelles $R_1, R_2$, on~a:
$(R_1 \circ R_2)_* = (R_1)_*\circ (R_2)_*$, et $(R_1 \circ R_2)^* = (R_2)^*\circ (R_1)^*$. 
\end{Prop}
\begin{Prop}\label{prop:calcul}
  Si~$\rho$ est une mesure de probabilit{\'e}, alors~$R^*\rho$ est une
  mesure positive de masse~$\deg(R)$ dont le support topologique est
  {\'e}gal {\`a} la pr{\'e}image par~$R$ du support topologique
  de~$\rho$.  Enfin, pour tout $\cS \in \PKber$ on~a
$$
(R^*\rho) \{ \cS \} = \deg_R(\cS) \times \rho \{ R(\cS) \} ~.
$$
\end{Prop}
\begin{proof}
  Il est clair que $R^*$ pr{\'e}serve la positivit{\'e} des mesures.
  L'image de la fonction constante {\'e}gale {\`a} $1$ par $R_*$ est
  la fonction constante {\'e}gale {\`a} $\deg(R)$
  par~\eqref{e:deg-local}. Donc la masse des mesures est
  multipli{\'e}e par $\deg(R)$ sous l'action de~$R^*$.  Le fait que
  pour une mesure~$\rho$ le support topologique de~$R^*\rho$ soit la
  pr{\'e}image par~$R$ du support topologique de~$\rho$, est une
  cons{\'e}quence imm{\'e}diate du fait que pour chaque fonction
  continue positive~$\phi : \PKber \to \R$ le support de la
  fonction~$R_*\phi$ est {\'e}gal {\`a} l'image par~$R$ du support
  de~$\phi$.

Pour montrer la derni{\`e}re {\'e}galit{\'e}, soient $\e > 0$ un r{\'e}el positif et
$V'$ un ouvert fondamental contenant~$R(\cS)$ tel que $|\rho
\{ R(\cS) \} - \rho (V') |\le \e $, et tel que la composante
connexe~$V$ de~$R^{-1}(V')$ contenant~$\cS$ v\'erifie $|(R^*\rho) \{
\cS \} - (R^*\rho)(V) |\le \e$ et $R^{-1} \{R(\cS)\} \cap V = \{
\cS\}$.  En particulier, pour tout $\cS_0' \in V'$ on~a $\sum_{\cS_0
\in R^{-1}(\cS_0') \cap V} \deg(\cS_0) = \deg_R(\cS)$
(Proposition-D{\'e}finition~\ref{pf:degre local}).
  
  Comme $\PKber$ est un espace normal, il existe une fonction continue
  $\phi$ {\`a} valeurs dans $[0,1]$ support{\'e}e dans $V$ et telle
  que~$\phi(\cS) = 1$.  Alors la fonction~$R_*\phi$ est {\`a}
  valeurs dans~$[0, \deg_R(\cS)]$, {\'e}tant nulle en dehors de~$V'$ et
  {\'e}gale {\`a} $\deg_R(\cS)$ en~$R(\cS)$.  on~a donc
$$ |\deg_R(\cS)\,\rho \{R( \cS) \} - \int (R_*\phi)\, d\rho|\le
\deg_R(\cS) \times \e~, \text{ et } |(R^*\rho) \{ \cS \} - \int
\phi\, d(R^*\rho)|\le \e~.
$$ On tire alors de l'{\'e}galit{\'e} $\int (R_*\phi)\, d\rho = \int
\phi\, d(R^*\rho)$, que
$$
|(R^*\rho)\{ \cS \} - \deg_R(\cS) \, \rho \{  R(\cS) \}|
\le
(\deg_R(\cS) + 1) \times \e.
$$
On conclut en faisant $\e\to0$.
\end{proof}


\subsection{Th{\'e}orie de Fatou et Julia}\label{ss:fatou-julia}
Fixons une fraction rationnelle~$R$ {\`a} coefficients dans~$K$ et de
degr{\'e} au moins~2.  Dans le cas complexe, il est plus
commode de d{\'e}finir en premier lieu l'ensemble de Fatou {\`a}
l'aide de propri{\'e}t{\'e}s d'{\'e}quicontinuit{\'e}. Dans le cas
non archim{\'e}dien il est plus convenable de proc{\'e}der de
mani{\`e}re l{\'e}g{\`e}rement diff{\'e}rente.  On va donc tout
d'abord rappeler quelques faits concernant les points exceptionnels.

\begin{Def}\label{d:exceptionel}
  Soit $R\in K(z)$ une fraction rationnelle de degr{\'e} au moins~2.
Un point $z\in \PK$ est dit \textit{exceptionnel} si l'ensemble
$\bigcup_{n\ge 0} R^{-n} \{ z \}$ est fini. On note $E_R\subset \PK$
l'ensemble des points exceptionnels.
\end{Def}
Lorsque~$\degtop(R) = 1$ la caract{\'e}ristique de~$K$ est strictement positive
et~$R$ est conjugu{\'e}e {\`a} un it{\'e}r{\'e} de l'automorphisme de Frobenius.
L'ensemble exceptionnel est alors infini d{\'e}nombrable et, apr{\`e}s un
changement de coordonn{\'e}es convenable, il est {\'e}gal {\`a} $\mathbb{P}^1(\bar{K})$ o\`u $\bar{K}$ est la fermeture alg{\'e}brique dans~$K$ du corps premier de~$K$.

Lorsque~$\degtop(R) > 1$ on montre que l'ensemble exceptionnel
contient au plus deux {\'e}l{\'e}ments.  Dans le cas o{\`u} il contient deux
points, $R$ est conjugu{\'e}e {\`a} $z^{\pm \deg(R)} \in K(z)$.  Dans le
cas o{\`u} l'ensemble exceptionnel contient un seul {\'e}l{\'e}ment, la
fraction rationnelle~$R$ est un polyn{\^o}me dans toute coordonn{\'e}e
telle que~$\infty$ est l'unique point excep\-tio\-nnel de~$R$.
\begin{Def}\label{d:julia}
  Soit $R\in K(z)$ une fraction rationnelle de degr{\'e} au moins~2.
  L'\textit{ensemble de Julia} de $R$, not{\'e}~$J_R$, est l'ensemble
  des points $\cS \in \PKber$ tels que pour tout voisinage~$U$
  de~$\cS$ contenu dans~$\PKber \setminus E_R$, on~a $\bigcup_{n\ge 0}
  R^n(U) = \PKber \setminus E_R$.  Le compl{\'e}mentaire de l'ensemble
  de Julia est l'\textit{ensemble de Fatou}, not{\'e}~$F_R$.
\end{Def}
Rappelons qu'un ensemble $J$ est dit \emph{totalement invariant} si $R^{-1}(J) \subset J$.
Les propri{\'e}t{\'e}s suivantes sont d{\'e}montr{\'e}es dans~\cite{R4}.
\begin{Prop}\label{prop:generalites Julia}
  L'ensemble de Julia~$J_R$ est compact, non vide, totalement
  invariant, et pour tout $n \ge 1$ on~a $J_{R^n} = J_R$.  De
  plus~$J_R$ est caract{\'e}ris{\'e} comme le plus petit sous-ensemble
  compact non vide de~$\PKber$ qui est disjoint de l'ensemble exceptionnel
  de~$R$ et qui est totalement invariant par~$R$.  De m{\^e}me, $F_R$
  est un ouvert totalement invariant dont l'intersection avec~$\PK$
  est non vide.
\end{Prop}
\subsubsection*{Bonne r{\'e}duction}
Un exemple important d'applications rationnelles a {\'e}t{\'e} mis en
exergue par Morton et Silverman~\cite{MS}.
\begin{Def}
  Une fraction rationnelle $R \in K(z)$ est dite avoir \textit{bonne
  r{\'e}duction}, si on peut l'{\'e}crire $R = P/Q$ avec deux
  polyn{\^o}mes $P, Q \in \cO_K[z]$, dont les r{\'e}ductions
  $\widetilde{P}, \widetilde{Q} \in \tK[\zeta]$ sont non nulles et
  telles que $\widetilde{R} \= \widetilde{P}/\widetilde{Q} \in \tK
  (\zeta)$ ait m{\^e}me degr{\'e} que~$R$.
\end{Def}
On montre alors la proposition suivante.
\begin{Prop}[\cite{R3}, Th{\'e}or{\`e}me~4]\label{prop:bonne}
  Si $R$ a bonne r{\'e}duction, le point $\cScan \in \HKrat$ est
  totalement invariant, et $J_R = \{ \cScan \}$.
  
  R{\'e}ciproquement, si $\cS \in \HK$ est tel que l'ensemble
  $\bigcup_{n \ge 0} R^{-n} (\cS)$ est fini, alors $\cS \in \HKrat$ et
  ce point est totalement invariant. En conjuguant $R$ par un
  automorphisme de M{\"o}bius ad{\'e}quat, on peut de plus supposer
  que $\cS = \cScan$ et dans ce cas~$R$ a bonne r{\'e}duction.
\end{Prop}
On utilisera le lemme suivant dans la preuve du
Th{\'e}or{\`e}me~\ref{t:htopzero}.
\begin{Lem}\label{L-estim-deg}
Pour une fraction rationnelle~$R$ de degr{\'e} au moins deux, il y a deux
cas~: soit~$R$ est conjugu{\'e}e {\`a} une fraction rationnelle ayant
bonne r{\'e}duction, soit il existe un entier positif~$n$ tel que pour
tout~$\cS \in J_R$ on ait
$$
\deg_{R^n}(\cS) < \deg(R)^n ~.
$$
\end{Lem}
\begin{proof}
Si l'ensemble exceptionnel poss{\`e}de au moins deux points, alors~$R$ est conjugu\'ee \`a une fraction
rationnelle ayant bonne r{\'e}duction.  Lorsque l'ensemble exceptionnel poss{\`e}de un
unique {\'e}l{\'e}ment, apr{\`e}s un changement de coordonn{\'e}es on se
ram{\`e}ne au cas o{\`u}~$R$ est un polyn{\^o}me de la forme
$$
R(z) = \rho (z^D + a_{D - 1} z^{D - 1} + \ldots + a_0) ~,
$$ avec $|\rho| \ge 1$, $D \ge 2$, $\max \{ |a_j|, \, j \in \{ 0,
\ldots, D - 1 \} \} = 1$, voir~\cite[Proposition~6.7]{R0}.
Lorsque $|\rho| = 1$ la fraction rationnelle~$R$ a bonne
r{\'e}duction, et lorsque $|\rho| > 1$ on v{\'e}rifie que~$J_R$ est
contenu dans la r{\'e}union des boules ouvertes de~$\PKber$ associ{\'e}es
aux classes r{\'e}siduelles des z{\'e}ros du polyn{\^o}me~$\zeta^D +
\widetilde{a}_{D - 1} \zeta^{D - 1} + \ldots + \widetilde{a}_0 \in
\tK[\zeta]$, et qu'on a~$\deg_R < \deg(R)$ sur cet ensemble.
L'assertion du lemme est alors v{\'e}rifi{\'e}e avec $n = 1$ dans ce cas.

Supposons maintenant que l'ensemble exceptionnel de~$R$ est vide.
Il suffit de montrer alors que si pour chaque $j \ge 1$ l'ensemble
$$
F_j \= \{ \cS \in \PKber, \, \deg_{R^j}(\cS) = \deg(R)^j \}~,
$$ est non vide, alors~$R$ est conjugu{\'e}e {\`a} une fraction
rationnelle ayant bonne r{\'e}duction.  Notons que $F_j = F_1 \cap
R(F_1) \cap \cdots \cap R^{j - 1}(F_1)$, et que~$F_j$ est compact et
d{\'e}croissant avec~$j$.
Il d{\'e}coule de la d{\'e}monstration de~\cite[Lemme~7.4]{R2} que~$F_j$ est connexe lorsqu'il est non
vide.  Par cons{\'e}quent, si pour tout $j \ge 1$ l'ensemble~$F_j$ est
non vide, alors l'ensemble
$$
F \= \bigcap_{j \ge 1} F_j~,
$$ est compact, non vide et connexe.  De plus il est invariant
par~$R$.  Par cons{\'e}quent~$R$ poss{\`e}de un point fixe~$\cS$ dans~$F$,
voir \cite{valtree} ou la \og{}propri{\'e}t{\'e} de point fixe\fg{}
dans~\cite{R1/2}.  Comme par hypoth{\`e}se l'ensemble exceptionnel est
vide, on a~$\cS \in \HK$ et la Proposition~\ref{prop:bonne}
implique que~$R$ est conjugu{\'e}e {\`a} une fraction rationnelle ayant
bonne r{\'e}duction.
\end{proof}

\subsubsection*{Dynamique sur l'ensemble de Fatou}
Comme dans le cas complexe, la dynamique dans l'ensemble de Fatou est
tr{\`e}s r{\'e}guli{\`e}re.  On en donne ici une tr{\`e}s br{\`e}ve
description, et on renvoie {\`a} ~\cite{R0,R1,R1/2,R4} pour plus de d{\'e}tails.  Rappelons tout d'abord quelques d{\'e}finitions.
Fixons une fraction rationnelle~$R \in K(z)$ de degr{\'e} au moins deux.
\begin{Def}
  Un point $z_0 \in \PK$ fix{\'e} par $R$ est dit \textit{attractif}
  (resp.  \textit{r{\'e}pulsif}, ou \textit{indiff{\'e}rent}) si
  localement $R(z_0 + h) = z_0 + \lambda h + \mathcal{O}(h^2)$ avec
  $|\lambda| <1$ (resp. $|\lambda |>1$, ou $|\lambda| =1$).  Un point
  $\cS \in \HK$ fix{\'e} par $R$ est dit \textit{r{\'e}pulsif} si
  $\deg_R(\cS) \ge 2$. Sinon il est dit \textit{indiff{\'e}rent}.

On d{\'e}finit de m{\^e}me la nature d'un point p{\'e}riodique de
p{\'e}riode $N$, c'est la nature de ce point vu comme point fixe de
$R^N$.
\end{Def}

Il est facile de voir que tout point p{\'e}riodique attractif et que
tout point p{\'e}riodique indiff{\'e}rent dans~$\PK$ appartient {\`a}
l'ensemble de Fatou.  Par contre, un point fixe $\cS \in \HK$
indiff{\'e}rent peut {\^e}tre dans l'ensemble de Fatou ou dans
l'ensemble de Julia.  D'autre part, tout point p{\'e}riodique
r{\'e}pulsif (dans~$\PK$ ou dans~$\HK$) est dans l'ensemble de Julia,
voir~\cite[Proposition~5.1]{R1/2} ou~\cite{R4}.

L'image par~$R$ d'une composante connexe de l'ensemble de Fatou est
aussi une composante connexe de l'ensemble de Fatou.  La fraction
rationnelle induit alors une action sur les composantes connexes de
l'ensemble de Fatou.  On dira qu'une composante connexe de l'ensemble
de Fatou est \textit{errante} si son orbite sous cette action est
infinie, et on dira qu'elle est \textit{pr{\'e}p{\'e}riodique} si son
orbite est finie.

Le \textit{bassin d'attraction} d'un point p{\'e}riodique~$\cS_0$ de
p{\'e}riode~$n \ge 1$, est par d{\'e}finition l'ensemble des points
dans~$\PKber$ qui convergent vers~$\cS_0$ sous l'action de~$R^n$.
Lorsque~$\cS_0$ est attractif, cet ensemble est ouvert et
contient~$\cS_0$ dans son int{\'e}rieur.  Dans ce cas la composante
connexe du bassin d'attraction contenant~$\cS_0$ sera appel{\'e}e
\textit{bassin d'attraction imm{\'e}diat}.

Le \textit{domaine de quasi-p{\'e}riodicit{\'e}}~$\cE_R$ de~$R$ est
l'ensemble des points dans~$\PKber$ poss{\'e}dant un voisinage sur
lequel une sous-suite des it{\'e}r{\'e}es de~$R$ converge uniform{\'e}ment vers l'identit{\'e}.  Par d{\'e}finition l'ensemble~$\cE_R$ est ouvert et
invariant par~$R$.  De plus on montre qu'il est contenu dans
l'ensemble de Fatou.

Nous aurons besoin du r{\'e}sultat suivant dans la d{\'e}monstration du Th{\'e}or{\`e}me~\ref{t:equi-periodiques}.
\begin{Lem}\label{lem:non qp}
Si la caract{\'e}ristique r{\'e}siduelle de~$K$ est nulle, alors l'ensemble de quasi-p{\'e}riodicit{\'e} d'une fraction rationnelle~$R$ de degr{\'e} au moins deux est vide.
\end{Lem}
\begin{proof}
  Supposons par contradiction que le domaine de quasi-p{\'e}riodicit{\'e} ne soit pas vide.
Apr{\'e}s un changement de coordon{\'e}e on se ram{\`e}ne au cas o{\`u} il existe une suite d'it{\'e}r{\'e}s de~$R$ qui convergent uniform{\'e}ment vers l'identiti{\'e} sur~$\cO_K$.
Il existe alors une suite strictement croissante d'entiers $\{ n_j \}_{j \ge 0}$ telle que pour tout $j \ge 0$,
$$ \sup \{ |R^{n_j}(z) - z|, z \in \cO_K \} < 1. $$
Notons en particulier que~$R^{n_j}$ induit une bijection de~$\cO_K$.
Par cons{\'e}quent, pour chaque $j, j' \ge 0$, on~a
\begin{multline*}
\sup \{ |R^{n_{j'} + n_j}(z) - z|, z \in \cO_K \}
\le \\
\max \{ \sup \{ |R^{n_j}(z) - z|, z \in \cO_K \}, \sup \{ |R^{n_{j'}}(z) - z|, z \in \cO_K \} \}.
\end{multline*}
On conclut qu'il existe une suite d'it{\'e}r{\'e}s de~$R^{n_0}$ qui converge uniform{\'e}ment vers l'identit{\'e} sur~$\cO_K$.

Comme le degr{\'e} de~$R^{n_0}$ est au moins deux, il existe une classe r{\'e}siduelle ne contenant aucun point fixe de~$R^{n_0}$.
Apr{\'e}s un changement de coordonn{\'e}e fixant~$\cO_K$, on se ram{\`e}ne au cas o{\`u}~$\MK$ ne contient aucun point fixe de~$R^{n_0}$.
Soit~$T$ la fraction rationnelle donn{\'e}e par $T(z) = R^{n_0}(z) - z$.
Alors~$T$ ne s'annule pas sur~$\MK$, et satisfait $\sup \{ |T(z)|, z \in \MK \} < 1$.
Par cons{\'e}quent il existe $a \in \MK$ tel que pour chaque $z \in \MK$ on~a
$$ T(z) \in \{ w \in K, |T(w) - a| < |a| \}. $$
Ceci implique que pour tout entier~$n \ge 1$ et tout~$z \in \MK$ on~a
$$ |R^{nn_0}(z) - z - na| = |T(z) + T(R^{n_0}(z)) + \cdots + T(R^{(n - 1)n_0}(z)) - na| < |a|.
$$
Comme par hypoth{\`e}se la caract{\'e}ristique r{\'e}siduelle de~$K$ est nulle, on a~$|n| = 1$  et donc $|R^{nn_0}(z) - z| = |a|$.
On obtient alors une contradiction avec le fait qu'une suite d'it{\'e}r{\'e}s de~$R^{n_0}$ converge vers l'identit{\'e} sur~$\MK$.
\end{proof}

Lorsque $K = \C_p$ le r{\'e}sultat suivant est exactement~\cite[Proposition~5.6]{R1}.
La d{\'e}\-monstration donn{\'e}e dans op.cit. s'applique sans changement dans ce cadre plus g{\'e}n{\'e}ral.
\begin{Thm}\label{t:qp}
Soit~$K$ un corps de caract{\'e}ristique z{\'e}ro et de caract{\'e}ristique r{\'e}siduelle $p > 0$, et
soit~$R$ une fraction rationnelle {\`a} coefficients dans~$K$ de degr{\'e}
au moins deux.
Alors pour chaque composante connexe~$Y$ de~$\cE_R$ il existe un entier positif~$n$ tel que $R^n(Y) = Y$ ainsi qu'une action continue
\begin{eqnarray}
T : \Z_p \times Y & \to & Y \\
(w, y) & \mapsto & T^w(y)~,
\end{eqnarray}
telle que pour chaque entier positif~$m$ on ait $T^{m} = R^{n m}$.
Si de plus $\infty \not \in Y$, alors il existe une fonction holomorphe non constante $T_* : Y \to K$ telle que pour tout $w_0 \in \Z_p$ la fonction holomorphe $\frac{T^w - T^{w_0}}{w - w_0}$ converge vers $T_*$ lorsque $w \in \Z_p \setminus \{ w_0 \}$ converge vers~$w_0$.
\end{Thm}
Nous utiliserons la proposition suivante dans la d{\'e}monstration du Th{\'e}or{\`e}me~\ref{t:estimation htop}.
\begin{Prop}\label{p:classification}
Soit~$R$ une fraction rationnelle {\`a} coefficients dans~$K$ et de degr{\'e} au moins deux, et~$U$ une composante composante connexe de l'ensemble de Fatou~$U$ fix{\'e}e par~$R$.
Si~$U$ n'est pas un bassin d'attraction imm{\'e}diat d'un point p{\'e}riodique attractif, alors~$R$ est injective sur~$U$.
\end{Prop}
Lorsque~$K = \C_p$ ce r{\'e}sultat est une cons{\'e}quence imm{\'e}diate du \og{}th{\'e}or{\`e}me de classification\fg{} dans~\cite{R1}.
Pour un corps de base~$K$ quelconque, ce r{\'e}sultat est d\'emontr\'e 
dans l'article en pr{\'e}paration~\cite{R4}.
En voici une esquisse de d{\'e}monstration.
\begin{proof}[Esquisse de d{\'e}monstration]
Si~$U$ n'est pas un bassin d'attraction imm{\'e}diat d'un point p{\'e}riodique attractif, la \og{}propri{\'e}t{\'e} de point fixe~\fg{}~\cite[\S8]{R1/2} implique que~$U$ contient un point fixe indiff{\'e}rent~$\cS$.
On regarde alors l'ouvert connexe maximal $U'$ contenant $\cS$ sur lequel le degr\'e local est constant \'egal \`a $1$. On montre que $U'$ est un ouvert fondamental en suivant~\cite[Proposition~2.9]{R2}.
On peut alors adapter~\cite[Lemme~5.5]{R2} pour voir qu'il existe un ouvert fondamental $U''$
contenant $\cS$ et contenu dans $U'$ dont les points du bord sont p{\'e}riodiques
r{\'e}pulsifs. Ceci implique l'invariance de $U''$, et donc le fait que $U''$ est inclus dans l'ensemble de Fatou. Comme les points r\'epulsifs sont dans l'ensemble de Julia, on a finalement $U'' =U$. De la Proposition~\ref{pf:degre local}, et du fait que $\deg_R =1 $ sur $U$, on en d\'eduit que $R$ est injective sur $U$.
\end{proof}


\subsection{Th{\'e}orie du potentiel $\PKber$}\label{ss:potentiel}   
Plusieurs approches~\cite{BR2,valtree,thuillier} existent pour
cons\-truire un analogue de la th{\'e}orie du potentiel complexe sur la
droite projective au sens de Berkovich. La plus aboutie est due {\`a}
Thuillier et s'applique sur n'importe quelle courbe lisse $K$-analytique.
Nous ne d{\'e}crirons cependant ici que les {\'e}l{\'e}ments qui nous seront
strictement n{\'e}cessaires par la suite en suivant~\cite[\S7]{valtree} et~\cite[\S4]{FRL}.

\medskip

Munissons $\PKber$ de la tribu des bor{\'e}liens associ{\'e}e {\`a} sa
topologie compacte.  On note~$\cM^+$ l'ensemble des mesures
bor{\'e}liennes positives finies, support{\'e}es dans~$\PKber$.  On
d{\'e}signe par~$\cM$ l'espace vectoriel des mesures r{\'e}elles
sign{\'e}es, diff{\'e}rences de mesures dans~$\cM^+$.  Toute suite de
mesures de probabilit{\'e} dans~$\cM^+$ admet une sous-suite
convergente pour la topologie de la convergence vague.  Notons que
toute mesure dans $\cM$ est de Radon et est donc repr{\'e}sent{\'e}e
par une forme lin{\'e}aire continue sur l'espace des fonctions
continues de $\PKber$ (voir par
exemple~\cite[Proposition~7.14]{valtree}).  Enfin on montre que bien
que~$\PKber$ muni de la topologie compacte ne soit pas m{\'e}trisable en
g{\'e}n{\'e}ral, le support de toute mesure dans~$\cM$ est 
m{\'e}trisable~\cite[Lemma~7.15]{valtree}.

\medskip

Nous allons maintenant d{\'e}finir un espace fonctionnel~$\cP$ et un
op{\'e}rateur~$\Delta$ d{\'e}fini sur~$\cP$ et {\`a} valeurs dans $\cM$.  Pour
cela, fixons un point base $\cS_0 \in \HK$.  Notons que pour tout $\cS
\in \HK$, la fonction $\cS' \mapsto \langle \cS , \cS'
\rangle_{\cS_0}$ est non n{\'e}gative et major{\'e}e par $\dhyp(\cS, \cS_0)$.
Etant donn{\'e}e une mesure bor{\'e}lienne $\rho \in \cM$, on peut donc
d{\'e}finir $\hg_\rho : \HK \to \R$ par
\begin{equation}\label{e:def-pot}
 \hg_\rho(\cS) \= - \rho(\PKber) - \int_\PKber \langle \cS , \cS'
\rangle_{\cS_0}\,  d \rho(\cS'),
\end{equation}
et on l'appelle le {\it potentiel de $\rho$ bas{\'e} en $\cS_0$}.  On
a $\hg_\rho(\cS_0) = -\rho (\PKber)$ et par
construction~$\hg_{[\cS_0]}$ est la fonction constante {\'e}gale {\`a}
$-1$ sur tout $\PKber$. Plus g{\'e}n{\'e}ralement, $\hg_{[\cS']}(\cS)
= -1 - \langle \cS , \cS' \rangle_{\cS_0}$.

On d{\'e}signe par $\cP$ l'ensemble de tous les potentiels.  C'est un
espace vectoriel qui contient toutes les fonctions d{\'e}finies
sur~$\HK$ et {\`a} valeurs r{\'e}elles de la forme $\cS \mapsto \langle
\cS \, , \cS' \rangle_{\cS_0}$.  Il r{\'e}sulte
de~\cite[Theorem~7.50]{valtree} que l'application $\rho
\mapsto \hg_\rho$ induit une bijection entre~$\cM$ et~$\cP$.  On peut
donc poser
$$
\Delta \hg_\rho \= \rho - \rho(\PKber) \times [\cS_0]~.
$$
Ceci d{\'e}finit une application lin{\'e}aire $\Delta : \cP \to \cM$ que
l'on appelle {\it le Laplacien}. On v{\'e}rifie que cet op{\'e}rateur ne
d{\'e}pend pas du choix du point base ainsi que la classe de fonctions $\cP$, voir~\cite[Proposition~4.1]{FRL}. On~a de plus $\Delta g =0 $ si et seulement 
si $g$ est constante.

Par construction, pour tout $g\in \cP$ on~a $\Delta g ( \PKber) =0$.
R{\'e}ciproquement, toute mesure v{\'e}rifiant $\rho (\PKber ) =0$ est
le Laplacien d'une fonction de $\cP$. Dans toute la suite, on
appellera \emph{potentiel} d'une mesure bor{\'e}lienne $\rho$ toute
fonction $g \in \cP$ telle que $\rho = \Delta g$.  Notons que:
$$
\Delta \langle \cdot \, , \cS \rangle_{\cS_0} = [\cS_0] - [\cS]~, \text{ et } 
\Delta \log \sup \{ \cdot, \cS \} = [\cS]- [\infty]~.
$$ De plus, pour tout potentiel $g : \HK \to \R$ et toute fraction
rationnelle non constante~$R$, on~a $R^* (\Delta g) = \Delta (g \circ
R)$ et $R_* \Delta (g) = \Delta (R_*g)$.

Soit $\cP_+$ l'ensemble des fonctions $g\in\cP$ telles que $[\cScan]
+ \Delta g$ soit une mesure positive.  L'op{\'e}rateur $\Delta$
poss{\`e}de des propri{\'e}t{\'e}s de continuit{\'e} remarquables sur
$\cP_+$ qui sont analogues aux propri{\'e}t{\'e}s de continuit{\'e}
des fonctions sous-harmoniques en analyse complexe, et que nous allons
maintenant d{\'e}tailler.

On v{\'e}rifie tout d'abord que toute fonction dans $\cP_+$ est
continue, convexe pour la m\'etrique $\dhyp$ et 
d{\'e}crois\-sante sur tout segment $[\cScan ,
\cS]$.  Une telle fonction s'{\'e}tend donc de mani{\`e}re unique
{\`a}~$\PKber$ en restant continue sur tout segment.  Cette
extension est semi-continue sup{\'e}rieurement sur~$\PKber$  pour
la topologie compacte.
\begin{Prop}[\cite{valtree}, Theorem~7.64]
  L'espace $\cP_+$ est convexe et ferm{\'e} pour la topologie de la
  convergence simple de $\HK$. De plus, l'op{\'e}rateur induit $(\Delta+[\cScan]): \cP_+
  \to \cM_+$ est continu si $\cM_+$ est muni de la topologie de la
  convergence vague des mesures.
\end{Prop}
Le fait suivant est un r\'esultat de base en th\'eorie du potentiel~\cite[Th\'eor\`eme~1.6.13]{hor}. 
\begin{Prop}[Lemme de Hartogs]\label{prop:hartogs}
Soit $\{ g_n \}_{n \ge 1}$ une suite de potentiels dans~$\cP_+$ telle
que pour chaque $n \ge 1$ on ait $g_n \le 0$.  Alors soit cette suite
converge uniform{\'e}ment sur $\PKber$ vers la fonction constante
{\'e}gale {\`a} $-\infty$; soit il existe une sous-suite $\{ g_{n_k} \}_{k
\ge 1}$ convergeant ponctuellement sur $\HK$ vers un potentiel~$g$ dans~$\cP_+$.
Dans ce dernier cas, pour toute fonction continue~$\varphi$ et toute
partie compacte~$C$ de~$\PKber$, on~a
$$ \limsup_{k \to + \infty} \left( \sup_C (g_{n_k} - \varphi) \right)
\le \sup_C (g - \varphi) ~.
$$
\end{Prop}
\begin{proof}
Pour chaque entier~$n \ge 1$ on pose $\rho_n \= [\cScan] + \Delta
g_n$.  C'est une mesure positive de masse $1$.  Quitte {\`a} extraire
une sous-suite, on peut donc supposer que~$\rho_n$ converge vers une
mesure de probabilit{\'e}~$\rho$ lorsque $n \to + \infty$.
  
Supposons tout d'abord que $\lim_{n \to + \infty} \{ g_n(\cScan), n \ge 1 \} = -
\infty$.  Comme $g_n$ est d{\'e}crois\-san\-te sur tout segment partant de
$\cScan$, on a $g_n \to -\infty$ uniform{\'e}ment.

Supposons maintenant que $\varlimsup_{n \to + \infty} \{ g_n(\cScan), n \ge 1 \} > - \infty$.
Quitte {\`a} remplacer $\{ g_n \}_{n \ge 1}$ par $\{ g_n -
g_n(\cScan)-1 \}_{n \ge 1}$ et {\`a} prendre une sous-suite, on peut supposer
que pour chaque $n \ge 1$ on~a $g_n(\cScan) =-1$.  {\`A} nouveau par
d{\'e}croissance des $g_n$, on a $g_n\le -1$ partout.  Soit
$\hat{g}_{\rho_n}$ d{\'e}finie par~\eqref{e:def-pot} lorsque le point
base~$\cS_0$ est {\'e}gal {\`a}~$\cScan$.  Par construction, $ \Delta
\hat{g}_{\rho_n} = \Delta g_n$ et ces deux fonctions sont {\'e}gales en~$\cScan$.
Elles sont donc {\'e}gales partout.
Par convergence domin{\'e}e, on~a $\hat{g}_{\rho_n} \to \hat{g}_{\rho}$ ponctuellement sur $\HK$, et par cons{\'e}quent on~a $g_n \to \hat{g}_{\rho}$ ponctuellement.

Pour montrer la derni{\`e}re assertion de la proposition, soit $\{ \cT_m
\}_{m \ge 1} \subset \HK$ une suite croissante d'arbres ferm\'es finis dont l'adh\'erence de la
r{\'e}union contient le support topologique de~$\rho$,
voir \cite[Lemma~7.15]{valtree}.  
Quitte {\`a} agrandir les~$\cT_m$, on suppose que tous ces arbres
contiennent~$\cScan$.  La suite de mesures de
probabilit{\'e}~$\{ \rho_{\cT_m} \}_{m \ge 1}$ converge vers~$\rho$, et $g \circ \pi_{\cT_m}$ d\'ecroit
vers $g$ ponctuellement. 

D'autre part, pour un entier~$m \ge 1$ fix{\'e}, la fonction $g_n|\cT_m$ est d\'ecroissante sur~$\cT_m$ et converge ponctuellement vers la fonction continue $g|\cT_m$ lorsque~$n \to + \infty$.
Par consequent la convergence est uniforme sur $\cT_m$.
On en d\'eduit que $g_n \circ \pi_{\cT_m}\to g \circ \pi_{\cT_m}$ uniform\'ement sur $\PKber$.
On a donc pour tout~$m \ge 1$
$$ \limsup_{n \to + \infty} \left( \sup_{C} \left(g_n - \varphi
\right) \right) \le \lim_{n \to + \infty} \left( \sup_{C} \left( g_n
\circ \pi_{\cT_m} - \varphi \right) \right) = \sup_{C} \left( g \circ
\pi_{\cT_m} - \varphi \right)~.
$$

Pour chaque $m \ge 1$ prenons $\cS_m \in C$ tel que
$$
(g \circ \pi_{\cT_m} - \varphi)(\cS_m)
=
\sup_C \, (g \circ \pi_{\cT_m} -  \varphi) ~.
$$ Quitte {\`a} extraire une sous-suite, on suppose que $\{ \cS_m \}_{m
\ge 1}$ converge vers un certain point~$\cS$.  Etant donn{\'e}~$\e > 0$
on choisit un ouvert fondamental~$U$ contenant~$\cS$ tel que $\sup_U |\varphi-
\varphi(\cS)| \le \e$.  Notons~$\cS_0$ le point du bord de $U$
le plus proche de $\cScan$ (si $\cScan \in U$, on pose $\cS_0 \= \cScan$). 
Comme les fonctions $g, g \circ \pi_{\cT_m}$ sont
d{\'e}croissantes sur tout segment partant de~$\cScan$, pour~$m \ge 1$
assez grand tel que $g \circ \pi_{\cT_m} (\cS_0) \le g(\cS_0) + \e$
et~$\cS_m \in U$, on~a,
\begin{multline*}
\sup_C\, (g-\varphi) \ge \sup_{U \cap C} g - \varphi(\cS_0) -\e =
g(\cS_0) - \varphi(\cS_0) - \e \ge \\ \ge g \circ \pi_{\cT_m}(\cS_0) -
\varphi(\cS_0) - 2\e \ge g \circ \pi_{\cT_m}(\cS_m) - \varphi(\cS_0) -
2\e \ge \\ \ge g \circ \pi_{\cT_m}(\cS_m)- \varphi(\cS_m) - 3\e =
\sup_C \, (g \circ \pi_{\cT_m} - \varphi) - 3\e~.
\end{multline*}
Donc $\sup_C (g-\varphi) \ge \limsup_{m \to + \infty} \left( \sup_C \,
(g \circ \pi_{\cT_m} - \varphi) \right)$.  Ce qui termine la preuve.
\end{proof}

%
%
%
%

\section{La mesure d'{\'e}quilibre}\label{s:equilibre}
Fixons une fraction rationnelle~$R$ de degr{\'e} au moins~2.

%
%
\subsection{Construction de la mesure}\label{ss:construction}
\begin{Prop-def}\label{p:def}
Pour tout {\'e}l{\'e}ment $\cS \in \HK$, la suite de mesures de
probabilit{\'e} $\{ \deg(R)^{-n} R^{n*} [\cS] \}_{n \ge 0}$ converge
vaguement vers une mesure~$\rho_R$ ind{\'e}pendante du choix de~$\cS$.
On appellera cette mesure la \emph{mesure d'{\'e}quilibre} de~$R$.
\end{Prop-def}
\begin{proof}
  Prenons $\cS$ un point arbitraire de $\HK$. La masse de la
  mesure~$R^*[\cS]$ est {\'e}gale au degr{\'e} de $R$, donc nous
  pouvons {\'e}crire $\deg(R)^{-1}R^* [\cS] = [\cS] + \Delta g$, o{\`u}
  $g: \HK \to \R$ est un potentiel.  La mesure~$R^*[\cS]$ est
  support{\'e}e sur le sous-ensemble fini $R^{-1}(\cS)$ de~$\HK$.
  L'enveloppe convexe de l'ensemble $\supp R^*[\cS] \cup \{ \cS\}$ est
  donc un arbre fini $\cT$ dont les bouts sont situ{\'e}s dans $\HK$.
  Le potentiel~$g$ est localement constant en dehors de~$\cT$ et est
  born{\'e} sur~$\cT$, il est donc born{\'e} sur~$\HK$ tout entier.
  On obtient donc pour tout entier~$n$~:
$$ \deg(R)^{-n} R^{n*}[\cS] = [\cS] + \Delta g_n~, \text{ avec } g_n =
\sum_{k = 0}^{n-1} \frac{g\circ R^k}{\deg(R)^k}~.
$$ La suite $\{ g_n \}_{n \ge 0}$ converge uniform{\'e}ment sur $\HK$
vers une fonction continue $g_\infty$, donc $\deg(R)^{-n} R^{n*}[\cS]$
con\-verge vaguement vers une mesure $\rho_R$ lorsque $n \to +
\infty$.

Si l'on choisit un autre {\'e}l{\'e}ment $\cS'\in\HK$, on peut {\'e}crire
$[\cS'] = [\cS] + \Delta h$ avec $h :\HK \to \R$ born{\'e}e. De l'{\'e}quation
$$
\deg(R)^{-n} R^{n*}[\cS']
=
\deg(R)^{-n} R^{n*}[\cS] + \Delta (\deg(R)^{-n} h\circ R^n)~,
$$
on tire $\deg(R)^{-n} R^{n*}[\cS']\to \rho_R$.
\end{proof}
\begin{Rem}\label{rem-charge}
  L'argument montre que plus g{\'e}n{\'e}ralement pour toute mesure
  $\rho$ v{\'e}rifiant $\rho = [\cS] + \Delta h$, avec~$h$ born{\'e}e
  et $\cS\in\HK$, on a $\deg(R)^{-n} R^{n*}\rho \to \rho_R$.  Nous
  verrons {\`a} la section suivante comment {\'e}tendre ce type de
  r{\'e}sultat dans certains cas o{\`u} la fonction~$h$ n'est plus
  born{\'e}e, par exemple lorsque~$\rho$ est atomique et support{\'e}e
  sur un point non exceptionnel dans~$\PK$.
\end{Rem}
L'op{\'e}rateur $\deg(R)^{-1}R^*$ est continu et pr{\'e}serve l'espace
des mesures de probabilit{\'e}s. Il est donc clair que $\rho_R$ est
une mesure de probabilit{\'e}
qui v{\'e}rifie l'{\'e}quation d'invariance:
\begin{equation}\label{eq-inv}
R^*\rho_R = \deg(R)\, \rho_R~.
\end{equation}
En particulier, c'est une mesure invariante $R_*\rho_R = \rho_R$.  L'{\'e}quation
d'invariance~\eqref{eq-inv}, combin{\'e}e {\`a} la
Proposition~\ref{prop:calcul} donne de plus pour tous
$\cS,\cS'\in\PKber$ tels que $R(\cS') = \cS$:
\begin{equation}\label{eq-inv2}
\rho_R \{\cS'\} = \frac{\deg_R(\cS')}{\deg(R)}\, \rho_R\{\cS\}
~.
\end{equation}
\begin{Prop}\label{p:holder}
Il existe $\alpha \in (0, 1]$ tel que pour chaque $\cS \in \HK$, tout
potentiel $g_0: \HK \to \R$ satisfaisant $\rho_R = [\cS] + \Delta g_0$
est H{\"o}lder d'exposant~$\alpha$ par rapport {\`a} la distance~$\dpk$.
Plus pr{\'e}cis{\'e}ment, on peut trouver une constante $C>0$ telle
que pour tous $\cS', \cS'' \in \HK$, on ait
$$
|g_0(\cS') - g_0(\cS'')| \le C\, \dpk (\cS', \cS'')^\alpha~.
$$
De plus, pour tout $z\in\PK$ et $r>0$, on a
$$
\rho_R (\sB(z,r)) \le C \, r^\alpha~,
$$
o{\`u} $\sB(z,r)$ d\'esigne l'enveloppe convexe dans $\PKber$ de la boule $\{w \in K, \dpk (z,w) \le r \}$.
En particulier, $\rho_R$ ne charge pas les points de~$\PK$.
\end{Prop}
En dimension sup{\'e}rieure, une version de ce r{\'e}sultat a {\'e}t{\'e} d{\'e}montr{\'e}e dans~\cite{KawSil}.
Notons que ce fait est classique sur le corps des complexes, et reste
valable en dimension sup{\'e}rieure~\cite[\S 1.7]{Sib99}
ou~\cite{DinSib05}.
La preuve que nous donnons ici suit celles de~\cite[Th\'eor\`eme~3.7.1]{DinSib03}, et~\cite[Proposition~1.2]{guedj}.
\begin{proof}
Soit $\cScan$ le point de $\HKrat$ associ{\'e} {\`a} la boule
unit{\'e} de~$K$.  Pour chaque $\cS \in \HK$, tout potentiel
$g_1$ satisfaisant $\Delta g_1 = [\cScan] - [\cS]$ est Lipschitz car $\langle \cdot, \cS\rangle_{\cScan}$ l'est. On
peut donc se ramener au cas $\cS = \cScan$.
Soit alors $g
 :\HK\to\R$ un potentiel tel que
$$ \Delta g = \deg(R)^{-1}R^*[\cScan] - [\cScan]. $$
La fonction~$g$ est Lipschitz sur $\PKber$, on peut donc trouver une
constante $C>0$ telle que $|g(z) - g(w)|\le C\cdot \dpk(z,w)$ pour
tous $z,w\in\PKber$.
D'autre part, la fonction~$R$ est aussi Lipschitz pour une autre constante~$M$ que l'on peut supposer strictement plus grande que~$\deg(R)$, voir par exemple~\cite[Theorem~10]{KawSil}.
De $\rho_R =\lim_{n \to + \infty} \deg(R)^{-n} R^{n*} [\cScan]$, on d{\'e}duit que $\rho_R = [\cScan] + \Delta g_0$, avec
$$ g_0 = \sum_0^\infty \deg(R)^{-k} g\circ R^k. $$
On fixe~$N$
entier, et on proc{\`e}de alors aux estimations suivantes~:
\begin{eqnarray*}
|g_0(z)-g_0(w)|
& \le &
\sum_{k = 0}^{N-1} \frac{|g\circ R^k(z) - g\circ R^k(w)|}{\deg(R)^{k}} + \sup|g|~
\sum_{k = N}^\infty  \frac1{\deg(R)^{k}} \\
& \le &
C_1 \cdot \dpk(z,w)\, (M/\deg(R))^N + C_2\cdot \deg(R)^{-N} ~,
\end{eqnarray*}
pour des constantes $C_1,C_2>0$ ind\'ependantes de $z,w,N$.  On choisit maintenant~$N = N(z,w)$ de telle
sorte que $M C_2/C_1 \ge \dpk(z,w)\,M^N \ge C_2/C_1$, ce qui est possible d\`es que $\dpk(z,w) \le MC_2/C_1$.
On en d{\'e}duit qu'il existe une constante $C > 0$, ind{\'e}pendante
de~$z$ et~$w$, telle que si l'on pose $\alpha = \log \deg(R) /\log M$,
alors $|g_0(z)-g_0(w)|\le C \cdot \dpk(z,w)^\alpha$.
On note enfin que ces in\'egalit\'es sur $\PK$ entraine les m\^emes sur $\PKber$.

Soit $z \in \PK$ et~$r > 0$.  Pour montrer l'in{\'e}galit{\'e} $\rho_R (
\sB(z,r)) \le C \, r^\alpha$ on se ram{\`e}ne au cas~$r \in (0,
\tfrac{1}{2})$.  Soit~$\cS$ (resp.~$\cS'$) le point de $\PKber$
associ{\'e} {\`a} la boule $B(z, r)$ (resp. $B(z, er)$), et soit $\chi :
\HK \to \R$ le potentiel constant {\'e}gal {\`a}~1 sur $B(z, r)$, constant
{\'e}gal {\`a}~0 hors de $B(z, er)$, et tel que $\Delta \chi = [\cS] -
[\cS']$.  Alors on a
$$
\rho_R(\sB(z, r)) \le \int \chi \, d(\Delta g_0)
=
\int g_0 \, d (\Delta \chi)
=
g_0(\cS) - g_0(\cS')
\le
C\dpk(\cS, \cS')^\alpha
=
Cr^\alpha ~.
$$
La premi\`ere \'egalit\'e r\'esulte de~\cite[Lemmes~4.4 et~4.3]{FRL} car la fonction \'energie est sym\'etrique.
\end{proof}

%
%
\subsection{D{\'e}monstration du Th{\'e}or{\`e}me~\ref{t:equi-measures}}\label{ss:equi-measures}
Le fait que~$\rho_R$ ne charge aucun point de~$\PK$ est une
cons{\'e}quence de la Proposition~\ref{p:holder}.  On montre maintenant
que le fait que le support de~$\rho_R$ soit {\'e}gal {\`a}~$J_R$ est une
cons{\'e}quence de~\eqref{e:convergence generale}.  En
effet, consid{\'e}rons une mesure de probabilit{\'e} quelconque~$\rho$
support{\'e}e dans~$J_R$.  Par la Proposition~\ref{prop:calcul}, le
support de la mesure~$R^{*n} \rho$ est contenu dans~$J_R$.  Comme
l'ensemble $J_R$ est disjoint de l'ensemble exceptionnel, et
compl{\`e}tement invariant par~$R$ (Proposition~\ref{prop:generalites
Julia}), la convergence~\eqref{e:convergence generale} implique que le
support topologique de~$\rho_R$ est contenu dans~$J_R$.  D'autre
part, l'{\'e}quation d'invariance~\eqref{eq-inv} et la
Proposition~\ref{prop:calcul} montrent que le support topologique
de~$\rho_R$ est compl{\`e}tement invariant.  Comme l'ensemble de Julia
est caract{\'e}ris{\'e} comme le plus petit ensemble compact disjoint de
l'ensemble exceptionnel qui est compl{\`e}tement invariant par~$R$
(Proposition~\ref{prop:generalites Julia}), on conclut que le support
topologique de~$\rho_R$ est {\'e}gal {\`a}~$J_R$.

Le reste de cette d{\'e}monstration est consacr{\'e} {\`a} la preuve
de~\eqref{e:convergence generale}.  L'ensemble exceptionnel~$E_R$ {\'e}tant
totalement invariant, pour tout $n \ge 1$ on a
$$
\deg(R)^{-n} R^{n*} \rho(  E_R) = \rho(E_R) ~.
$$ Si $\rho(E_R)>0$, on ne peut pas donc avoir $\lim_{n \to + \infty}
\deg(R)^{-n} R^{n*}\rho =\rho_R$.

R{\'e}ciproquement, supposons que $\rho(E_R)=0$.  On d\'esigne par~$g :
\HK \to \R$ le potentiel donn{\'e} par~\eqref{e:def-pot}, lorsque $\cS_0
= \cScan$, de telle sorte que $\Delta g = \rho - [\cScan]$.  Alors
pour tout $n \ge 1$ on a
$$
\deg(R)^{-n}R^{n*}\rho
=
\deg(R)^{-n}R^{n*}[\cScan] + \Delta ( \deg(R)^{-n} g \circ R^n)~.
$$ Il suffit donc de montrer que $\deg(R)^{-n} g \circ R^n$ converge
ponctuellement vers z{\'e}ro sur~$\HK$, lorsque $n \to + \infty$.
Pour cela on va estimer les vitesses de convergence des points de
$\HK$ vers~$\PK$.  Comme on~a~$g \le 0$, il suffit de montrer
$\liminf_{n \to + \infty} \deg(R)^{-n} g \circ R^n\ge 0$.

Fixons un point~$\cS$ dans~$\HK$.

\noindent\textrm{Cas 1.} \textit{Le point $\cS$ appartient au bassin
d'attraction d'un point exceptionnel $z_0 \in \PK$ de~$R$}.  Quitte
{\`a} remplacer $R$ par un it{\'e}r{\'e} on peut supposer~$z_0$ fix{\'e}
par~$R$.  Apr{\`e}s un changement de coordonn{\'e}es on suppose qu'il
existe $r_0 \in (0, 1)$ tel que pour tout point $\cS' \in \HK$
satisfaisant $\dpk(\cS', z_0) < r_0$ on~a
$$
\dpk(R(\cS'), \PK)
\ge
\dpk(\cS', \PK)^{\deg(R)}~.
$$ Quitte {\`a} remplacer $\cS$ par un point dans son orbite positive,
on suppose que $\dpk(\cS, z_0) < r_0$.  Alors pour tout $n \ge 1$
on~a,
$$
\dpk(R^n(\cS), \PK)
\ge
\dpk(\cS, \PK)^{\deg(R)^n}~.
$$
Fixons~$\varepsilon > 0$, et soit~$r \in (0, 1)$ assez petit tel que~$\rho(\sB(z_0, r)) \le \varepsilon$.
De~\eqref{e:def-pot} on d{\'e}duit que pour tout $\cS' \in \HK$ satisfaisant $\dpk(\cS', z_0) < r$ on a
$$ g(\cS) - g(\cS') \ge \varepsilon \, \dhyp (\cS, \cS'). $$
Comme de plus $B_{\cS'} \subset B_\cS$, par~\eqref{eforget} on a, 
$$ \dhyp(\cS,\cS') = \log \dpk(\cS,\PK) - \log \dpk (\cS',\PK), $$
d'o\`u on obtient,
$$ g(\cS') \ge \varepsilon \log \dpk(\cS', z_0) + \left( g(\cS) - \varepsilon \log \dpk(\cS, z_0) \right). $$
On~a donc
$$
\liminf_{n \to + \infty} \deg(R)^{-n} g(R^n(\cS))
\ge
\e \log \dpk(\cS, \PK)~.
$$

\noindent\textrm{Cas 2.} \textit{Le point~$\cS$ n'appartient pas au
bassin d'attraction d'un point exceptionnel et $\degtop(R) > 1$}.
Comme $\degtop(R) > 1$, il existe au plus un nombre fini de points
dans~$\PK$ dont le degr{\'e} local est {\'e}gal au degr{\'e} de~$R$.  Quitte
{\`a} remplacer~$R$ par un it{\'e}r{\'e} on suppose que tout point de~$\PK$
dont le degr{\'e} local est {\'e}gal {\`a} $\deg(R)$ est exceptionnel.  On s'appuie alors sur le lemme suivant.
\begin{Lem}\label{l:dist} Supposons que $\degtop(R) > 1$. Alors,
on peut trouver des constantes $r > 0$ et $C > 0$ tel que tout point~$\cS$
satisfaisant~$\dpk(\cS, E_R) < r$ appartienne au bassin d'attraction
d'un point exceptionnel, et pour tout point $\cS$ satisfaisant
$\dpk(\cS, E_R) \ge r$ on ait
\begin{equation}\label{e:dist}
\dpk(R(\cS), \PK) \ge C \times \dpk(\cS, \PK)^{\deg(R) -1 } ~.
\end{equation}
\end{Lem}
Nous donnons une preuve de ce lemme ci-apr\`es.

En it\'erant la relation pr\'ec\'edente, on en d\'eduit pour tout entier $n \ge 1$:
$$
\dpk (R^n(\cS), \PK) \ge C(n) \times \dpk (\cS, \PK)^{(\deg(R) - 1)^n} ~,
$$ avec $C(n) \= C \times C^{\deg(R) - 1} \times \cdots \times
C^{(\deg(R) - 1)^{n - 1}}$. De~\eqref{eforget} et~\eqref{e:def-pot}, on tire pour tout $\cS' \in \HK$
$$
g(\cS') \ge \log \dpk (\cS', \PK) -1 ~.
$$
Par cons{\'e}quent
$$
g(R^n(\cS)) \ge \log C(n) + (\deg(R) - 1)^n \log \dpk (\cS, \PK) -1~,
$$ d'o{\`u} $\liminf_{n \to + \infty} \deg(R)^{-n} g(R^n(\cS)) = 0$.
 
\noindent\textrm{Cas 3.} \textit{Le point~$\cS$ n'appartient pas au
bassin d'attraction d'un point exceptionnel et $\degtop(R) = 1$}.
Dans ce cas la caract{\'e}ristique de~$K$ est strictement positive, {\'e}gale {\`a} un
certain nombre premier~$p$.  De plus, le degr{\'e} $q > 1$ de~$R$ est
une puissance de~$p$, et quitte {\`a} faire un changement de
coordonn{\'e}es, on peut supposer que~$R(z) = z^q$.  On note~$\tR$ la
fraction rationnelle {\`a} coefficients dans le corps r{\'e}siduel~$\tK$
de~$K$, d{\'e}finie par~$\tR(\zeta) = \zeta^q$.

On v{\'e}rifie ais{\'e}ment que pour tout $\cS \in \PKber$ on~a
$$
\dhyp (R(\cS), \cScan) = q\times \dhyp (\cS, \cScan) ~.
$$ On~a en particulier $R(\cScan) = \cScan$.  Le cas $\cS = \cScan$
est donc imm{\'e}diat, et on suppose alors que $\cS \in \PKber \setminus
\{ \cScan \}$.  Consid{\'e}rons la partition de $\PKber \setminus \{
\cScan \}$ en boules ouvertes de~$\PKber$ associ{\'e}es aux classes
r{\'e}siduelles,
$$
\PKber \setminus \{ \cScan \} = \sqcup_{ \zeta \in \PtK} \Bber(\zeta) ~,
$$ de telle sorte que pour chaque $\zeta \in \PtK$ on ait $R(\Bber
(\zeta)) = \Bber (\tR(\zeta))$.  On montre (cf. \cite[Proposition~4.32]{R1}) que pour chaque
$\zeta \in \PtK$ p{\'e}riodique par~$\tR$, la boule $\Bber(\zeta)$
contient un point p{\'e}riodique par~$R$ de m{\^e}me p{\'e}riode.  De plus
ce point est exceptionnel pour~$R$ et son bassin d'attraction contient
la boule~$\Bber(\zeta)$.  Comme par hypoth{\`e}se le point~$\cS$
n'appartient pas {\`a} un bassin d'attraction d'un point exceptionnel
de~$R$, si l'on d{\'e}signe par~$\zeta$ l'{\'e}l{\'e}ment de~$\PtK$ tel que
la boule $\Bber(\zeta)$ contienne le point~$\cS$, alors l'orbite
positive de~$\zeta$ par $\tR$ est infinie.  On a en particulier
$$
\lim_{n \to + \infty} \rho(\Bber(R^n(\zeta))) = 0 ~.
$$

Par~\eqref{e:def-pot} pour chaque $\zeta' \in \PtK$ et $\cS' \in \Bber(\zeta')$ on~a
$$
g(\cS') \ge - 1 - \rho(\Bber(\zeta')) \times \dhyp(\cS', \cScan) ~,
$$
d'o{\`u} l'on d{\'e}duit
$$
\liminf_{n \to + \infty} \deg(R)^{-n} g(R^n(\cS))
\ge
\liminf_{n \to + \infty} \rho(\Bber(\tR^n(\zeta))) \times \dhyp(\cS, \cScan)
= 0 ~.
$$

\begin{proof}[D\'emonstration du Lemme~\ref{l:dist}]

L'existence de $r>0$ suit de la finitude de $E_R$.
Pour montrer~\eqref{e:dist}, par compacit\'e il suffit de v\'erifier cette \'equation localement en tout point~$\cS$ tel que
$\dpk(\cS, E_R) \ge r$.
Supposons pour commencer  que $\deg(R) = \degtop(R)$, c'est-\`a-dire que
l'ensemble $\Crit(R) \subset\PK$ o\`u la d\'eriv\'ee de $R$ s'annule soit fini.

Prenons tout d'abord un point $\cS$ qui ne soit pas dans $\Crit(R)$. Quitte \`a changer de coordonn\'ees, on peut fixer un ouvert fondamental $U$ contenant $\cS$ tel que
$$ \{ \infty, R^{-1}(\infty), \Crit(R)\} \cap U = \emptyset. $$
Dans ce cas, $R'$ ne s'annule pas dans $U\cap \PK$, et on peut donc trouver $C_0>0$ tel que $|R'(z)| \ge C_0$ pour tout $z \in U\cap \PK$.
La fonction
$$ D(z_1,z_2) \= \frac{R(z_1)-R(z_2)}{z_1-z_2} - R'(z_2) $$
est analytique sur $U\times U$ et s'annule identiquement sur la diagonale. Pour un voisinage convenable $V$ de la diagonale, et pour tout couple $(z_1,z_2) \in V \cap (\PK \times \PK)$, on a donc
$$| D(z_1,z_2)| < C_0, \text{ et } \left| \frac{R(z_1)-R(z_2)}{z_1-z_2}\right| = |R'(z_2)| \ge C_0~. $$
Prenons maintenant $W$ un voisinage ouvert de $\cS$ dans $\PKber$ tel que $W\times W \subset V$. Alors pour tout $\cS' \in W$, on a $\dpk(R(\cS'), \PK) \ge C_0 \times \dpk (\cS', \PK)$  ce qu'il fallait d\'emontrer.

Si le point $\cS$ est dans $\Crit(R)$, on se ram\`ene au cas o\`u $\cS = R(\cS) =0$ et dans une coordonn\'ee ad\'equate  $R(z) = \sum_{j\ge k} a_j z^j$ avec $ k\le \deg(R)-1$, $a_k \neq0$ et $1 \ge |a_j| \to 0$. 
Notons que $\sup_{|z-w|\le \tau} |z^j-w^j| = \tau^j$ pour tout $\tau \in [0, 1]$, o\`u le supremum est pris sur tout~$z, w \in B(0, 1)$.
Fixons~$N$ assez grand pour que pour tout $\tau \le r$, on ait
$\max_{j \ge 1}  |a_j| \tau^j = \max_{j \le N} |a_j| \tau^j$.
Prenons $z \in B(0,1)$ et $\tau \le r$. Alors
$\sup_{|z-w|\le \tau} |R(z)-R(w)| = \max_{j \le N} |a_j| \tau^j \ge |a_k| \tau^k$.
L'image d'une boule par $R$ \'etant une boule on conclut donc que $R(\sB(z,\tau)) \supset \sB(R(z), |a_k| \tau^k)$, ce qui implique $\dpk(R(\cS), \PK) \ge |a_k| \, \dpk(\cS, \PK)^k$.

Ceci conclut la preuve de~\eqref{e:dist} dans le cas $\deg(R) = \degtop(R)$. Si ce n'est pas le cas, alors $K$ est de caract\'eristique 
positive $p>0$ et $R$ est la compos\'ee d'une fraction $R_0$ v\'erifiant $\deg(R_0) = \degtop(R_0)$ et d'un it\'er\'e de $F(z) \= z^p$. On a vu que  $\dpk(F(\cS), \PK) \ge \dpk(\cS, \PK)^{p}$, ce qui permet de conclure.
\end{proof}

%
%

\subsection{M{\'e}lange et th{\'e}or{\`e}me limite central}\label{ss:melange et clt}
La propri{\'e}t{\'e} d'{\'e}quidistribution d{\'e}montr{\'e}e au
paragraphe pr{\'e}c{\'e}dent permet de montrer l'ergodicit{\'e} de la
mesure d'{\'e}quilibre, comme a {\'e}t{\'e} fait dans~\cite[Lemme~4.11]{FG}.
Afin d'avoir une estimation sur la vitesse de m{\'e}lange nous
utiliserons plut{\^o}t une m{\'e}thode bas{\'e}e sur la notion
d'{\'e}nergie analogue de la preuve de Fornaess-Sibony~\cite{ForSib94,ForSib95} dans
le cas complexe.

Rappelons tout d'abord quelques notions sur l'{\'e}nergie des mesures.
Nous renvoyons {\`a}~\cite{FRL} pour les d{\'e}tails.  Une mesure
sign{\'e}e~$\rho$ est dite \textit{{\`a} potentiel continu} s'il
existe un potentiel continu $g : \HK \to \R$ tel que $\Delta g = \rho
- \rho(\PKber)[\cScan]$.  Si $\rho, \rho'$ sont deux mesures
sign{\'e}es dont la mesure trace\footnote{toute mesure sign\'ee $\rho$ est diff\'erence de deux mesures positives \`a support disjoints, $\rho = \rho_1 - \rho_2$, voir~\cite[\S 6]{rudin}. La mesure trace $|\rho|$ est par d\'efinition la mesure positive $\rho_1 + \rho_2$.}  est {\`a} potentiel continu, on
montre que la fonction $\log \sup \{ \cS , \cS'\}$ est int{\'e}grable
pour la mesure $\rho \otimes \rho'$ (voir~\cite[Lemme 4.3]{FRL}), et
on pose
$$
(\rho, \rho') \= - \int_{\mathsf{A}^1_K \times \mathsf{A}^1_K
  \setminus \mathrm{Diag}} \log \sup \{ \cS, \cS' \} \ d\rho(\cS)
\otimes d\rho'(\cS')~.
$$ Lorsque de plus $\rho = \Delta g$, on v{\'e}rifie que $(\rho,
\rho') = - \int_{\mathsf{A}^1_K} g \, d\rho'$,~\cite[Lemme 4.4]{FRL}.
Enfin, l'application $\rho, \rho' \mapsto (\rho, \rho')$ d{\'e}finit
une forme bilin{\'e}aire sym{\'e}trique sur l'espace des mesures
sign{\'e}es dont la mesure trace est {\`a} potentiel continu et telles que $\rho (\PKber) =
\rho'(\PKber) =0$.  On montre que cette forme bilin{\'e}aire est
d{\'e}finie positive,~\cite[Proposition 4.5]{FRL}.  On a donc
l'in{\'e}galit{\'e} de type Cauchy-Schwarz suivante~:
$$
(\rho, \rho' ) \le (\rho, \rho)^{1/2} \, (\rho' , \rho')^{1/2}~.
$$ D'autre part, une fonction $\phi : \PKber \to \R$ est dite
\textit{de classe~$\cC^1$}, si elle est localement constante hors d'un
arbre fini et ferm{\'e} $\cT \subset \HK$, et si $\cT$ est la
r{\'e}union d'un nombre fini de segments sur lesquels $\phi$ est de
classe $\cC^1$, par rapport {\`a} la distance hyperbolique sur~$\HK$.
On d{\'e}finit alors pour chaque $\cS \in \HK \setminus \{ \cScan \}$
le nombre $\partial \phi (\cS)$, comme la d{\'e}riv{\'e}e {\`a} gauche
en~$\cS$ de l'application~$\phi$ restreinte au segment $[\cScan, \cS]$
(param{\'e}tr{\'e} par la distance hyperbolique {\`a} $\cScan$).  Un
calcul montre alors que
$$
(\Delta \phi, \Delta \phi)
=
\int_{\PKber} (\partial \phi)^2 \, d\lambda~,
$$
o{\`u} $\lambda$ est la mesure de Hausdorff $1$-dimensionnelle sur
chaque segment de $\HK$.  Pour reprendre les notations de ~\cite[\S
5.5]{FRL}, on posera $\langle \phi, \phi \rangle \= \int_{\PKber}
(\partial \phi)^2 \, d\lambda$.

Rappelons finalement qu'une mesure invariante $\rho$ est
\textit{m{\'e}langeante} pour $R$ si pour tout couple de fonctions $\phi, \psi :
\PKber \to \R$ de carr{\'e} int{\'e}grable par rapport {\`a}~$\rho$, on a
$\int (\phi\circ R^n) \, \psi\, d\rho_R\to \int \phi \, d\rho_R \times
\int \psi \, d \rho$ lorsque $n \to + \infty$. Une mesure m\'elangeante est 
\emph{ergodique} au sens o\`u tout sous-ensemble $E$ invariant est soit de mesure
nulle soit de mesure totale. 
\begin{Prop}\label{p:melange}
La mesure~$\rho_R$ est m{\'e}langeante, et elle est exponentiellement
m{\'e}\-langeante par rapport aux observables de classe~$\cC^1$.  Plus
pr{\'e}cis{\'e}ment, il existe une constante $C > 0$ telle que pour
toute fonction $\phi : \PKber \to \R$ dans $L^\infty(\rho_R)$, pour
toute fonction $\psi : \PKber \to \R$ de classe~$\cC^1$ sur~$\PKber$,
et pour tout entier~$n\ge0$ on~a,
\begin{equation}\label{eq:mixing}
\left|\int (\phi\circ R^n) \times \psi\, d\rho_R - \int \phi \,
d\rho_R\times \int \psi \, d \rho_R \right| \le C \times \| \phi
\|_{L^\infty} \times \langle \psi, \psi \rangle^{1/2} \times
\deg(R)^{-n/2} ~.
\end{equation}
\end{Prop}
La d{\'e}monstration de cette proposition est ci-dessous.  Rappelons
qu'une fonction in\-t{\'e}grable $\psi : \PKber \to \R$ de moyenne nulle
$\int \psi \, d\rho_R=0$ v{\'e}rifie le th{\'e}or{\`e}me limite central
de poids $\sigma>0$, si pour tout intervalle~$I$ de~$\R$ on~a
$$
\lim_{n\to + \infty} \rho_R \left\{ \frac{1}{\sqrt{n}} \sum_{k=0}^{n-1}
\psi \circ R^k \in I \right\} = \frac{1}{\sqrt{2\pi\sigma}} \int_I\exp
\left( - \frac{x^2}{2\sigma^2}\right)\, dx~.
$$
La d{\'e}croissance exponentielle des corr{\'e}lations combin{\'e}e au th{\'e}or{\`e}me
de Gordin-Livera\-ni (voir~\cite{gordin,liverani}) donne facilement le r{\'e}sultat suivant.
\begin{Prop}\label{p:clt}
  Pour toute fonction $\psi : \PKber \to \R$ de classe $\cC^1$ telle que $\int
  \psi\, d\rho_R=0$, le nombre
  $$ -\int\psi^2\, d\rho_R + 2 \sum_{n=0}^\infty \int (\psi\circ R^n) \times \psi \, d\rho_R~,$$
est fini, positif et il est {\'e}gal {\`a}~$0$ si et seulement si la fonction~$\psi$ est un cobord, i.e., s'il existe une fonction mesurable~$\chi$ telle que $\psi = \chi - \chi \circ R$.
De plus, lorsque ce nombre est strictement positif, on note~$\sigma > 0$ sa racine carr{\'e}e positive, et la fonction~$\psi$ v{\'e}rifie le th{\'e}or{\`e}me limite central pour le poids~$\sigma$.
\end{Prop}
\begin{proof}
  La preuve est maintenant classique,
  voir~\cite{CanLeB05,DinSib06}. Notons $R^*$
  l'op{\'e}\-ra\-teur de composition agissant sur $L^2(\rho_R)$.
  L'invariance $R_*\rho_R = \rho_R$ implique l'{\'e}galit{\'e} $\int |R^*\psi|^2\,
  d\rho_R = \int |\psi|^2 \, d\rho_R$ donc $R^*$ est une
  isom{\'e}trie. Soit $\Lambda$ l'adjoint de $R^*$.  L'{\'e}qua\-tion
  d'invariance $R^* \rho_R = \deg(R) \rho_R$ montre facilement que
  $\Lambda = \deg(R)^{-1}\, R_*$.
  
D'apr{\`e}s le th{\'e}or{\`e}me de
Gordin-Liverani~\cite{gordin,liverani}, pour montrer qu'une fonction
$\psi \in L^\infty(\rho_R)$ satisfaisant $\int \psi d \rho_R = 0$
v{\'e}rifie les assertions de la proposition, il suffit de v{\'e}rifier
$\sum_{n\ge 0} \| \Lambda^n \psi\|_{L^1(\rho_R)} < + \infty$.
Supposons alors que~$\psi$ est de classe~$\cC^1$ et appliquons la
Proposition~\ref{p:melange}, en utilisant l'hypoth{\`e}se $\int \psi
\, d\rho_R =0$.  On obtient
\begin{multline*}
\int |\Lambda^n\psi | \, d\rho_R
=
\sup_{\| \phi \|_{L^\infty}\le1} \int \Lambda^n\psi \times \phi\, d\rho_R
=
\sup_{\| \phi \|_{L^\infty}\le1} \int \psi \times (\phi \circ R^n) \, d\rho_R
\\ \le
C \times \langle  \psi, \psi \rangle^{1/2} \times \deg(R)^{-n/2}~.
\end{multline*}
On a donc bien la convergence d{\'e}sir{\'e}e.
\end{proof}

\begin{proof}[D{\'e}monstration de la Proposition~\ref{p:melange}]
Quitte {\`a} remplacer~$\phi$ par $\phi - \int \phi~d\rho_R$ on se ram{\`e}ne au cas o{\`u} $\int \phi \, d \rho_R = 0$, en utilisant $\| \phi - \int \phi \, d \rho_R\|_{L^\infty} \le 2 \| \phi \|_{L^\infty}$.
On~a donc les estimations:
\begin{multline*}
\left| \int (\phi \circ R^n) \, \psi\, d\rho_R - \int \phi \,
 d\rho_R\times \int \psi \, d \rho_R \right|
\\=
\left|\int (\phi \circ R^n)  \psi \, d (\deg(R)^{-n} R^{n*} \rho_R) \right|
\\ = 
\deg(R)^{-n} \left|\int R^n_* ((\phi \circ R^n) \psi) \, d \rho_R \right|
=
\deg(R)^{-n} \left|\int \phi R^n_*\psi \, d \rho_R \right|
\\ =
\deg(R)^{-n} \left|\int \psi \, d R^{n*}(\phi \rho_R) \right|
=
\deg(R)^{-n} \times \left( \Delta \psi, \, R^{n*}(\phi \rho_R) \right)
\\ \le 
\deg(R)^{-n}
\times
(\Delta \psi, \Delta \psi)^{1/2}
\times
(R^{n*}(\phi \rho_R) , \, R^{n*}(\phi \rho_R))^{1/2}~.
\end{multline*}
Consid{\'e}rons le potentiel $g : \HK \to \R$, d{\'e}fini par
$$
g(\cS)
=
-\int_{\PKber} \langle \cS , \cS' \rangle_\cScan \,
\phi \, d\rho_R (\cS') ~,
$$
de telle sorte que $\Delta g = \phi \rho_R$. 
On a alors~: 
\begin{multline*}
|(R^{n*}(\phi \rho_R) , \, R^{n*}(\phi \rho_R))|
=
\left| \int g \circ R^n \, d(R^{n*}(\phi \rho_R))\right|
=
\left| \int R^n_* (g \circ R^n) \phi \, d \rho_R \right|
\\ =
\deg(R)^n \times \left| \int g \phi \, d \rho_R \right|
\le
\deg(R)^{n} \times \|g \|_{L^\infty} \times \| \phi \|_{L^\infty}
\\ \le
\deg(R)^{n} \times \left(\int_{\PKber} \langle \cS, \cS' \rangle_{\cScan} d\rho_R(\cS') \right) \times \| \phi \|_{L^\infty}^2 ~.
\end{multline*}
Ceci termine la d{\'e}monstration de~\eqref{eq:mixing}.  Il reste {\`a}
montrer que la mesure~$\rho_R$ est m{\'e}\-langeante.  La preuve
ci-dessus montre que la convergence d{\'e}sir{\'e}e est v{\'e}rifi{\'e}e
lorsque~$\phi$ et~$\psi$ sont de classe $\cC^1$.  On remarque alors
que les fonctions de classe~$\cC^1$ forment une alg{\`e}bre contenant
les fonctions constantes. Cette alg{\`e}bre est dense dans l'ensemble
des fonctions continues (pour la topologie compacte) par le th{\'e}or{\`e}me
de Stone-Weierstrass.  On a donc convergence d{\`e}s que $\phi$ et
$\psi$ sont continues.  Enfin~$\rho_R$ est une mesure
r{\'e}guli{\`e}re dont le support est compact et m{\'e}trisable.  
Par cons{\'e}quent l'ensemble des fonctions
continues d{\'e}finies sur le support de~$\rho_R$ est dense dans
$L^1(\rho_R)$.  Ce qui conclut la preuve.
\end{proof}

\subsection{D{\'e}monstration du Th{\'e}or{\`e}me~\ref{t:equi-periodiques}}
\label{ss:equi-periodiques}
Pour chaque entier $n \ge 1$ on pose $D_n \= \deg(R)^n + \deg(S)$ et
on d{\'e}finit la fonction $A_n : \PKber \to [- \infty, 0]$ par
$$
A_n \= - \langle R^n, S \rangle_{\cScan} ~,
$$
de telle sorte que
$$
\exp A_n = \frac{ \sup \{ R^n, S \} }{ \max \{ 1, | R^n| \} \times \max \{ 1, |S| \} }
$$ sur $\PKber \setminus \{ R^n =\infty \text{ ou } S  = \infty \}$, et $\exp (A_n) \ge
\tfrac{1}{2} \dpk (R^n, S)$ avec {\'e}galit{\'e} sur~$\PK$.
De
$$ A_n(z) = \log|R^n(z) - S(z)| - \log \max \{ 1, |R^n(z)|\} - \log \max \{ 1, |S(z)|\}, $$
on tire
$$
\Delta A_n = [R^n = S] - R^{n*}[\cScan] - S^*[\cScan]~.
$$
Comme dans la preuve de la Proposition-D{\'e}finition~\ref{p:def} on choisit un potentiel~$g$ tel que
$$ \Delta g = (\deg(R))^{-1} R^*[\cScan]- [\cScan], $$
et pour chaque entier $n \ge 1$ on pose
$$ g_n = \sum_{k=0}^{n-1}\deg(R)^{-k} \, g\circ R^k. $$
D'autre part on choisit un potentiel~$h$ tel que
$$ \Delta h = (\deg(S))^{-1} S^*[\cScan]- [\cScan], $$
et pour chaque entier $n \ge 1$ on pose
$$ h_n \= D_n^{-1} (\deg(R^n) g_n + \deg(S) h) ~. $$
On v{\'e}rifie alors qu'on a
$$
\Delta (h_n + D_n^{-1} A_n) = D_n^{-1}[R^n = S] - [\cScan] ~.
$$
Comme dans la preuve de la Proposition-D{\'e}finition~\ref{p:def} on montre que les potentiels~$h_n$ convergent alors uniform\'ement sur~$\PKber$ vers
$$ g_\infty \= \sum_{k=0}^{\infty}\deg(R)^{-k} \, g\circ R^k ~. $$
En particulier, pour tout $n \ge 1$ on~a $h_n + D_n^{-1} A_n \in \cP_+$. Cette suite est de plus uniform\'ement major\'ee car $A_n\le0$. La Proposition~\ref{prop:hartogs} implique
donc qu'il existe une suite strictement croissante d'entiers positifs $\{ n_k
\}_{k \ge 1}$ telle que $D_{n_k}^{-1} A_{n_k}$ converge ponctuellement 
sur~$\HK$ vers une fonction continue $\psi : \HK
\to (- \infty, +\infty)$ qui est, soit la fonction
constante {\'e}gale {\`a}~$-\infty$, soit un potentiel tel que $g_{\infty} + \psi \in \cP_+$.  Dans les deux cas la fonction~$\psi$ s'{\'e}tend de
fa{\c c}on unique {\`a}~$\PKber$, en une fonction continue sur chaque
segment, et semi-continue sup{\'e}rieurement sur~$\PKber$.

Pour finir la d{\'e}monstration du th{\'e}or{\`e}me il suffit donc de montrer que~$\psi$ est constante.
Comme~$\psi$ est continue sur chaque segment de~$\PKber$ et comme $\psi \le 0$ (car $A_n \le 0$), il suffit de montrer~$\psi$ est localement constante sur l'ensemble ouvert $\{ \psi < 0 \}$.
Soit~$U$ un ouvert fondamental dont la fermeture topologique est contenue dans $\{ \psi < 0 \}$.
On montre tout d'abord que $S(\overline{U})$ est contenu dans l'ensemble de quasi-p{\'e}riodicit{\'e}~$\cE_R$ de~$R$.
Lorsque pour chaque~$\varepsilon > 0$ on applique le lemme de Hartogs (Proposition~\ref{prop:hartogs}) \`a $\varphi \equiv - \varepsilon$ et~$\cC \= \overline{U}$, on obtient que l'application~$R^{n_k}$ converge uniform{\'e}ment vers~$S$ sur~$\overline{U}$.
Ceci implique que pour tout~$k$ suffisament grand on a~$R^{n_k}(U) = S(U)$, et par cons{\'e}quent que~$R^{n_{k + 1} - n_k}$ converge vers l'identit{\'e} sur~$S(U)$.
On a donc $S(U) \subset \cE_R$, et en particulier~$\cE_R \neq \emptyset$.

Le Lemme~\ref{lem:non qp} implique alors que la caract{\'e}ristique r{\'e}siduelle de~$K$ est positive, {\'e}gale {\`a} un nombre premier~$p$.
Soit~$Y$ la composante connexe de~$\cE_R$ contenant~$S(U)$.
Apr{\`e}s changement de coordonn{\'e}e on suppose que $\infty \not \in Y$.
Soit~$n \ge 1$ l'entier et $T : \Z_p \times Y \to Y$ l'action donn{\'e}s par le Th{\'e}or{\`e}me~\ref{t:qp}.
Quitte {\`a} prendre une sous-suite on peut supposer que~$R^{n_1}(U) = S(U) \subset Y$, que pour chaque $k \ge 1$ l'entier~$n$ divise $n_k - n_1$, et que $a_k \= n^{-1} (n_k - n_1)$ converge dans~$\Z_p$ vers un certain $w_0 \in \Z_p$ lorsque $k \to + \infty$.
Comme~$R^{n_k}$ converge uniform{\'e}ment vers~$S$ sur~$U$, on conclut qu'on~a $S = T^{w_0} \circ R^{n _1}$ sur~$U$.
Le Th{\'e}or{\`e}me~\ref{t:qp} implique alors qu'on~a
$$ \frac{R^{a_k n + n_1} - S}{a_k - w_0}
= \left(\frac{T^{a_k - w_0} - T^0}{a_k - w_0}\right) \circ S \to T_* \circ S, $$
localement uniform{\'e}ment sur~$U$ lorsque $n \to + \infty$.
On a donc
\begin{multline*}
\psi
=
\lim_{k \to + \infty} D_{n_k}^{-1}A_{n_k}
=
\lim_{k \to + \infty} D_{n_k}^{-1} \left( \log |a_k - w_0| + \log|T_* \circ S| \right)
= \\
\lim_{k \to + \infty} D_{n_k}^{-1} \log |a_k - w_0|
\end{multline*}
ponctuellement sur~$U \cap \HK$.
Ceci montre que~$\psi$ est localement constante sur~$\{ \psi < 0 \}$, et termine la d{\'e}monstration.

\section{Entropie}\label{s:entropie}
%
%
\subsection{Entropie topologique dans les espaces compacts}\label{ss:generalites entropie}
L'espace~$\PKber$ est compact mais non m{\'e}\-trisable en
g{\'e}n{\'e}ral.  Les d\'efinitions et propri\'et\'es
de l'entropie souvent d\'ecrites dans le cadre restreint des espaces m\'etriques (voir~\cite{walters})
restent cependant valables dans notre cadre. Nous renvoyons \`a~\cite{mis} pour plus de d\'etails sur les preuves. Nous rappelons ici bri{\`e}vement
quelques d{\'e}\-fi\-ni\-tions.

Soit $X$ un espace compact et $f: X \to X$ une application continue.
{\'E}tant donn{\'e} un recouvrement
ouvert fini~$\fU$ de~$X$, on note~$N(\fU)$ le nombre minimal d'ouverts
de la famille n{\'e}cessaire pour recouvrir~$X$.  Si $\fV$ est un
autre recouvrement ouvert fini de~$X$, on note
$$ \fU \vee \fV \= \{ U \cap V, \, U \in \fU, \, V \in \fV\}. $$
C'est aussi un recouvrement ouvert fini de~$X$.
\emph{L'entropie topologique} de $f$ not\'ee $\htop(f)$ est alors d\'efinie comme suit:
$$
\htop(f) \= 
\sup_\fU \lim_{n\to \infty} \frac1n \log N \left( \bigvee_{i=0}^{n-1} f^{-i} \fU \right)~.
$$
Le supremum est ici pris sur tous les recouvrements ouverts finis de~$X$.

Si $\rho$ est une mesure de probabilit\'e $f$-invariante $f_*\rho = \rho$, on peut aussi d\'efinir \emph{l'entropie m\'etrique} de $f$ par rapport \`a $\rho$. L'entropie d'une partition $\fA$ finie de $X$ par des ensembles mesurables
est donn\'ee par $H(\fA) = -\sum_{A \in \fA} \rho(A) \log \rho(A)$.
L'entropie de $f$ not\'ee $h_\rho(f)$ est alors: 
$$
h_\rho(f) \=
\sup_\fA \lim_{n\to\infty} \frac1n H(\fA_n)
$$
o\`u $\fA_n = \fA \vee f^{-1} \fA \vee \cdots \vee f^{-n} \fA$ et le supremum est pris sur toutes les partitions de $X$ 
finies par des ensembles mesurables.

On rappelle maintenant le principe variationnel.  
\begin{Prop}[Principe Variationnel]\label{prop:principe variationnel}
Soit $X$ un espace topologique compact et soit $f : X \to X$ une
application continue.  Alors on a $\htop(f) = \sup_\rho h_\rho(f)$,
o{\`u} le supremum est pris sur l'ensemble de toutes les mesures de
probabilit{\'e} invariantes par~$f$.
\end{Prop}

Rappelons de plus, que l'ensemble des mesures de probabilit\'e invariantes est un convexe compact pour 
la convergence vague dont les points extr\'emaux sont exactement les mesures ergodiques. On montre alors que toute mesure invariante~$\rho$ est \og{} moyenne \fg{} de mesures ergodiques (c'est la d\'ecomposition de Choquet). En particulier, $\htop(f)$ est aussi \'egal au supr\'emum $h_\rho(f)$ pris sur l'ensemble des mesure ergodiques. Nous renvoyons \`a~\cite{walters} pour plus de pr\'ecisions.

%
%
\subsection{D\'emonstration du Th{\'e}or{\`e}me~\ref{t:estimation htop}}\label{ss:estimation htop}
La d{\'e}monstration du Th{\'e}or{\`e}me~\ref{t:estimation htop} s'appuie sur
deux lemmes.
On commence avec la propri{\'e}t{\'e} g{\'e}n{\'e}rale suivante.
\begin{Lem}\label{lem:recouvrement minimal}
Soit~$X$ un espace topologique s{\'e}par{\'e} et connexe, et soit~$\fU$ un recouvrement fini de~$X$ par des ensembles ouverts et connexes, tel que l'ensemble
$$ \partial \fU \= \bigcup_{U \in \fU} \partial U $$
soit fini et non vide.
Alors on a $N(\fU) \le \# \partial \fU$.
\end{Lem}
\begin{proof}
  Pour chaque $x \in \partial \fU$ on choisit un {\'e}l{\'e}ment~$U_x$ de~$\fU$ contenant~$x$.
Il suffit de montrer que pour chaque $x' \in X$ il existe $x \in X$ tel que~$U_x$ contient~$x'$.
Lorsque $x' \in \partial \fU$ il n'y a rien a montrer.
On se ram{\`e}ne alors au cas o{\`u} $x' \not \in \partial \fU$.
Notons~$Y$ la composante connexe de $X \setminus \partial \fU$ contenant~$x'$.
Comme l'ensemble $\partial \fU$ est fini et donc ferm{\'e}, on a $\partial Y \subset \partial \fU$.
Du fait que~$X$ est connexe et~$\partial U$ est non vide, on conclut que l'ensemble~$\partial Y$ est non vide.
Soit $x \in \partial Y$ et notons que~$U_x$ rencontre~$Y$.
Comme~$Y$ est connexe et comme~$\partial U_x$ est disjoint de~$Y$, on a $x' \in Y \subset U_x$.
\end{proof}

\begin{Lem}\label{lem:entropie topologique}
Soit~$R$ une fraction rationnelle de degr{\'e} au moins deux {\`a} coefficients dans~$K$, et soit~$X$ une partie compacte et connexe de~$\PKber$ invariante par~$R$.
Alors on~a
$$ \htop(R|_X) \le \log \sup \{ \# (R|_X)^{-1}(\cS) \mid \cS \in X \}. $$
\end{Lem}
\begin{proof}
On pose,
$$ D \= \sup \{ \# (R|_X)^{-1}(\cS) \mid \cS \in X \}. $$
Au vu du Lemme~\ref{lem:recouvrement minimal} il suffit de montrer que pour tout recouvrement fini~$\fU$ de~$X$ on~a
$$ \limsup_{n \to + \infty} \frac{1}{n} \log \# \partial \left( \bigvee_{j = 0}^{n - 1} (R|_X)^{-j} \fU \right)
\le
\log D.
$$
Soit~$\fU$ un tel recouvrement.
Comme les ouverts fondamentaux forment un base de la topologie de~$\PKber$, on se ram{\`e}ne au cas o{\`u} chaque {\'e}l{\'e}ment de~$\fU$ est l'intersection d'un ouvert fondamentaux de~$\PKber$ avec~$X$.
Alors l'ensemble~$\partial\fU$ est fini.
S'il est vide, alors on a~$X \in \fU$ et alors l'assertion est imm{\'e}diate dans ce cas.
On se ram{\`e}ne alors au cas o{\`u} l'ensemble~$\partial \fU$ n'est pas vide.
L'inclusion
$$ \partial \left( \bigvee_{j = 0}^{n - 1} (R|_X)^{-j} \fU \right)
\subset
\bigcup_{j = 0}^{n - 1} (R|_X)^{-j} \partial \fU, $$
implique qu'on a
$$ \partial \left( \bigvee_{j = 0}^{n - 1} (R|_X)^{-j} \fU \right)
\le
\# \partial \fU \times (1 + D + \cdots + D^{n - 1})
\le
\# \partial \fU \times n D^{n - 1}, $$
d'o{\`u} l'on d{\'e}duit l'in{\'e}galit{\'e} d{\'e}sir{\'e}e.
\end{proof}

\begin{proof}[D\'emonstration du Th{\'e}or{\`e}me~\ref{t:estimation htop}]
L'in{\'e}galit{\'e} $h_{\rho_R}(R) \le \htop(R)$ est cons{\'e}\-quence du principe variationnel (ou Proposi\-tion~\ref{prop:principe variationnel}).
L'in{\'e}galit{\'e} $\htop(R) \le \log \degtop(R)$ est donn{\'e}e par le
Lem\-me~\ref{lem:entropie topologique} avec $X = \PKber$.

Il reste {\`a} montrer l'{\'e}galit{\'e} $\htop(R) = \htop(R|_{J_R})$.
Au vu du principe variationnel, il suffit de montrer que pour toute mesure de probabilit\'e~$\rho$ invariante par~$R$ et ne chargeant pas $J_R$, on a~$h_\rho(R) = 0$.
On se ram{\`e}me au cas o{\`u}~$\rho$ est ergodique.
Le th{\'e}or{\`e}me de r{\'e}currence de Poincar{\'e} montre le support topologique de~$\rho$ est disjoint de l'ouvert des points errants.
En particulier,~$\rho$ ne charge pas les composantes de Fatou errantes, ou strictement pr{\'e}p{\'e}riodiques.
Comme~$\rho$ est de plus ergodique, on peut trouver une composante connexe~$U$ de~$F_R$ p{\'e}riodique et un entier~$N \ge 1$ tel que $R^N(U) = U$ et $\rho \left( \bigcup_{k = 0}^{N - 1} R^k(U)\right) = 1$.
  
Si~$U$ est le bassin d'attraction imm{\'e}diat d'un point p{\'e}riodique
attractif, alors~$\rho$ charge ce point et est donc {\`a} support fini par ergodicit\'e. Son entropie
est donc nulle.
Sinon, la Proposition~\ref{p:classification} implique que~$R^N$ est injective sur~$U$, et alors~$R^N$ est injective sur~$X \= \overline{U}$.
Comme~$X$ est un sous-ensemble compact et connexe de~$\PKber$ invariant par~$R^N$, le Lemme~\ref{lem:entropie topologique} appliqu{\'e} {\`a}~$R^n$ au lieu de~$R$, implique que~$\htop(R^n|_X) = 0$.
Le principe variationnel nous donne alors~$h_{\rho}(R) = \frac{1}{N} h_{\rho|_U}(R^N) = 0$.
\end{proof}
%
%
\subsection{Preuve des Th{\'e}or{\`e}mes~\ref{t:estimation h metrique} 
et~\ref{t:htopzero}}\label{ss:estimation h metrique} La preuve du
Th{\'e}or{\`e}me~\ref{t:estimation h metrique} s'appuie sur le lemme
suivant.  Rappelons que pour une mesure de probabilit\'e~$\rho$ sur~$\PKber$ le
\textit{Jacobien} $\Jac_{\rho} : \PKber \to [0, + \infty]$ de~$R$
pour~$\rho$ est la fonction mesurable caract{\'e}ris{\'e}e, en dehors d'un ensemble de mesure zero, par la propri{\'e}t{\'e}
suivante.  Pour tout ensemble bor{\'e}lien~$E$ sur lequel~$R$ est
injective, on a $\rho(R(E)) = \int_E \Jac_{\rho}\, d\rho$. Une telle fonction
existe toujours dans notre cas.
Nous renvoyons \`a~\cite[\S 1.9]{PU} ou~\cite{parry} pour plus de d\'etails.

\begin{Lem}\label{L-calcul-jac}
\

\begin{enumerate}
\item[1.]  Soit~$E$ un sous ensemble bor{\'e}lien de~$\PKber$ et
soit~$\mathbf{1}_E$ la fonction caract{\'e}ris\-tique de~$E$.  Alors on
a,
\begin{equation}\label{e:star}
\int R_* \mathbf{1}_E \, d\rho_R
=
\deg(R) \int \mathbf{1}_E\, d\rho_R~.
\end{equation}
\item[2.]  On a $$\Jac_{\rho_R}(\cS) = \deg(R)/\deg_R(\cS).$$
\end{enumerate}
\end{Lem}
\begin{proof}[D\'emonstration du Lemme~\ref{L-calcul-jac}]
\

\noindent\textbf{1.}
Nous remercions le rapporteur de nous avoir fourni la preuve suivante.
Notons d'abord que l'{\'e}quation $R^*
\rho_R = \deg(R) \, \rho_R$ se traduit par l'{\'e}galit{\'e}
\begin{equation}\label{e:star lisse}
\int R_*\phi\, d\rho_R
=
\deg(R) \int \phi\, d\rho_R
\end{equation}
pour toute fonction continue~$\phi$.  Par convergence domin\'ee, on a~\eqref{e:star} pour tout bor\'elien~$E$
tel qu'il existe une suite de fonctions continues $\varphi_n : \PKber \to [0,1]$ telle que $\varphi_n \to \mathbf{1}_E$ ponctuellement.
Soit~$\cF$ la collection de tous ces ensembles bor{\'e}liens.
Il est clair que $\cF$ est stable par intersection finie (prendre le produit des approximants continus), et par compl\'ementaire (prendre $1-\varphi_n$).
Il est aussi stable par intersection d\'enombrable par convergence monotone.
On conclut en remarquant que~$\cF$ contient toutes les boules ferm{\'e}es de~$\PKber$.
\smallskip

\noindent\textbf{2.}
Soit  $E\subset \PKber$ un ensemble bor{\'e}lien tel que $R: E \to
R(E)$ soit bijective.
Pour chaque $k = 1, \ldots, \deg(R)$, notons $E_k = \{ \cS \in E, \, \deg_R (\cS) =k \}$.
De~\eqref{e:star}, on tire
$$
\deg(R) \times \rho_R(E_k)
=
\int R_* \mathbf{1}_{E_k}\, d\rho_R
= k \times \rho_R (R(E_k))~,
$$
d'o{\`u},
$$
\rho_R( R(E))
=
\sum_{k=1}^{\deg(R)} \rho_R (R(E_k))
= 
\int_E \frac{\deg(R)}{\deg_R(\cS)}\, d\rho_R(\cS)~.$$
\end{proof}

\begin{proof}[D{\'e}monstration du Th{\'e}or{\`e}me~\ref{t:estimation h metrique}]
  La preuve de l'estimation~\eqref{e:hmetr} est une cons{\'e}quence de la
  formule de Rokhlin~\cite[\S 1.9]{PU} ou~\cite[\S10]{parry} 
qui s'{\'e}nonce sous la forme suivante:
\begin{equation}\label{e-rohlin}
 h_{\rho_R}(R) \ge \int \log|\Jac_{\rho_R}(\cS)|\, d\rho_R(\cS)~.
\end{equation}
Cette formule s'applique dans notre contexte car~$R$ est {\`a} fibres
finies.  Il est alors clair que l'{\'e}quation~\eqref{e:hmetr}
d{\'e}coule de~\eqref{e-rohlin}, de la partie~2 du
Lemme~\ref{L-calcul-jac}, et de la d{\'e}finition du degr{\'e} moyen.

Supposons maintenant que $\rho_R$ ne charge pas $\HK$.  Comme le
degr{\'e} local de~$R$ est constant sur~$\PK$ {\'e}gal
{\`a}~$\deg(R)/\degtop(R)$, hors d'un ensemble fini, la
Proposition~\ref{p:holder} implique $\mdeg = \deg(R) / \degtop(R)$.
Les {\'e}galit{\'e}s d{\'e}sir{\'e}es d{\'e}coulent alors de l'estimation
$h_{\rho_R}(R) \ge \log \frac{\deg(R)}{\mdeg(R)}$.
\end{proof}

\begin{proof}[D{\'e}monstration du Th{\'e}or{\`e}me~\ref{t:htopzero}]
On montre tout  d'abord l'{\'e}quivalence de~(1), (2), (3), et~(4).
L'implication (1)$\Rightarrow$(2) est une cons{\'e}quence imm{\'e}diate du
principe variationnel (Proposition~\ref{prop:principe variationnel}).
Pour montrer l'implication~(2)$\Rightarrow$(3) supposons que
$\phi\circ R \circ \phi^{-1}$ n'ait bonne r{\'e}duction pour
aucune application de M{\"o}bius $\phi$.  Quitte {\`a} remplacer~$R$ par
un it{\'e}r{\'e} assez grand on se ram{\`e}ne au cas o{\`u} pour tout $\cS
\in J_R$ on a $\deg_R(\cS) \le \deg(R) - 1$, voir
Lemme~\ref{L-estim-deg}. Comme $\rho_R$ est une mesure
support{\'e}e dans l'ensemble de Julia, on a
$$
\log \mdeg(R)
=
\int \log \deg_R(\cS)\, d\rho_R(\cS) \le \log (\deg(R) - 1)
<
\log \deg(R)~,
$$ et le Th{\'e}or{\`e}me~\ref{t:estimation h metrique} implique alors
$h_{\rho_R}(R)>0$.  Pour montrer l'implication (3)$\Rightarrow$(4),
notons d'abord que si~$R$ a bonne r{\'e}duction, alors le point
$\cScan\in\HK$ associ{\'e} {\`a} la boule unit{\'e} est totalement
invariant.  On a donc $R^*[\cScan] = \deg(R) [\cScan]$, et par
d{\'e}finition de la mesure d'{\'e}quilibre, on a $\rho_R \= \lim_{n \to
+ \infty} \deg(R)^{-n} R^{n*} [\cScan] = [\cScan]$.  Notons finalement
que l'implication (4)$\Rightarrow$(1) est imm{\'e}diate.

Comme l'implication (4)$\Rightarrow$(5) est imm{\'e}diate, pour
compl{\'e}ter la preuve du th{\'e}or{\`e}me il suffit de montrer
l'implication (5)$\Rightarrow$(3).  Supposons qu'il existe $\cS
\in \PKber$ tel que~$\rho_R\{\cS\} >0$.  La Proposition~\ref{p:holder}
implique $\cS \in \HK$.  On tire de l'{\'e}quation
d'invariance~\eqref{eq-inv2} que pour tout point~$\cS'$ dans la grande
orbite
$$ \mathcal{G} \= \bigcup_{n \ge 0, m \ge 0} R^{-m} (R^n\{\cS\}) $$
de~$\cS$, on a $\rho_R \{\cS'\} >0$.
La masse de $\rho_R$ \'etant finie, quitte {\`a} remplacer $\cS$ par
un autre {\'e}l{\'e}ment de $\mathcal{G}$, on peut supposer $\rho_R\{\cS\}
\ge \rho_R\{\cS'\}$ pour tout point $\cS' \in\mathcal{G}$.  

On a alors
$$ \rho_R \{ \cS \} = \frac{\deg_R(\cS)}{\deg(R)}\, \rho_R\{R(\cS)\}\le
\rho_R\{ R(\cS)\} \le 
\rho_R\{\cS\}~,
$$ d'o{\`u} $\rho_R(\cS) = \rho_R(R(\cS))$, $\deg_R(\cS) = \deg(R)$,
et
$R^{-1}(R(\cS)) = \{ \cS \}$.  Par le m\^eme argument $\rho_R(R^n(\cS))= \rho_R(\cS)$ pour tout
$n\ge0$. La masse de $\rho_R$ {\'e}tant finie, on
en d{\'e}duit que la grande orbite~$\mathcal{G}$ est finie.  La
Proposition~\ref{prop:bonne} montre alors que~$R$ est conjugu{\'e}e
{\`a} une fraction rationnelle ayant bonne r{\'e}duction.
\end{proof}

%
%
%
%

\section{Exemples}\label{sec:exemple}
Nous donnons dans cette section quelques exemples significatifs afin
d'{\'e}clairer les diff{\'e}rences entre les contextes complexe et
non archim{\'e}dien.

\subsection{Exemples de Latt{\`e}s}
Les exemples que nous allons d{\'e}crire
s'inspirent directement de leurs analogues complexes
(voir~\cite{milnor} pour une discussion d{\'e}taill{\'e}e dans ce
cadre).  Afin de simplifier la discussion nous supposerons que 
la caract{\'e}ristique de $K$ est diff{\'e}rente de $2$ bien que ce ne soit pas
strictement n{\'e}cessaire. Rappelons 
de plus que~$K$ est toujours suppos{\'e}
alg{\'e}briquement clos et complet.

Soit $q\in K$ satisfaisant $0 < |q| <1$. Le quotient $K^* / q^{\Z}$ du groupe multiplicatif $K^*$  
est une courbe elliptique
$E_q$ dont le $j$-invariant v{\'e}rifie $|j(E_q)|> 1$.
R{\'e}ciproquement, toute courbe elliptique dont le $j$-invariant est de module strictement plus grand que~1 peut {\^e}tre obtenue ainsi, voir~\cite[II, \S6, Theorem~6.1]{Sil94} ou~\cite[\S 6]{roquette}.
Pour chaque entier $m \in \Z$, l'endomorphisme $z \mapsto z^m$ sur $K^*$ passe au quotient et induit 
l'endomorphisme $[m]$ de multiplication par $m$ sur $E_q$.

Le quotient de $E_q$ par l'involution $[-1]$ est isomorphe {\`a} $\PK$. Par exemple si $E_q$ est donn{\'e} par une {\'e}quation de Weierstrass dans le plan $y^2 = x^3 + a_4 x + a_6$, alors $[-1]$ est le morphisme $(x,y) \to (x,-y)$ et l'isomorphisme $E_q/[-1] \to\PK$ est induit par la projection $(x,y) \mapsto x$.
Comme $[m]$ commute avec $[-1]$, l'endomorphisme $[m]$ d{\'e}finit une fraction rationnelle $L_{m,q}$ sur $\PK$ de type Latt{\`e}s, voir Remarque~\ref{rem:lattes}.

Pour tout $t$ r{\'e}el, notons $\cS(t)$ le point de $\HK$ associ{\'e} {\`a} la boule de centre $0$ et de rayon $\exp(t)$.
\begin{Prop}\label{p:lattes}
La fraction rationnelle $L_{m,q}$ est de degr{\'e} $m^2$ et
son ensemble de Julia $J_{m,q}$ est un segment ferm{\'e}  inclus dans $\HK$ et  non r{\'e}duit {\`a} un point.
De plus, si pour tout $i = 0, \ldots, 2m - 1$ et $t \in [-\frac{i}{2m}\log|q| , -\frac{i+1}{2m}\log |q|]$ on pose
$$
L(t) = 
\begin{cases}
m t + \frac{i}{2} \log|q| & \text{ si } i \text{ est pair} \\
-mt - \frac{i+1}{2} \log |q| & \text{ si } i \text{ est impair}
\end{cases}
$$
dans des coordonn{\'e}es convenables, on a
$$
J_{m,q} = \{ \cS(t), t\in [ 0, -2^{-1}\log|q|]\}
$$
et $L_{m,q} \circ \cS = \cS \circ L$ sur $[ 0, -2^{-1}\log|q|]$.
En particulier, $\htop(L_{m,q})  = \log m$.
\end{Prop}

\begin{Rem}\label{rem:lattes}
Une fraction rationnelle $R\in K(z)$ est dite de Latt{\`e}s s'il existe une courbe elliptique $E$, un endomorphisme $g: E \to E$ et un morphisme fini $\pi: E \to \PK$ tel que $ R \circ \pi = \pi \circ g$. Lorsque $|j(E)| \le 1$, alors $R$ est conjugu\'ee \`a une fraction rationnelle ayant bonne r{\'e}duction. Sur une courbe pour laquelle $\mathrm{End}(E) \not\simeq \Z$, on peut aussi prendre pour $g$ un endomorphisme de degr{\'e} plus grand ou {\'e}gal {\`a}~2 qui n'est pas la multiplication par un entier. Dans ce cas l'invariant $j(E)$ est un entier alg{\'e}brique~\cite[\S 6]{roquette} et par cons{\'e}quent les exemples obtenus ont bonne r{\'e}duction.
\end{Rem}

\begin{Rem}
Dans le cas complexe, une application de Latt{\`e}s est caract{\'e}ris{\'e}e par le fait que sa mesure
d'{\'e}quilibre est absolument continue par rapport {\`a} la mesure de Lebesgue, voir~\cite{Zdu90,May02} et aussi~\cite{BerDup05} pour une g{\'e}n{\'e}ralisation aux dimensions sup{\'e}rieures.
Ce type de r{\'e}sultat n'est pas transposable tel quel dans le cadre non archim{\'e}dien.
En effet une petite perturbation de $L_{m,q}$  ne change pas l'action sur un compact fix{\'e} de $(\HK,\dhyp)$, et par suite toute fraction de degr{\'e} $m^2$ suffisament proche de $L_{m,q}$
poss{\`e}de le m{\^e}me ensemble de Julia et la m{\^e}me action sur cet ensemble.
\end{Rem}

\begin{Rem}
La loi d'addition sur une courbe elliptique est explicite. On peut ainsi obtenir facilement une formule pour les fractions $L_{m,q}$ avec $m$ petit. On trouve par exemple $L_{2,q}(T)= \frac{(T^2-\la)^2}{4T (T-1)(T-\la)}$, pour un $\la \in \C \setminus \{ 0,1 \}$ d\'ependant de $q$. Nous renvoyons
par exemple \`a~\cite[\S 5]{milnor2}.
\end{Rem}

\begin{proof}[D\'emonstration de la Proposition~\ref{p:lattes}]
Notons $\mathsf{A}^* = \PKber\setminus \{ 0, \infty\}$. C'est aussi l'espace analytique de Berkovich associ{\'e} au groupe multiplicatif $K^*$. Notons $\mathcal{L}\subset \mathsf{A}^*$ le segment ouvert joignant $0$ {\`a} $\infty$ dans $\PKber$.  Soit $G$ le sous-groupe d'automorphismes de $\mathsf{A}^*$ engendr{\'e} par les morphismes $z \mapsto qz$ et  $z \mapsto z^{-1}$.  Le quotient $\mathsf{A}^*/G$ est isomorphe {\`a} $\PKber$. Par ailleurs,~$G$  pr{\'e}serve $\mathcal{L}$, et le quotient topologique de~$\mathcal{L}$ par~$G$ est un segment $J \subset \HK$. On remarque maintenant que~$\mathcal{L}$ est totalement invariant par le morphisme $z \mapsto z^m$, donc~$J$ est aussi totalement invariant par~$L_{m,q}$. Il est facile de voir que~$J$ ne contient pas de sous-ensemble compact strict totalement invariant, donc~$J$ est l'ensemble de Julia de~$L_{m,q}$ par la Proposition~\ref{prop:generalites Julia}.
La formule explicite de l'action de~$L_{m,q}$ sur son ensemble de Julia r{\'e}sulte de l'action de $z \mapsto z^m$ sur~$\mathcal{L}$ par passage au quotient.
\end{proof}


\subsection{La mesure d'{\'e}quilibre n'est pas d'entropie maximale}\label{ss:contreexample}
Dans cette section on donne quelques exemples robustes de fractions rationnelles dont la mesure d'{\'e}quilibre n'est pas d'entropie maximale.

{\'E}tant donn{\'e} un segment~$I$ de~$\HK$ et une partie~$D$ de~$I$, on dira
qu'une application continue $T : D \to I$ est \textit{affine par
  morceaux} s'il existe~$k \ge 1$ et des segments ferm{\'e}s~$I_1, \ldots,
I_k$ de~$\HK$, dont la r{\'e}union est {\'e}gale {\`a}~$D$, et tels que pour chaque
$j = 1, \ldots, k$ l'application~$T : I_j \to I$ soit affine par
rapport {\`a} la distance~$\dhyp$.  Si de plus $D \subset I$, $k \ge 2$,
et pour chaque $j = 1, \ldots, k$ l'application~$T : I_j \to I$ est
bijective, alors on dira que~$T$ est \textit{Bernoulli} (voir Figure~\ref{f:contreexemple connexe} pour un exemple).  
Dans ce
cas l'entier~$k$ et les segments~$I_1, \ldots, I_k$ sont uniquement
d{\'e}termin{\'e}s par~$T$.  On appelle~$k$ le \textit{degr{\'e} topologique
  de~$T$}.

Soit~$T : D \to I$ une application affine par morceaux et Bernoulli,
$k$ un entier et~$I_1, \ldots, I_k$ comme ci-dessus.  L'ensemble invariant
maximal de~$T$, que l'on noter
$$ J_T \= \bigcap_{n \ge 0} T^{-n}(I), $$
est alors l'unique partie compacte de~$I$ compl{\`e}tement invariante par~$T$.
Notons de plus que l'entropie topologique de~$T|_{J_T}$ est {\'e}gale
{\`a}~$\log k$, et que les segments~$I_1, \ldots, I_k$ forment une partition
g{\'e}n{\'e}ratrice finie\footnote{i.e. Pour chaque suite~$\{ \sigma(j) \}_{j \ge 1}$ d'{\'e}l{\'e}ments de $\{ 1, \ldots, k \}$ l'ensemble $\bigcap_{j\ge 0} T^{-j} I_{\sigma(j)}$ est r{\'e}duit~{\`a} un point} pour~$T$.
Par cons{\'e}quent la formule de Rokhlin s'applique, et donne que pour toute mesure~$\rho$ invariante par~$T$, on a
$$h_\rho(T|_{J_T}) = \int \Jac_\rho \, d\rho~.$$
Nous renvoyons \`a nouveau \`a~\cite[\S 1.9]{PU} ou~\cite[\S 10]{parry} pour plus de pr\'ecisions.
\begin{Lem}\label{l:Bernoulli}
  Soit $R$ une fraction rationnelle {\`a} coefficients dans~$K$.
  Supposons qu'il existe un segment~$I$ de~$\HK$ et une partie~$D$
  de~$I$ telle que $R : D \to I$ soit une application affine par
  morceaux et Bernoulli.  Supposons de plus que, si l'on note
  par~$k$ le degr{\'e} topologique de cette application, alors les pentes
  correspondantes $d_1, \ldots, d_k$ satisfont~$\sum_{j = 1}^k d_j =
  \deg(R)$.  Alors l'ensemble de Julia de~$R$ est {\'e}gal {\`a} l'ensemble
  invariant maximal de~$R : D \to I$, l'entropie topologique de~$R$
  est {\'e}gale {\`a} $\log k$, et on~a
$$
h_{\rho_R}(R)
=
\sum_{j = 1}^k \frac{d_j}{\deg(R)} \log \left( \frac{d_j}{\deg(R)} \right)
\le
\htop(R)~,
$$
avec {\'e}galit{\'e} si et seulement si~$d_1 = \cdots = d_k$.
\end{Lem}

La d{\'e}monstration de ce lemme est ci-dessous.  {\'E}tant donn{\'e} un entier~$d
\ge 5$ et $a \in K$ satisfaisant~$|a| \in (0, 1)$, soit $R_0(z) =
\frac{z^{d - 2}}{1 + (az)^d}$. 
Montrons
l'in{\'e}galit{\'e}~\eqref{e:contreexample}.  
Soit~$\cS(t)$ le point dans~$\HK$ associ{\'e} {\`a} la boule $\{z \in \PK, \,
\log |z| \le t \}$.  On v{\'e}rifie alors que pour chaque $t \in \R$ on~a
$$
R_0(\cS(t))
=
\begin{cases}
\cS((d - 2) t) & \text{si } t \le - \log|a| ~;
\\
\cS( - d \log|a| - 2t) & \text{si } t \ge - \log|a| ~.
\end{cases}
$$
Si l'on pose $I \= \{ \cS(t), \, t \in [0,  - \tfrac{d}{2} \log |a|] \}$, et,
$$
I_1
\=
\{ \cS(t), \, t \in [0, - \tfrac{d}{2(d - 2)} \log |a|] \} ~, \,
I_2
\=
\{ \cS(t), \, t \in [- \tfrac{d}{4} \log |a|, - \tfrac{d}{2} \log |a|] \} ~,
$$
alors les applications $R_0 : I_1 \to I$ et $R_0 : I_2 \to I$ sont
bijectives et affines de pente~$d - 2$ et $2$, respectivement.
L'application~$R_0 : I_1 \cup I_2 \to I$ est donc affine par morceaux
et Bernoulli, et le lemme implique alors~\eqref{e:contreexample}.
Notons de plus que, comme l'ensemble~$I_1 \cup I_2$ est born{\'e}
dans~$\HK$ par rapport {\`a} la m{\'e}trique~$\dhyp$, il existe un
voisinage~$\cU_d$ de~$R_0$ dans l'espace des fractions rationnelles de
degr{\'e}~$d$, tel que toute~$R \in \cU_d$ co{\"\i}ncide avec~$R_0$ sur~$I_1
\cup I_2$.  Le lemme s'applique alors {\`a} chaque {\'e}l{\'e}ment~$R$ de~$\cU_d$,
et on obtient~\eqref{e:contreexample} lorsqu'on remplace~$R_0$
par~$R$.

Plus g{\'e}n{\'e}ralement, soient $k \ge 2$ et $d_1, \ldots , d_k >1$ des
entiers, et soient $a_2, \ldots, a_k \in K^*$ tels que $|a_k| > \cdots > |a_2| > 0$.  On pose $\de_1 = d_1$, $t_{k + 1} = + \infty \in
\overline{\R}$, et pour chaque $j \in \{ 2, \ldots, k \}$, 
$\de_j = d_j + d_{j-1}$ et ~$t_j = - \ln |a_j|$.  Finalement, on d{\'e}signe
par~$T : \R \to \R$ l'unique application continue {\'e}gale {\`a} $t \mapsto
d_1 t$ sur $(- \infty, t_2)$, et pour chaque $j \in \{ 2, \ldots, k
\}$, affine de pente~$(-1)^{j} d_j$ sur~$[t_j, t_{j + 1}]$.  Alors on
v{\'e}rifie que la fraction rationnelle
\begin{equation}\label{e:general}
R(z) \= z^{d_1} \prod_{j=2}^k \left( 1 + ( a_jz)^{\de_j}\right)^{(-1)^j}~,
\end{equation}
est de degr{\'e} $\sum_{j = 1}^k d_j$, et que pour tout~$t \in \R$ on a
$R(\cS(t)) = \cS(T(t))$.  Il est alors facile de voir que
lorsque~$\sum_{j = 1}^{k} d_j^{-1} \le 1$, on peut choisir $a_2,
\ldots, a_k$ de telle sorte qu'il existe un intervalle compact~$J$
de~$\R$, pour lequel l'application $R : \{ \cS(t), \, T(t) \in J \}
\to \{ \cS(t), \, t \in J \}$ soit Bernoulli de degr{\'e}~$k$, avec
pentes~$d_1, \ldots, d_k$.  Le lemme implique alors que l'ensemble de
Julia de~$R$ est {\'e}gal {\`a} l'image par l'application~$\cS : \R \to (0,
\infty) \subset \HK$ de l'ensemble invariant maximal de~$T$ dans~$J$.
Notons en particulier que l'ensemble de Julia de~$R$ est contenu dans
un segment compact de~$\HK$, et qu'il est un segment (resp. un
ensemble de Cantor) lorsque $\sum_{j = 1}^{k} d_j^{-1} = 1$ (resp.
$\sum_{j = 1}^{k} d_j^{-1} < 1$).  Notons d'autre part, que lorsque
les entiers $d_1, \ldots, d_k$ ne sont pas tous {\'e}gaux, le degr{\'e} de la
fraction rationnelle~\eqref{e:general} est au moins~$5$.  Si de plus
l'ensemble de Julia de cette fraction rationnelle est un intervalle,
c'est-{\`a}-dire lorsque $\sum_{j = 1}^k {d_j^{-1}} = 1$, alors le degr{\'e}
est au moins {\'e}gal {\`a}~10.

Concr{\`e}tement, notons que la fraction rationnelle~$R_0$ d{\'e}finie
ci-dessus, est de de\-gr{\'e}~$d \ge 5$, et qu'elle est donn{\'e}e
par~\eqref{e:general} lorsque~$k = 2$, $d_1 = d - 2$, $d_2 = 2$ et
$a_2 = a$.  D'autre part, si $k = 3$, $d_1 = 2$, $d_2 = 4$, $d_3 = 4$
et $a_2, a_3 \in K$ v{\'e}rifient $0 < |a_2| = |a_3|^{2/3} < 1$, alors la
fraction rationnelle
\begin{equation}\label{e:exR}
R_1(z) = z^2 \frac{1 + (a_3z)^8}{1 + (a_2 z)^6}~,
\end{equation}
est de degr{\'e}~10, et si l'on note $I \= \{ \cS(t), \, t \in [0, - 2
\log |a_2|] \}$, alors l'application~$R : I \to I$ est affine par
morceaux, Bernoulli, et le lemme implique que $J_{R_1} = I$
et~$h_{\rho_{R_1}}(R_1) < \htop(R_1)$ (voir Figure~\ref{f:contreexemple connexe}).  
\begin{figure}[ht]\label{f:contreexemple connexe}
\centering \input{affine2.pstex_t}
\caption{Action de $R_1$ avec $|a_2| = |a_3|^{2/3}=e^{-2}$}
\end{figure}
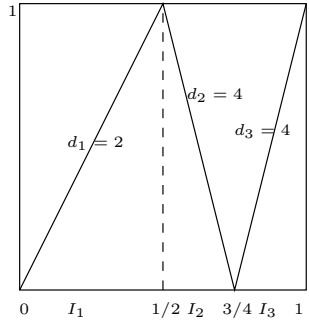

\begin{qst}
  Les exemples pr{\'e}c{\'e}dents motivent les questions suivantes.
  Existe-t-il un polyn{\^o}me (resp. une fraction rationnelle de degr{\'e} au plus
  4, ou une fraction rationnelle de degr{\'e} au plus~9 et d'ensemble de Julia
  connexe) dont la mesure d'{\'e}quilibre ne soit pas d'entropie
  maximale~?
\end{qst}

\begin{proof}[D{\'e}monstration du Lemme~\ref{l:Bernoulli}]
On peut supposer que $k\ge 2$. Sinon $J$ est r\'eduit \`a un point
et $R$ a bonne r\'eduction dans une coordonn\'ee ad\'equate.
  Soient~$I_1, \ldots, I_k$ les segments dans~$I$ tels que pour cha\-que
  $j = 1, \ldots, k$ l'application $R : I_j \to I$ soit bijective et
  affine de pente~$d_j$.  Alors le degr{\'e} local de~$R$ en tout point
  dans l'int{\'e}rieur de~$I_j$ est au moins~$d_j$.  Comme par hypoth{\`e}se
  on a~$\sum_{j = 1}^k d_j = \deg(R)$, on en d{\'e}duit que pour chaque~$j =
  1, \ldots, k$ le degr{\'e} local de~$R$ en chaque point de l'int{\'e}rieur
  de~$I_j$ est {\'e}gal {\`a}~$d_j$, et que la pr{\'e}image par~$R$ de
  l'int{\'e}rieur de~$I$ est contenue dans~$I$.  Comme~$R$ est continue,
  on a~$R^{-1}(I) \subset I$.  La Proposition~\ref{prop:generalites
    Julia} implique alors que l'ensemble de Julia de~$R$ est contenu
  dans l'ensemble invariant maximal~$J_{R|_D}$ de~$R : D \to I$.
  Comme ce dernier ensemble est le plus petit sous-ensemble compact
  de~$I$ compl{\`e}tement invariant par~$R|_D$, on a
  $J_R = J_{R|_D}$.  Par le Th{\'e}or{\`e}me~\ref{t:estimation htop} on a donc
$$
\htop(R) = \htop(R|_{J_R}) = \htop(R|_{J_{R|_D}}) = \log k ~.
$$

Comme le degr{\'e} local de~$R$ est constant {\'e}gal {\`a}~$d_j$ sur l'int{\'e}rieur
de~$I_j$, la partie~2 du Lemme~\ref{L-calcul-jac} implique que le Jacobien
$\Jac_{\rho_R}$ de la mesure~$\rho_R$ est {\'e}gal {\`a} $\deg(R)/d_j$ sur un
sous-ensemble de mesure pleine de~$I_j$.  On a alors $\rho_R(I_j) =
d_j / \deg(R)$, et par la formule de Rokhlin,
$$
h_{\rho_R}(R)
=
\int \Jac_{\rho_R} \, d \rho_R
=
\sum_{j = 1}^k \frac{d_j}{\deg(R)} \log \left( \frac{d_j}{\deg(R)} \right) ~.
$$
L'in{\'e}galit{\'e} et la derni{\`e}re assertion du lemme sont alors des
cons{\'e}quences de la convexit{\'e} stricte de la fonction $x \mapsto - x \log
x$.
\end{proof}


\subsection{Deux exemples de polyn{\^o}mes en caract{\'e}ristique mixte}\label{ss:exemple sauvage}

Nous donnons enfin deux exemples de polyn{\^o}mes sur $\C_p$ le compl{\'e}t{\'e}
de la cl{\^o}ture alg{\'e}brique du corps $\Q_p$ muni de la norme $p$-adique.
La particularit{\'e} de ce corps est le fait que sa caract{\'e}ristique est
nulle, tandis que sa caract{\'e}ristique r{\'e}siduelle est {\'e}gale {\`a} $p>1$.
Rappelons que $\HKo$ d\'esigne l'ensemble des points non singuliers de $\HK$, c'est-\`a-dire l'union 
des points rationnels et irrationnels.

Le premier exemple est extrait de~\cite[Exemple~6.3]{R1}.
\begin{Prop}\label{p:ex1}
  Soit $P(z) = p^{-1} ( z^p - z^{p^2})$.  Alors l'ensemble de Julia de
  $P$ est inclus dans le sous-ensemble $\{ \diam = p^{-1/(p-1)}\}
  \subset \HK$. De plus, il existe une application continue $\pi: J_P
  \to \{ 1, \cdots, p\}^\N$ conjuguant $P|_{J_P}$ au d{\'e}calage.  La
  mesure d'{\'e}quilibre est l'unique mesure d'entropie maximale,
  d'entropie {\'e}gale {\`a} $\log p$. Les points de $\HKo\cap J_P$ sont tous 
pr\'ep\'eriodiques, et les points p\'eriodiques sont tous r\'epulsifs. En particulier, 
presque tout point pour la mesure
  d'{\'e}quilibre appartient {\`a} $\HK\setminus \HKo$.
\end{Prop}
Une l{\'e}g{\`e}re modification du polynome pr{\'e}c{\'e}dent permet d'obtenir
le r{\'e}sultat suivant.
\begin{Prop}\label{p:ex2}
  Soit $Q(z) = p^{-1} ( z^p - z^{p^2})+ pz^{p^2+1}$. Alors la
  restriction $Q|_{J_Q}$ n'est pas localement injective (pour la
  topologie induite par la m{\'e}trique $\dpk$) sur un ensemble non
  d{\'e}nombrable.
\end{Prop}

\begin{proof}[D\'emonstration de la Proposition~\ref{p:ex1}]
  On montre que $|P(z+w) - P(w)| = p |z|^p$ pour tout $p^{-1/(p-1)} <
  |z| < 1$ et tout $|w| \le 1$. En particulier la pr{\'e}image de toute
  boule ferm{\'e}e $B \subset \{ |z| \le 1 \}$ et de diam{\`e}tre au moins
  $p^{-1/(p-1)}$ est une r{\'e}union de $p$ boules ferm{\'e}es disjointes
  $B_i$ avec $ \diam (B_i) = ( p^{-1} \diam (B))^{1/p} >
  p^{-1/(p-1)}$, et $P: B_i \to B$ est de degr{\'e} $p$.
  
  En particulier, posons $P^{-1} \{ |z| \le 1 \} = B'_1 \cup \cdots
  \cup B'_p$, et notons $B'_{i_1, \cdots , i_k}$ la composante de $P^{-n} \{
  |z| \le 1 \}$ telle que $P^j( B'_{i_1, \cdots , i_k}) \subset
  B'_{i_j}$ pour tout $j$. Cette composante est une boule ferm{\'e}e de
  rayon $p^{-(1-p^{-k})/(p-1)}> p^{-1/(p-1)}$, et pour toute suite
  infinie $\a = \{ i_k\} \in \{ 1, \cdots , p\}^\N$, la suite
  $B'_{i_1, \cdots , i_k}$ est d{\'e}croissante avec $k$. Les normes
  $\cS_{i_1, \cdots, i_k}\in \HK$ associ{\'e}es sont donc d{\'e}croissantes, et
  convergent vers un point $\cS_\a$ de $\HK \cap J_P$ tel que $\diam
  (\cS_\a) = p^{-1/(p-1)}$.

  Tout point hors de $\{ |z| >1\}$ est attir{\'e} par l'infini, donc on a l'inclusion
  $J_P
  \subset K_P\= \bigcap_{n\ge0} P^{-n} \{ |z| \le 1 \}$. Il v\'erifie de plus que
  $J_P$ est {\'e}gal au bord de $K_P$.  L'application $ \pi : J_P \to \{
  1, \cdots , p\}^\N$ d\'efinie en envoyant $\cS$ sur la suite $(i_k)$ telle que
  $P^n(\cS) \in B'_{i_n}$ pour tout $n\ge 0$ est alors continue et son
  inverse est donn{\'e}e par $\a \to \cS_\a$.  On conclut donc que $P$ est
  conjugu{\'e} au d{\'e}calage par $\pi$.

  Calculons maintenant l'image de la mesure d'{\'e}quilibre par $\pi_*$.
  On v{\'e}rifie par r{\'e}\-cur\-ren\-ce que $p^{-2n} P^{n*} [ \cScan] =
  \sum_{(i_1, \cdots, i_n)} p^{-n} [\cS_{i_1, \cdots, i_n}]$ pour tout
  entier $n$. Comme $p^{-2n} P^{n*} [ \cScan] \to \rho_P$, on voit donc que
  $\pi_* \rho_P$ s'identifie {\`a} la mesure {\'e}quilibr{\'e}e sur l'espace des
  suites de symboles $\{ 1, \cdots, p \}^\N$.

  Enfin on montre que $\cS_\a \in \HKo$ si et seulement si $\a$ est pr\'ep{\'e}riodique pour
  le d{\'e}calage.
 Cette assertion est d{\'e}montr{\'e}e pr{\'e}cis{\'e}ment
  dans~\cite[Exemple~6.3]{R1}. Le fait que les points p\'eriodiques sont r\'epulsifs r\'esulte 
de $A \subset \{ \deg_P = p \}$.
\end{proof}

\begin{proof}[D\'emonstration de la Proposition~\ref{p:ex2}]
  Notons que  $ Q = P + p
  z^{p^2}$. On v{\'e}rifie  tout d'abord que $ P =Q$ sur l'ensemble $
A\=  \{ |\cS |\le 1 \} \cap \{ \diam = p^{-1/(p-1)}\}$. On a vu sur
  l'exemple pr{\'e}c{\'e}dent que $J_P \subset A$, et que les points r\'epulsifs \'etaient denses dans
$J_P$. On en d\'eduit que  donc $J_P \subset J_Q$. Dans la suite, on notera $\hprat = \HKrat$ pour $K = \C_p$.
\begin{Lem}\label{l:interm}
  Pour tout $\cS_0 \in J_P$, et tout $\e>0$, il existe un $\cS \in
  J_Q\cap \hprat$ tel que $\dhyp(\cS, \cS_0)\le \e$ et $ \diam(\cS) <
  p^{-1/(p-1)}$.
\end{Lem}
Prenons $\cS_0$ dans $J_Q$ et $\cS'_0$ une pr{\'e}image de $\cS_0$ par $Q$ contenue dans $\{ |\cS| =1
\}$. Nous allons montrer que pour tout $\e>0$, la restriction de $Q$ {\`a}
la boule de centre $\cS'_0$ et de diam{\`e}tre $\e$ (pour la m{\'e}trique
$\dhyp$) n'est pas injective.  Pour cela, on utilise le lemme
pr{\'e}c{\'e}dent et on prend $\cS_1\in J_P \cap \hprat$ tel que $\dhyp
(\cS_1, \cS_0) \le \e$, et $\diam(\cS_1)< p^{-1/(p-1)}$. Un tel point
admet une pr{\'e}image $\cS'_1$ dont la distance {\`a} $\cS'_0$ est $\le p
\times \e$. Il est facile de voir que $Q$ envoie toute boule centr{\'e}e
en un point $|z|=1$ et de diam{\`e}tre $ p^{-1/(p-1)}$ sur une boule de
diam{\`e}tre $p^{-1/(p-1)}$.  On en d{\'e}duit l'existence d'un point $\cS'_2
\in ]\cS'_0, \cS'_1]$ de diam{\`e}tre {\'e}gal {\`a} $p^{-1/(p-1)}$.
 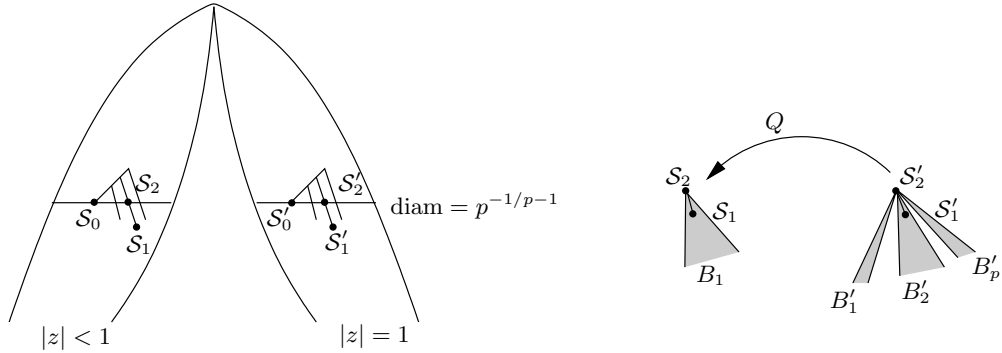
\begin{figure}[ht]\label{f:pos}
 \centering \input{non-inj2.pstex_t}
\caption{Positions des points $S_i$ et $S_i'$}
 \end{figure}
Notons que $\deg_P(\cS) =1$ d\`es que $\diam (\cS) < p^{-1/(p-1)}$. Comme $Q=P$
dans un voisinage (pour la m\'etrique $\dhyp$) de $A$, la
boule ouverte $B_1$ associ{\'e}e {\`a} $\cS_2$ et
contenant $\cS_1$ admet  $p$ pr{\'e}images $B'_i$ dont les adh{\'e}rences
contiennent $\cS'_2$ (voir Figure~\ref{f:pos}). La restriction de $P$ {\`a} chacune des $B'_i$ est
injective et on en d{\'e}duit donc que $\cS_1$ admet une pr{\'e}image par $Q$
dans chaque $B'_i$. Chacune de ces pr{\'e}images est {\`a} $\dhyp$-distance 
au plus $\e$ de $\cS'_2$ donc au plus $(p+1)\e$ de $\cS'_*$.
Nous avons donc montr{\'e} que $Q$ n'est injective sur aucune boule
pour $\dhyp$ centr{\'e}e en $\cS'_0$.

 On en d{\'e}duit que $Q$ ne peut {\^e}tre localement
injective en toute pr{\'e}image de $J_P$ par $Q$ contenue dans $\{|\cS|=1
\}$. L'ensemble $J_P$ {\'e}tant de type de Cantor, on en d{\'e}duit la
non locale injectivit{\'e} de $Q$ sur un ensemble non d{\'e}nombrable.
\end{proof}
  \begin{proof}[D\'emonstration du Lemme~\ref{l:interm}]
    On traite tout d'abord le cas du
    point de $\hprat$ associ{\'e} {\`a} la boule centr{\'e}e en $0$ et de diam{\`e}tre
    $p^{-1/(p-1)}$.

    On remarque tout d'abord l'existence d'un point fixe $z_0$ de
    module $p^2$. Pour voir cela, on applique~\cite[Theorem~1,
    p.307]{Robert} en notant que pour $|z| = p^2$, on a $|p^{-1}
    z^{p^2}| = |p z ^{p^2+1}| > \max \{ |p^{-1} z^p|, |z| \}$. Un
    calcul montre que $|Q'(z)|= |p z^{p^2}|>1$ pour $|z| = p^2$, donc
    $z_0$ est r{\'e}pulsif.  En particulier il appartient {\`a} $J_Q$.  On
    remarque maintenant que $Q^2 \{ |z| < 1 \} = Q \{ |z| < p \} = \{
    |z| < p^{p^2+1}\}$. On peut donc trouver une pr{\'e}image de $z_0$ par
    $Q^2$ dans $ \{ |z| <1 \}$: on la note $z_1$.
  
    Soit $B'$ une boule ferm{\'e}e contenant $z_0$ de rayon $\e$
    arbitrairement petit. Pour $N$ assez grand, $Q^N (B')$ contient
    $J_Q$ car $z_0 \in J_Q$, donc il existe une boule ferm{\'e}e $B$ de
    rayon $\le \e$ dont l'image par $Q^N$ est la boule de rayon
    $p^{-1/(p-1)}$. Le point de $\hprat$ associ{\'e} {\`a} $B$ est donc dans
    $J_Q$, et la pr{\'e}image de $B$ par $Q^2$ est une boule ferm{\'e}e $B_0$
    de rayon comparable {\`a} $\e$ dont le point de $\hprat$ associ{\'e} est
    aussi dans $J_Q$. Quitte {\`a} prendre $\e$ assez petit, on peut donc
    s'arranger pour que $B_0$ soit de diam{\`e}tre $< p^{-1/(p-1)}$ et
    incluse dans $\{ |z| <1 \}$.
 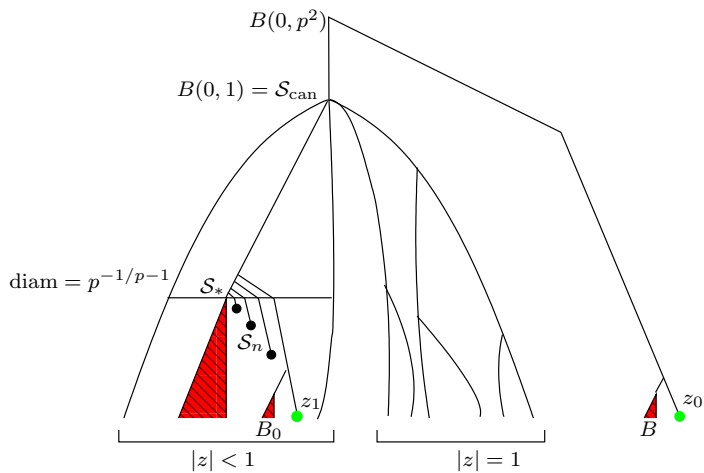
\begin{figure}[h]
 \centering \input{non-inj3.pstex_t}
\caption{Preuve du Lemme~\ref{l:interm}}
 \end{figure}
    Soit $B_n= B(z_n,r_n)$ une suite de boules de telle sorte que
    $Q(B_{n+1}) = B_n$ et $B_n \subset \{ |z| < 1\}$. Par r{\'e}currence,
    on montre que $\diam (B_{n+1}) = (\diam(B_n)/p)^{1/p} <
    p^{-1/(p-1)}$ donc $\diam(B_n) \to p^{-1/(p-1)}$. De plus pour tout
    point $|z|<1$, on a $|Q(z)| = p |z|^p$, donc les centres de $B_n$
    v{\'e}rifient aussi $|z_n| \to p^{-1/(p-1)}$. Si $\cS_n$ est le point
    de $\hprat$ associ{\'e} {\`a} $B_n$, ces estimations impliquent
    $\dhyp(\cS_n , \cS_*) \to 0$, ce qui nous permet de conclure dans
    le cas $\cS_0 = \cS_*$.

\smallskip

Soit maintenant $\cS_0$ un point arbitraire de $J_P$, et $\e>0$.
L'ensemble $U_\e = J_P \cap \{ \dhp (\cS_0, \cdot) < \e \}$ d{\'e}finit un
voisinage ouvert de $\cS_0$ dans $J_P$. On a vu que $P$ restreint {\`a}
cet ensemble {\'e}tait conjugu{\'e} au d{\'e}calage, donc il existe un entier
$N\gg 1$ tel que $P^N (U_\e) \supset J_P$.  Comme $P=Q$ sur $J_P$ et
que $Q$ est ouverte pour $\dhp$, on conclut que $Q^N \{ \dhp (\cdot,
\cS_0) < \e \}$ contient un voisinage de $\cS_*$ (toujours pour la
m{\'e}trique $\dhp$). On peut donc trouver un point $\cS' \in J_Q$ dans ce
voisinage de diam{\`e}tre $<p^{-1/(p-1)}$.  Sa pr{\'e}image $\cS$ dans $\{ \dhp
(\cdot, \cS_0) \le \e \}$ est de diam{\`e}tre $<p^{-1/(p-1)}$ et appartient \`a
l'ensemble de Julia de $Q$ car cet ensemble est totalement invariant.
On a donc trouv{\'e} un point $\cS\in J_Q$ tel que $\dhp (\cS, \cS_0) \le
\e$ et $\diam(\cS) < p^{-1/(p-1)}$.
\end{proof}


\end{document}

%% file: arbre.pstex_t
\begin{picture}(0,0)%
\includegraphics{arbre.pstex}%
\end{picture}%
\setlength{\unitlength}{3947sp}%
\begingroup\makeatletter\ifx\SetFigFont\undefined%
\gdef\SetFigFont#1#2#3#4#5{%
  \reset@font\fontsize{#1}{#2pt}%
  \fontfamily{#3}\fontseries{#4}\fontshape{#5}%
  \selectfont}%
\fi\endgroup%
\begin{picture}(5173,2555)(289,-2832)
\put(1877,-564){\makebox(0,0)[lb]{\smash{{\SetFigFont{7}{8.4}{\rmdefault}{\mddefault}{\updefault}Un ouvert pour la topologie}}}}
\put(1877,-683){\makebox(0,0)[lb]{\smash{{\SetFigFont{7}{8.4}{\rmdefault}{\mddefault}{\updefault}compacte: le bord est fini}}}}
\put(1877,-801){\makebox(0,0)[lb]{\smash{{\SetFigFont{7}{8.4}{\rmdefault}{\mddefault}{\updefault}et constitu\'e de deux pts.}}}}
\put(4287,-465){\makebox(0,0)[lb]{\smash{{\SetFigFont{7}{8.4}{\rmdefault}{\mddefault}{\updefault}Une boule ouverte pour la m\'etrique}}}}
\put(4406,-743){\makebox(0,0)[lb]{\smash{{\SetFigFont{7}{8.4}{\rmdefault}{\mddefault}{\updefault}infini.}}}}
\put(4358,-604){\makebox(0,0)[lb]{\smash{{\SetFigFont{7}{8.4}{\rmdefault}{\mddefault}{\updefault}sph\'erique. Le bord est de cardinal}}}}
\end{picture}%

%% file: affine2.pstex_t
\begin{picture}(0,0)%
\includegraphics{affine2.pstex}%
\end{picture}%
\setlength{\unitlength}{3947sp}%
\begingroup\makeatletter\ifx\SetFigFont\undefined%
\gdef\SetFigFont#1#2#3#4#5{%
  \reset@font\fontsize{#1}{#2pt}%
  \fontfamily{#3}\fontseries{#4}\fontshape{#5}%
  \selectfont}%
\fi\endgroup%
\begin{picture}(1938,1984)(1051,-2333)
\put(2851,-2311){\makebox(0,0)[lb]{\smash{{\SetFigFont{6}{7.2}{\rmdefault}{\mddefault}{\updefault}$1$}}}}
\put(1126,-2311){\makebox(0,0)[lb]{\smash{{\SetFigFont{6}{7.2}{\rmdefault}{\mddefault}{\updefault}$0$}}}}
\put(1426,-1261){\makebox(0,0)[lb]{\smash{{\SetFigFont{6}{7.2}{\rmdefault}{\mddefault}{\updefault}$d_1=2$}}}}
\put(2401,-2311){\makebox(0,0)[lb]{\smash{{\SetFigFont{6}{7.2}{\rmdefault}{\mddefault}{\updefault}$3/4$}}}}
\put(1051,-436){\makebox(0,0)[lb]{\smash{{\SetFigFont{6}{7.2}{\rmdefault}{\mddefault}{\updefault}$1$}}}}
\put(2176,-961){\makebox(0,0)[lb]{\smash{{\SetFigFont{6}{7.2}{\rmdefault}{\mddefault}{\updefault}$d_2=4$}}}}
\put(2476,-1186){\makebox(0,0)[lb]{\smash{{\SetFigFont{6}{7.2}{\rmdefault}{\mddefault}{\updefault}$d_3=4$}}}}
\put(2626,-2311){\makebox(0,0)[lb]{\smash{{\SetFigFont{6}{7.2}{\rmdefault}{\mddefault}{\updefault}$I_3$}}}}
\put(1426,-2311){\makebox(0,0)[lb]{\smash{{\SetFigFont{6}{7.2}{\rmdefault}{\mddefault}{\updefault}$I_1$}}}}
\put(1951,-2311){\makebox(0,0)[lb]{\smash{{\SetFigFont{6}{7.2}{\rmdefault}{\mddefault}{\updefault}$1/2$}}}}
\put(2176,-2311){\makebox(0,0)[lb]{\smash{{\SetFigFont{6}{7.2}{\rmdefault}{\mddefault}{\updefault}$I_2$}}}}
\end{picture}%

%% file: non-inj2.pstex_t
\begin{picture}(0,0)%
\includegraphics{non-inj2.pstex}%
\end{picture}%
\setlength{\unitlength}{3947sp}%
\begingroup\makeatletter\ifx\SetFigFont\undefined%
\gdef\SetFigFont#1#2#3#4#5{%
  \reset@font\fontsize{#1}{#2pt}%
  \fontfamily{#3}\fontseries{#4}\fontshape{#5}%
  \selectfont}%
\fi\endgroup%
\begin{picture}(6409,2181)(1146,-5630)
\put(1895,-5020){\makebox(0,0)[lb]{\smash{{\SetFigFont{9}{10.8}{\rmdefault}{\mddefault}{\updefault}$\cS_1$}}}}
\put(3142,-5007){\makebox(0,0)[lb]{\smash{{\SetFigFont{9}{10.8}{\rmdefault}{\mddefault}{\updefault}$\cS'_1$}}}}
\put(1355,-5599){\makebox(0,0)[lb]{\smash{{\SetFigFont{9}{10.8}{\rmdefault}{\mddefault}{\updefault}$|z|<1$}}}}
\put(3230,-5588){\makebox(0,0)[lb]{\smash{{\SetFigFont{9}{10.8}{\rmdefault}{\mddefault}{\updefault}$|z|=1$}}}}
\put(3552,-4784){\makebox(0,0)[lb]{\smash{{\SetFigFont{9}{10.8}{\rmdefault}{\mddefault}{\updefault}$\diam = p^{-1/p-1}$}}}}
\put(1951,-4636){\makebox(0,0)[lb]{\smash{{\SetFigFont{9}{10.8}{\rmdefault}{\mddefault}{\updefault}$\cS_2$}}}}
\put(1576,-4861){\makebox(0,0)[lb]{\smash{{\SetFigFont{9}{10.8}{\rmdefault}{\mddefault}{\updefault}$\cS_0$}}}}
\put(2776,-4861){\makebox(0,0)[lb]{\smash{{\SetFigFont{9}{10.8}{\rmdefault}{\mddefault}{\updefault}$\cS'_0$}}}}
\put(3226,-4636){\makebox(0,0)[lb]{\smash{{\SetFigFont{9}{10.8}{\rmdefault}{\mddefault}{\updefault}$\cS'_2$}}}}
\put(5577,-4817){\makebox(0,0)[lb]{\smash{{\SetFigFont{9}{10.8}{\rmdefault}{\mddefault}{\updefault}$\cS_1$}}}}
\put(6972,-4802){\makebox(0,0)[lb]{\smash{{\SetFigFont{9}{10.8}{\rmdefault}{\mddefault}{\updefault}$\cS'_1$}}}}
\put(5907,-4247){\makebox(0,0)[lb]{\smash{{\SetFigFont{9}{10.8}{\rmdefault}{\mddefault}{\updefault}$Q$}}}}
\put(6752,-4597){\makebox(0,0)[lb]{\smash{{\SetFigFont{9}{10.8}{\rmdefault}{\mddefault}{\updefault}$\cS'_2$}}}}
\put(5472,-5212){\makebox(0,0)[lb]{\smash{{\SetFigFont{9}{10.8}{\rmdefault}{\mddefault}{\updefault}$B_1$}}}}
\put(7202,-5162){\makebox(0,0)[lb]{\smash{{\SetFigFont{9}{10.8}{\rmdefault}{\mddefault}{\updefault}$B'_p$}}}}
\put(6762,-5317){\makebox(0,0)[lb]{\smash{{\SetFigFont{9}{10.8}{\rmdefault}{\mddefault}{\updefault}$B'_2$}}}}
\put(6312,-5377){\makebox(0,0)[lb]{\smash{{\SetFigFont{9}{10.8}{\rmdefault}{\mddefault}{\updefault}$B'_1$}}}}
\put(5282,-4592){\makebox(0,0)[lb]{\smash{{\SetFigFont{9}{10.8}{\rmdefault}{\mddefault}{\updefault}$\cS_2$}}}}
\end{picture}%

%% file: non-inj3.pstex_t
\begin{picture}(0,0)%
\includegraphics{non-inj3.pstex}%
\end{picture}%
\setlength{\unitlength}{3947sp}%
\begingroup\makeatletter\ifx\SetFigFont\undefined%
\gdef\SetFigFont#1#2#3#4#5{%
  \reset@font\fontsize{#1}{#2pt}%
  \fontfamily{#3}\fontseries{#4}\fontshape{#5}%
  \selectfont}%
\fi\endgroup%
\begin{picture}(4244,2886)(439,-5792)
\put(1982,-5574){\makebox(0,0)[lb]{\smash{\SetFigFont{8}{9.6}{\rmdefault}{\mddefault}{\updefault}$B_0$}}}
\put(4397,-5574){\makebox(0,0)[lb]{\smash{\SetFigFont{8}{9.6}{\rmdefault}{\mddefault}{\updefault}$B$}}}
\put(2267,-5387){\makebox(0,0)[lb]{\smash{\SetFigFont{8}{9.6}{\rmdefault}{\mddefault}{\updefault}$z_1$}}}
\put(4667,-5370){\makebox(0,0)[lb]{\smash{\SetFigFont{8}{9.6}{\rmdefault}{\mddefault}{\updefault}$z_0$}}}
\put(1871,-5023){\makebox(0,0)[lb]{\smash{\SetFigFont{8}{9.6}{\rmdefault}{\mddefault}{\updefault}$\cS_n$}}}
\put(1576,-5761){\makebox(0,0)[lb]{\smash{\SetFigFont{8}{9.6}{\rmdefault}{\mddefault}{\updefault}$|z|<1$}}}
\put(1637,-4671){\makebox(0,0)[lb]{\smash{\SetFigFont{8}{9.6}{\rmdefault}{\mddefault}{\updefault}$\cS_*$}}}
\put(3245,-5761){\makebox(0,0)[lb]{\smash{\SetFigFont{8}{9.6}{\rmdefault}{\mddefault}{\updefault}$|z|=1$}}}
\put(439,-4641){\makebox(0,0)[lb]{\smash{\SetFigFont{8}{9.6}{\rmdefault}{\mddefault}{\updefault}$\diam = p^{-1/p-1}$}}}
\put(1943,-3002){\makebox(0,0)[lb]{\smash{\SetFigFont{8}{9.6}{\rmdefault}{\mddefault}{\updefault}$B(0,p^2)$}}}
\put(1492,-3437){\makebox(0,0)[lb]{\smash{\SetFigFont{8}{9.6}{\rmdefault}{\mddefault}{\updefault}$B(0,1)=\cScan$}}}
\end{picture}

%% file: theorie_ergodique.bbl
\begin{thebibliography}{BGR}



\bibitem[BD]{BerDup05}
F. Berteloot, C. Dupont.
\newblock \emph{Une caract{\'e}risation des endomorphismes de Latt{\`e}s par leur mesure de Green.}
\newblock Comment. Math. Helv. \textbf{80} (2005), 433--454. 

\bibitem[BH]{BH}  
M. Baker, L.C. Hsia.  
\newblock \emph{Canonical Heights, Transfinite Diameters, and Polynomial Dynamics.}  
\newblock J. Reine Angew. Math. \textbf{585} (2005), 61--92.
  

\bibitem[BR1]{BR}
M.~Baker, R.~Rumely.
\newblock \emph{Equidistribution of small points, rational dynamics 
and potential theory.}
\newblock  Ann. Inst. Fourier (Grenoble)  \textbf{56} (2006), 625--688.

\bibitem[BR2]{BR2}
M.~Baker, R.~Rumely.
\newblock \emph{Potential theory on the Berkovich projective line.}
\newblock Livre en pr\'eparation.


%
%
%
%
%
\bibitem[Ben]{Be3}  
R. Benedetto.  
\newblock \emph{ Heights and preperiodic points of polynomials over function fields.}
\newblock Int. Math. Res. Not. 2005, no. 62, 3855--3866.


\bibitem[Ber1]{Ber}  
V.G. Berkovich.  
\newblock \emph{Spectral theory and analytic geometry
 over non-Archimedean fields.}  
\newblock Math. Surveys Monographs 33, Amer. Math. Soc. Providence RI, 1990.  

\bibitem[Ber2]{berko2}  
V.G. Berkovich.  
\newblock \emph{Etale cohomology for non-Archimedean analytic spaces.}
\newblock Inst. Hautes {\'E}tudes Sci. Publ. Math.  \textbf{78}  (1993), 5--161 (1994).

\bibitem[Bou]{Bou}  
N.~Bourbaki.
\newblock \emph{Topologie g{\'e}n{\'e}rale.}  
\newblock Hermann, Paris.


\bibitem[BGR]{remmert}  
S.~Bosch, U.~G{\"u}ntzer, R.~Remmert.
\newblock \emph{Non-Archimedean analysis.}
\newblock Grundlehren der Mathematischen Wissenschaften, 261.
 Springer-Verlag, Berlin, 1984. xii+436 pp.

\bibitem[Br]{Brolin}
H.~Brolin.
\newblock \emph{Invariant sets under iteration of rational functions.}
\newblock Ark. Mat. \textbf{6} (1965), 103--144. 


 \bibitem[CL]{ACL}
A.~Chambert-Loir.
\newblock \emph{Mesures et {\'e}quidistribution sur les espaces de Berkovich.}
\newblock  J. Reine Angew. Math. \textbf{595} (2006), 215--235. 


\bibitem[CLT]{ChaThu}
A.~Chambert-Loir, A. Thuillier.
\newblock \emph{Formule de Mahler et {\'e}quidistribution logarithmique.}
A para\^{\i}tre aux Annales de Fourier.
Pr{\'e}publication 2006, \texttt{arxiv.org/abs/math.NT/0612556}.



\bibitem[CLB]{CanLeB05}
S.~Cantat, S.~Le Borgne.
\newblock \emph{Th{\'e}or{\`e}me limite central pour les 
endomorphismes holomorphes et les correspondances modulaires.}
\newblock  Int. Math. Res. Not.  \textbf{56} (2005), 3479--3510.
Version corrig\'ee disponible \`a \texttt{http://perso.univ-rennes1.fr/serge.cantat/publications.html} \hfill


\bibitem[E]{Esc03}
A. Escassut.
\newblock \emph{Ultrametric Banach algebras.}
\newblock World Scientific Publishing Co., Inc., River Edge, NJ, 2003. xiv+275 pp.

\bibitem[DPU]{DPU}
M. Denker, F. Przytycki, M. Urbanski.
\newblock \emph{On the transfer operator for rational 
functions on the Riemann sphere}.
\newblock Ergodic Theory and Dynam. Systems \textbf{16} (1996), 255--266.

\bibitem[DS1]{DinSib03} T.-C. Dinh, N. Sibony.
\newblock{Dynamique des applications d'allure polynomiale.}
\newblock J. Math. Pures Appl. \textbf{82} (2003), 367--423.

\bibitem[DS2]{DinSib05}
T.-C.~Dinh, N.~Sibony.
\newblock{Green currents for holomorphic automorphisms of compact 
K{\"a}hler manifolds.}
\newblock J. Amer. Math. Soc. \textbf{18} (2005), 291--312. 

\bibitem[DS3]{DinSib06}
T.-C.~Dinh, N.~Sibony.
\newblock \emph{Decay of correlations and the 
central limit theorem for meromorphic maps.}
\newblock  Comm. Pure Appl. Math. \textbf{59} (2006), 754--768.


\bibitem[FG]{FG}  
C. Favre, V. Guedj 
\newblock \emph{Dynamique des applications rationnelles des espaces 
  multiprojectifs.} 
\newblock Indiana Univ. Math. J. \textbf{50} (2001), 881--934. 
 

\bibitem[FJ]{valtree}
C.~Favre, M.~Jonsson.
\newblock \emph{The valuative tree}.
\newblock Lecture Notes in Mathematics, 1853. 
Springer-Verlag, Berlin, 2004. xiv+234 pp. 

\bibitem[FR1]{FR}
C.~Favre, J.~Rivera-Letelier.
\newblock \emph{Th{\'e}or{\`e}me d'{\'e}quidistribution de Brolin 
en dynamique $p$-adique.}
\newblock  C. R. Math. Acad. Sci. Paris \textbf{339} (2004), 271--276. 


\bibitem[FR2]{FRL}
C.~Favre, J.~Rivera-Letelier.
\newblock \emph{Equidistribution quantitative des
  points de petite hauteur sur la droite projective.}
\newblock  Math. Ann. \textbf{335} (2006), 311--361.
\newblock Corrigendum.  Math. Ann.  \textbf{339} (2007), 799--801.

\bibitem[FR3]{FR3}
C.~Favre, J.~Rivera-Letelier.
\newblock \emph{Propri{\'e}t{\'e}s ergodiques des fractions rationnelles
 mod{\'e}r{\'e}es.}
\newblock  En pr{\'e}paration.

\bibitem[FS1]{ForSib94}
J.-E.~Fornaess, N.~Sibony.
\newblock \emph{Complex dynamics in higher dimensions.}
Notes partially written by Estela A. Gavosto. NATO Adv. Sci. Inst. Ser. C Math. Phys. Sci., 439,  Complex potential theory (Montreal, PQ, 1993),  131--186, Kluwer Acad. Publ., Dordrecht, 1994.

\bibitem[FS2]{ForSib95}
J.-E.~Fornaess, N.~Sibony.
\newblock \emph{Complex dynamics in higher dimension. II.}
\newblock Modern methods in complex analysis (Princeton, NJ, 1992), 135--182,
Ann. of Math. Stud., 137, Princeton Univ. Press, Princeton, NJ, 1995. 


\bibitem[FLM]{FLM}
A.~Freire, A.~Lopes, R.~Ma{\~n}{\'e}.
\newblock \emph{An invariant measure for rational maps.}
\newblock Bol. Soc. Brasil. Mat. \textbf{14} (1983), 45--62.

\bibitem[Go]{gordin}
M.I.~Gordin.
\newblock \emph{The central limit theorem for stationary processes.}
\newblock  Dokl. Akad. Nauk SSSR  188  1969 739--741.

\bibitem[Gr]{Gro03}
M.~Gromov.
\newblock \emph{On the entropy of holomorphic maps.}
\newblock Enseign. Math. (2)  49  (2003),  no. 3-4, 217--235.

\bibitem[Gu]{guedj}
V.~Guedj.
\newblock \emph{Propri{\'e}t{\'e}s ergodiques des applications rationnelles.}
\newblock  Pr{\'e}publication 2006 (130 pages), \texttt{arxiv.org/abs/math.CV/0611302}.


\bibitem[H]{H}
N. Haydn.
\newblock \emph{Convergence of the transfer operator for rational maps.}
\newblock Ergodic Theory Dynam. Systems \textbf{19} (1999), 657--669.

\bibitem[HY]{HY}  
M. Herman, J.C. Yoccoz.  
\newblock \emph{Generalizations of some theorems of small 
divisors to non-Archimedean  
fields.}  
\newblock In Geometric Dynamics (Rio de Janeiro 1981) LNM 1007, 
Springer-Verlag,  
Berlin-New York, 1983, 408-447.  
  
\bibitem[H\"o]{hor}  
L.V. H\"ormander.  
\newblock \emph{An introduction to complex analysis in several variables.}
\newblock Third edition. North-Holland Mathematical Library, 7. North-Holland Publishing Co., Amsterdam, 1990.


  
\bibitem[KS]{KawSil}
S.~Kawaguchi and J.~Silverman.
\newblock \emph{Nonarchimedean Green functions and dynamics on projective space.}
\newblock  Pr{\'e}publication 2007, \texttt{arxiv.org/abs/arXiv:0706.2169}


\bibitem[K]{kiwi}
J. Kiwi.
\newblock \emph{Puiseux series polynomial dynamics and 
iteration of complex cubic polynomials}
\newblock   Ann. Inst. Fourier (Grenoble) \textbf{56} (2006), 1337--1404. 



\bibitem[Li]{liverani}
C.~Liverani.
\newblock \emph{Central limit theorem for deterministic systems.}
\newblock International Conference on Dynamical
 Systems (Montevideo, 1995),  56--75, Pitman Res. Notes Math. Ser.,
 362, Longman, Harlow, 1996.

\bibitem[Lj]{Lyu}
M.~J.~Ljubich.
\newblock \emph{Entropy properties of rational endomorphisms 
of the Riemann sphere.}
\newblock  Ergodic Theory Dynam. Systems \textbf{3} (1983), 351-385.


\bibitem[Man]{Man83}
R. Ma{\~n}{\'e}.
\newblock \emph{On the uniqueness of the maximizing measure for rational maps.}
\newblock Bol. Soc. Brasil. Mat. \textbf{14}  (1983), 27--43.


\bibitem[Mat]{matsu}
H. Matsumura.
\newblock \emph{Commutative algebra.}
\newblock W. A. Benjamin, Inc., New York 1970 xii+262 pp. 


\bibitem[May]{May02}
V. Mayer.
\newblock \emph{Comparing measures and invariant line fields.}
\newblock Ergodic Theory Dynam. Systems \textbf{22} (2002), 555-570.

\bibitem[Mi1]{milnor2}  
J. Milnor.
\newblock \emph{Dynamics in one complex variable.} 
\newblock Annals of Mathematics Studies, 160. Princeton University Press, Princeton, NJ, 2006. viii+304 pp. 

\bibitem[Mi2]{milnor}  
J. Milnor.
\newblock \emph{On Latt{\`e}s Maps.}  
\newblock  Dynamics on the Riemann sphere,  9--43, Eur. Math. Soc., Z�rich, 2006.

\bibitem[Mis]{mis}
M. Misiurewicz.
\newblock \emph{A short proof of the variational principle for a $Z\sb{+}\sp{N}$ 
action on a compact space.}
\newblock International Conference on Dynamical Systems in Mathematical Physics (Rennes, 1975),  pp. 147--157. Asterisque, No. 40, Soc. Math. France, Paris, 1976. 

\bibitem[MS]{MS}
P. Morton, J. Silverman.
\newblock \emph{Periodic points, multiplicities, and dynamical units.}
\newblock J. Reine Agnew. Math.~\textbf{461} (1995), 81-122.

\bibitem[P]{parry}
 W. Parry .
 \newblock \emph{Entropy and generators in ergodic theory.}
 \newblock W. A. Benjamin, Inc., New York-Amsterdam 1969 xii+124 pp. 

  
\bibitem[PST]{PST}  
J. Pineiro, L. Szpiro, T.J. Tucker. 
\newblock \emph{Mahler measure for dynamical systems on ${\mathcal P}^1$ and  
intersection theory on a singular arithmetic surface}.  
\newblock Geometric methods in algebra and number theory,
219--250, Progr. Math., 235, Birkh{\"a}user Boston, Boston, MA, 2005.

\bibitem[PU]{PU}
F. Przytycki, M. Urbanski.
\newblock \emph{Fractals in the Plane -- the Ergodic Theory Methods.}
\newblock A paraitre au Cambridge Univ. Press.


\bibitem[R1]{R0}  
J. Rivera-Letelier.  
\newblock \emph{Dynamique des fractions rationnelles sur des corps locaux.}  
\newblock Th{\`e}se, Orsay 2000.  

\bibitem[R2]{R1}  
J. Rivera-Letelier.  
\newblock \emph{Dynamique des fonctions rationnelles sur des corps locaux.}  
\newblock Ast{\'e}risque~\textbf{287} (2003), 147-230.  

\bibitem[R3]{R2}  
J. Rivera-Letelier.  
\newblock \emph{Espace hyperbolique p-adique et 
dynamique des fonctions rationnelles.}  
\newblock Compositio Mathematica~\textbf{138} (2003), 199-231.  
  

\bibitem[R4]{R3}  
J. Rivera-Letelier.  
\newblock \emph{ Points p{\'e}riodiques des fonctions rationnelles
 dans l'espace hyperbolique $p$-adique.}
\newblock Comment. Math. Helv. \textbf{80} (2005), 593--629. 

\bibitem[R5]{R1/2}  
J. Rivera-Letelier.  
\newblock \textit{Sur la structure des ensembles de Fatou p-adiques.}
\newblock Pr{\'e}publication (2004), \texttt{www.arxiv.org/math.DS/0412180}.

\bibitem[R6]{R3.5}  
J. Rivera-Letelier.  
\newblock \textit{Notes sur la droite projective de Berkovich.} 
\newblock Pr{\'e}publication 2006, \texttt{www.arxiv.org/math.MG/0605676}.
  

\bibitem[R7]{R4}  
J. Rivera-Letelier.  
\newblock \emph{Th{\'e}orie de Julia et Fatou sur la 
droite projective de Berkovich.}  
\newblock En pr{\'e}paration. 

\bibitem[Rob]{Robert}
A.~Robert.
\newblock \emph{A course in $p$-adic analysis.}
\newblock  Graduate Texts in Mathematics, 198. Springer-Verlag, New York, 2000.

\bibitem[Roq]{roquette}
P.~Roquette.
\newblock \emph{Analytic theory of elliptic functions over local fields.}
\newblock  Hamburger Mathematische Einzelschriften (N.F.), Heft 1 Vandenhoeck \& Ruprecht, G\"ottingen 1970 

\bibitem[Ru]{rudin}  
W.  Rudin.  
\newblock \emph{Real and complex analysis.}
\newblock  McGraw-Hill Book Co., New York, 1987. xiv+416 pp.

%




\bibitem[Sib]{Sib99}
N.~Sibony.
\newblock \emph{Dynamique des applications rationnelles de $\mathbb{P}\sp k$.}
\newblock  Dynamique et g{\'e}om{\'e}trie complexes (Lyon, 1997), 
 ix--x, xi--xii, 97--185, Panor. Synth{\`e}ses, 8, Soc. Math. France, 
Paris, 1999. 


\bibitem[Sil1]{Sil94}
 J. Silverman.
 \newblock \emph{Advanced topics in the arithmetic of elliptic curves.}
 \newblock Graduate Texts in Mathematics, 151. Springer-Verlag, New York, 1994.

\bibitem[Sil2]{Sil07}
 J. Silverman.
 \newblock \emph{The arithmetic of dynamical systems.}
 \newblock Graduate Texts in Mathematics, 241. Springer-Verlag, New York, 2007.

\bibitem[ST]{ST}  
L. Szpiro, T.J. Tucker. 
\newblock \emph{ Equidistribution and generalized Mahler measures.}
\newblock  Pr{\'e}publication, \texttt{arxiv.org/abs/math.NT/0510404}. 


\bibitem[Th]{thuillier}
A.~Thuillier.
\newblock \emph{Th{\'e}orie du potentiel sur les courbes en
  g{\'e}om{\'e}trie analytique non-archim{\'e}dienne.}
\newblock Th{\`e}se de l'universit{\'e} de Rennes, 2005.


\bibitem[Ti]{T}  
J. Tits.  
\newblock \emph{A theorem of Lie-Kolchin for trees.}  
\newblock Contributions to Algebra (Collection of papers dedicated 
to Ellis Kolchin),  
Academic Press, New York, 1977, 377-388.  
  

\bibitem[To]{Tortrat}  
 P.~Tortrat.
\newblock \emph{Aspects potentialistes de l'it{\'e}ration des polyn{\^o}mes.}  
\newblock  S{\'e}minaire de Th{\'e}orie du Potentiel, Paris, No. 8,  195--209, 
Lecture Notes in Math., 1235, Springer, Berlin, 1987. 
  
\bibitem[W]{walters}
P.~Walters.
\newblock \emph{An introduction to ergodic theory.}
\newblock  Graduate Texts in Mathematics, 79. Springer-Verlag, 
New York-Berlin, 1982. ix+250 pp.

\bibitem[Y]{Y}  
J.C. Yoccoz.  
\newblock \emph{Notes sur la g{\'e}om{\'e}trie et la dynamique $p$-adique.} 
\newblock Cours au Coll{\`e}ge de France 2001/2002.
  


\bibitem[ZS]{ZS}
Zariski, O., Samuel, P.
\newblock \emph{Commutative algebra. Vol. 2}.
\newblock  Graduate Texts in Mathematics, No. 29. Springer-Verlag, New
York-Heidelberg-Berlin (1975).


\bibitem[Z]{Zdu90}
A. Zdunik.
\newblock \emph{Parabolic orbifolds and the dimension of the maximal measure for rational maps.}
\newblock Invent. Math. \textbf{99} (1990), 627-649.


\end{thebibliography}
